\documentclass[11pt]{amsart}
\usepackage[margin= 1in]{geometry}
\usepackage{graphicx,dsfont}
\usepackage{amssymb}
\usepackage{mathtools}
\usepackage{tikz-cd}
\usepackage[title]{appendix}
\usepackage{amsmath, mathrsfs, stmaryrd, color,slashed,bbm}
\usepackage{float}
\usepackage{verbatim}
\usepackage{comment}
\usepackage[utf8]{inputenc}
\usepackage[english]{babel}

\usepackage{enumitem}
\usepackage{booktabs}

\usepackage{amssymb}
\usepackage{empheq}
\usepackage{enumitem, hyperref}
\usepackage{pgfplots}
\usepackage[title]{appendix}

\newtheorem{theorem}{Theorem}
\newtheorem{remark}[theorem]{Remark}

\newtheorem{definition}[theorem]{Definition}
\newtheorem{proposition}[theorem]{Proposition}
\newtheorem{corollary}[theorem]{Corollary}
\newtheorem{lemma}[theorem]{Lemma}

\newcommand{\free}{\textup{free}}

\newcommand{\R}{\mathbb{R}}
\newcommand{\E}{\mathbb{E}}
\newcommand{\N}{\mathbb{N}}
\newcommand{\Z}{\mathbb{Z}}
\newcommand{\bP}{\mathbb{P}}

\newcommand{\whitenoise}{\ensuremath{\mathscr{\dot{W}}}}
\newcommand{\Ham}{\mathbf{H}}
\newcommand{\Fext}{\mathcal{F}_{\textrm{ext}}}

\newcommand{\ug}{\underline{g}}

\numberwithin{theorem}{section}
\numberwithin{equation}{section}

\makeatletter
\def\namedlabel#1#2{\begingroup
    #2%
    \def\@currentlabel{#2}%
    \phantomsection\label{#1}\endgroup
}
\makeatother

\begin{document}
\title[]{Brownian regularity for the KPZ line ensemble}
\author{Xuan Wu}
\address{Department of Mathematics, University of Chicago, 
Chicago IL, 60637}
\email{xuanw@uchicago.edu}
\maketitle
\begin{abstract}
This paper seeks a quantitative comparison between the curves in the KPZ line ensemble \cite{CH16} and a standard Brownian bridge under the $t^{1/3}$ vertical and $t^{2/3}$ horizontal scaling. The estimate we obtained is parallel to the one established in \cite{Ham1}, where the Airy line ensemble was studied. Our main tool is the soft Brownian Gibbs property enjoyed by the KPZ line ensemble. In view of the Gibbs property, the KPZ line ensemble differs from the Airy line ensemble mainly due to the intersecting nature of its curves, which results in the main technical difficulty in this paper. We develop a resampling framework, the soft jump ensemble, to tackle this difficulty. Our method is highly inspired by the jump ensemble technique developed in \cite{Ham1}. 
\end{abstract}

\vspace{1cm}

\section{Introduction}
The Kardar-Parisi-Zhang (KPZ) line ensemble, $\{\mathcal{H}^t_n\}$, is a countable collection of continuous random curves. It may be viewed as a multi-layer extension to the well-known KPZ equation, associated with a universality class which bears the same name. The KPZ line ensemble has attracted extensive attention during the past ten years due to its rich algebraic and probabilistic structures, which also served as powerful tools to study the KPZ equation.\\[-0.25cm]

The KPZ line ensemble $\{\mathcal{H}^t_n\}$ is a Gibbsian line ensemble, a special class of Gibbs measures which have received considerable attention in the past two decades owing, in part, to their occurrence in
integrable probability. A Gibbsian line ensemble can be thought of as a countable collection of independent random curves reweighed by a local interaction energy. Local means that the energies only depend on the values of nearby curves. A simple example of a Gibbsian line ensemble is a collection of Brownian motions conditioned never to collide, known as the Dyson Brownian motion with $\beta=2$. Another well-known example is the Airy line ensemble \cite{CH14}, a central object in the KPZ university class.\\[-0.25cm]

There is a large class of stochastic integrable models from random matrix theory, interacting particle systems, last passage percolation, directed polymers that naturally carry the structure of random paths with some Gibbsian resampling invariance. The structure of a Gibbsian line ensemble can be utilized to great benefit when studying their asymptotic scaling limits or path regularity. Starting with \cite{CH14}, there has been a fruitful development of techniques which leverage the Gibbs property of Gibbsian line ensembles to prove their tightness under various scalings given only one-point tightness information about their top curve – see for instance \cite{CH16, CD18, Wu19, BCD21}. Initiated in \cite{Ham1}, the Gibbs property was further employed to study the path regularity for Gibbsian line ensemble. Furthermore, the understanding of the path regularity for the Brownian last passage percolation (BLPP) then helped study the geometry of the geodesics for (BLPP) and the space-time Airy sheet, see for instance \cite{Ham2, Ham3, Ham4, BGH}.  \\[-0.25cm]

The Gibbs property enjoyed by the KPZ line ensemble ensures that curves in the KPZ line ensemble locally behave like Brownian motions. This paper is devoted to provide a quantitative comparison between the curves in the KPZ line ensemble and a standard Brownian bridge through utilizing such a Gibbs property. We aim to address the following question. \\[-0.25cm]

\indent{\bf Question:} Suppose an arbitrary event $\mathsf{A}$ is of small probability $\varepsilon>0$ under the law of a standard Brownian bridge. How comparable is $\bP[\mathsf{A}]$ with respect to $\varepsilon$ under the law of the KPZ line ensemble?\\[-0.1cm]

Highly inspired by the jump ensemble method, introduced in \cite{Ham1}, we develop a resampling framework, the soft jump ensemble, to study the Brownian regularity for the KPZ line ensemble and to answer the above question.

\subsection{The Kardar-Parisi-Zhang equation}
The KPZ equation was introduced in 1986 by Kardar, Parisi and Zhang \cite{KPZ} as a model for random interface growth. In one-spatial dimension (sometimes also called (1+1)-dimension to emphasize that time is one dimension too), it describes the evolution of a function $\mathcal{H}(t, x)$ recording the height of an interface at time $t > 0$ above position $x\in\R$. The KPZ equation is written formally as a stochastic partial differential equation (sPDE), 
\begin{equation}\label{eq:KPZ}
\partial_t\mathcal{H} = \frac{1}{2}\partial_x^2 \mathcal{H} + \frac{1}{2}(\partial_x\mathcal{H})^2 + \whitenoise,
\end{equation}
where $\whitenoise$ is a space-time white noise (for mathematical background or literature review, see \cite{Cor, QS15} for instance). 

The KPZ equation is a canonical member of the associated KPZ universality class and a model belongs to the KPZ universality class if it bears the same long-time, large-scale behavior as the KPZ equation. All models in the KPZ universality class can be transformed to a kinetically growing interface reflecting the competition between growth in a direction normal to the surface, a surface tension smoothing force, and a stochastic term which tends to roughen the interface. These features may be illustrated by the Laplacian $\frac{1}{2}\partial_x^2 \mathcal{H}$, non-linear term $\frac{1}{2}(\partial_x\mathcal{H})^2$ and white noise $\whitenoise$ in the KPZ equation \eqref{eq:KPZ}. Numerical simulations along with some theoretical results have confirmed that in the long time $t$ scaling limit, fluctuations in the height of such evolving interfaces scale like $t^{1/3}$ and display non-trivial spatial correlations in the scale $t^{2/3}$ (known as the $3:2:1$ KPZ scaling).

The KPZ equation is related to the stochastic heat equation (SHE) with multiplicative noise through the Hopf–Cole transformation. Denote $\mathcal{Z}(t, x)$ as the solution to the following SHE, 
\begin{equation}\label{eq:SHE}
\partial_t \mathcal{Z} = \frac{1}{2}\partial_x^2 \mathcal{Z} + \mathcal{Z} \whitenoise.
\end{equation} 
The Hopf-Cole solution to the KPZ equation \eqref{eq:KPZ} is defined through taking $$\mathcal{H}(t,x)=\log \mathcal{Z}(t,x).$$  It was first proved in \cite{Mue} that $\mathcal{Z}(t,x)$ is almost surely strictly positive, which justifies the validity of the transform. The fundamental solution $\mathcal{Z}^{nw}(t, x)$ to SHE \eqref{eq:SHE} is of great importance. It solves \eqref{eq:SHE} with a delta mass initial value problem at origin, i.e. $\mathcal{Z}(0, x)=\delta_{x=0}$. Meanwhile $\mathcal{H}^{nw}(t, x) = \log \mathcal{Z}^{nw}(t, x)$ is known as the narrow wedge solution to the KPZ equation. The initial condition of $\mathcal{H}^{nw}(0, x)$ is not well-defined; however $\mathcal{H}^{nw}(t, x)$ is stationary around a parabola $\frac{-x^2}{2t}$, which resembles a sharp wedge for small $t$, hence known as the narrow wedge initial condition.

Using the Feynman-Kac representation, $\mathcal{Z}^{nw}(t, x)$ formally take the following expression,
\begin{equation}\label{eq:Z^nw}
\mathcal{Z}^{nw}(t, x)=p(t,x)\E\left[:\exp :\left(\int_0^t\whitenoise(s,B_s)ds\right)\right],
\end{equation}
where $p(t,x)=(2\pi t)^{-1/2} \exp(-x^2/2t)$ is the heat kernel, the expectation is taken with respect to a Brownian bridge $(B_s,s\leq t)$ which starts at origin at time $0$ and arrives at $x$ at time $t$. The $:{\rm exp}:$ is the {\it Wick exponential}, see \cite{Jan} for instance. This bridge representation arises because of the $\delta_{x=0}$ initial condition, and hence the factor $p(t,x)$. This Feynman-Kac representation is mostly formal since the integral of white noise over a Brownian path is not well-defined pathwise or to exponentiate the integral.

We adopt this representation to emphasize on its interpretation as being the partition function of a continuum directed random polymer (CDRP) that is modeled by a continuous path interacting with a space-time white noise. We will make sense of the expression \eqref{eq:Z^nw} through a chaos expansion in Section \ref{sec:KPZLE}.

\subsection{The KPZ line ensemble}\label{sec:KPZLE}
We emphasize that the above description is very useful for a generalization to the multi-layer scenario. The KPZ line ensemble is defined through taking the logarithm of ratios of partition functions of the continuum directed random polymers (CDRP). The partition functions of the CDRP $\mathcal{Z}_n(t,x)$ are formally written as 
\begin{equation}\label{wick}
\mathcal{Z}_{n}(t,x) = p(t,x)^n \E\left[:{\rm exp}:\, \left\{\sum_{i=1}^{n} \int_{0}^{t} \whitenoise(s,B_i(s)) ds \right\} \right],
\end{equation}
where the expectation is taken with respect to the law on $n$ independent Brownian bridges $\{B_i\}_{i=1}^{n}$ starting at $0$ at time $0$ and ending at $x$ at time $t$. Intuitively these path integrals represent energies of non-intersecting paths, and thus the expectation of their exponential represents the partition functions for this polymer model. It is worth noting that the first layer, $\mathcal{Z}_1$, is the sames as the fundamental solution to the stochastic heat equation (SHE). These partition functions $\mathcal{Z}_n(t,x)$ also solve a multi-layer extension of SHE, see \cite{OW}. 

For $n\in \N$, $t\geq 0$ and $x\in \R$, $\mathcal{Z}_n(t,x)$ is rigorously defined via the following chaos expansion,
\begin{equation}\label{Zpartfunc}
\mathcal{Z}_{n}(t,x) = p(t,x)^n \sum_{k=0}^{\infty} \int_{\Delta_k(t)}\int_{\R^k} R_k^{(n)}\left((t_1,x_1),\cdots, (t_k,x_k)\right) \whitenoise(dt_1 dx_1)\cdots \whitenoise(dt_k dx_k),
\end{equation}
where $\Delta_k(t) = \{0<t_1<\cdots <t_k<t\}$, and $R_k^{(n)}$ is the $k$-point correlation function for a collection of $n$ non-intersecting Brownian bridges each of which starts at $0$ at time $0$ and ends at $x$ at time $t$. For notational simplicity set $\mathcal{Z}_0(t,x)\equiv 1$. For details about integration with respect to a white noise, we refer to \cite{Jan}.

\cite{LW} show that for any $t>0$, with probability 1, for all $x\in \R$ and all $n\in \N$, $\mathcal{Z}_n(t,x)>0$. The positivity result permits the following important definition of the KPZ line ensemble $\mathcal{H}_n^t(x)$, a process given by the logarithm of ratios of partition functions $\mathcal{Z}_n$. 
\begin{definition}
For $t>0$ fixed, the {\it KPZ$_t$ line ensemble} is a continuous $\N$-indexed line ensemble  $\mathcal{H}^t = \{\mathcal{H}^t_{n}\}_{n\in\N}$  on $\R$ given by
\begin{equation}
\mathcal{H}^t_{n}(x) = \log \left(\frac{\mathcal{Z}_n(t,x)}{\mathcal{Z}_{n-1}(t,x)}\right).
\end{equation}
Sometimes we will omit $t$ and just write the KPZ line ensemble as $\mathcal{H}$.
\end{definition}
This paper continues the investigation of the KPZ line ensemble under the KPZ scaling in \cite{Wu21}. The scaled KPZ line ensemble $\mathfrak{H}^t$ is defined as follows.
\begin{definition}\label{def:scaledKPZLE}
For $t>0$ fixed, the scaled {\it KPZ$_t$ line ensemble} is a continuous $\N$-indexed line ensemble  $\mathfrak{H}^t = \{\mathfrak{H}^t_{n}\}_{n\in\N}$  on $\R$ given by
\begin{equation}\label{eq:scaledKPZLE}
\mathfrak{H}^t_n(x):=\frac{\mathcal{H}^t_n(t^{2/3}x)+\frac{t}{24}}{t^{1/3}}. 
\end{equation}
The result we prove is uniform in $t\geq 1$ and sometimes we will omit $t$ and just write $\mathfrak{H}$.
\end{definition}

\subsection{Main result}
The KPZ line ensemble is known to have a beautiful Gibbs property \cite{CH16}, which provides a precise Radon-Nikodym description of the paths in $\mathfrak{H}$ with respect to Brownian bridges. The Gibbs property immediately implies the absolute continuity of the paths in $\mathfrak{H}$ with respect to Brownian bridge measure \cite{CH16}. \cite{Wu21} studied the KPZ line ensemble and showed that the the curves in the KPZ line ensemble fluctuate on the same scale compared to the local fluctuation of Brownian bridges. In this paper and aim to investigate how general one may push such a comparison between the curves of the KPZ line ensemble and Brownian bridges.\\[-0.3cm]

To be precise, let $s>0$ and ${C}_{0,0}([0,s])$ be the space of continuous functions on $[0,s]$ which vanish both at $0$ and $s$. Let $B$ be a standard Brownian bridge on $[0,s]$ and denote its law by $\mathbb{P}_{\free}$. Denote $\mathfrak{H}_k^{t,[0,s]}$ as the bridge part of $\mathfrak{H}_k$ on $[0,s]$, i.e.
$$\mathfrak{H}_k^{t,[0,s]}(x):=\mathfrak{H}^t_k(x)-s^{-1}x\mathfrak{H}^t_k(s)-s^{-1}(s-x)\mathfrak{H}^t_k(0).$$
We will write $\mathfrak{H}_k^{[0,s]}$ for notational simplicity sometimes.
\begin{theorem}\label{thm:main}
For any $k\in\mathbb{N}$ and $s\geq 1$, there exists $D:=D(k,s)$ such that the following statement holds uniformly in $t\geq 1$. Given a Borel set $A\subset {C}_{0,0}([0,s])$, let $\varepsilon=\mathbb{P}_{\free}(B\in A)$. For $\varepsilon\in (0,e^{-s^3}]$, we have
\begin{align*}
\mathbb{P}(\mathfrak{H}_k^{t,[0,s]}\in A)\leq D\varepsilon \cdot e^{D(\log\varepsilon^{-1})^{5/6}}.
\end{align*}
\end{theorem}

\begin{remark}For notational simplicity, we state and prove the above result for interval $[0,s]$. This can be easily adapted to general intervals as $\mathfrak{H}$ is stationary under horizontal shifts after subtracting the parabola $\frac{x^2}{2}$. The dependence of $D(k,s)$ on $k$ and $s$ can be traced through our arguments, which we do not pursue here.
\end{remark}
 
Theorem~\ref{thm:main} gives an informative comparison between the curves of the scaled KPZ line ensemble and a standard Brownian bridge. This is parallel to the comparison established in \cite{Ham1} between curves in the Airy line ensemble, BLPP line ensemble and a standard Brownian bridge. Such a comparison serves as an important ingredient in investigating geodesic coalescence and geodesic energy profiles for Brownian last passage percolation, see for instance \cite{Ham2,Ham3,Ham4, BGH}. Another quantitative approximation of the Airy line ensemble by connecting independent Brownian bridges was studied in \cite{DV} and served as a major input in the construction of the direct landscape (space-time Airy sheet) \cite{DOV}. We hope Theorem~\ref{thm:main} may similarly shed light on the further study of the KPZ line ensemble and its related directed random polymer models. \\[-0.25cm]

It is worth mentioning that \cite{CHH} further removes the affine shift and  strengthen the comparison to a comparison between the curves of Airy line ensemble or BLPP line ensemble and a \textbf{Brownian motion}.  The energy landscape of last passage percolation models (together with another 5 applications) were also studied in \cite{CHH} based on their comparison results. It is interesting to investigate whether Brownian motion comparison holds true for the scaled KPZ line ensemble and explore further applications.

\subsection{Gibbs property, local fluctuation estimates and the high jump difficulty}\label{sec:jump difficulty} The KPZ line ensemble has been studied in \cite{Wu21} through its Gibbs property. \cite{Wu21} obtains local fluctuation estimates, of which a key application is the convergence of the KPZ line ensemble to the Airy line ensemble, together with the work in \cite{QS20, Vir}. In this section, we briefly introduce the Gibbs property and and explain framework of the argument in \cite{Wu21}. We emphasize on the limitation of the resampling techniques from \cite{Wu21} towards the Brownian comparison pursued in this paper. \\[-0.3cm]

The Gibbs property for the scaled KPZ line ensemble $\mathfrak{H}$ is a spatial Markov property which further specifies the law of $\mathfrak{H}$ in a domain, given the boundary information. We illustrate with the $k=1$ case. More specifically, given the values of $\mathfrak{H}_1(\ell), \mathfrak{H}_1(r)$ and $\mathfrak{H}_2$ on $[\ell, r]$, called the boundary data, we may recover the law of $\mathfrak{H}_1$ on $[\ell, r]$. The law of $\mathfrak{H}_1$ on $[\ell,r]$ satisfies the following Radon-Nikodym derivative relation with respect to a Brownian bridge connecting $(\ell, \mathfrak{H}_1(\ell))$ and $(r,\mathfrak{H}_1(r))$.
\begin{equation}\label{eq:RN}
\frac{d\mathbb{P}_{\mathfrak{H}_1}}{d\mathbb{P}_{\free}} (B)=\frac{\displaystyle W(B)}{Z^{\mathfrak{H}_1(\ell),\mathfrak{H}_1(r), \mathfrak{H}_2}}.
\end{equation}
$W$ is known as the Boltzmann weight and it penalizes the curves for going out of order. Moreover, $W$ is always less than $1$, while the exact expression of $W$ does not matter here for our purpose. See \eqref{def:Boltzmann_Brownian} for the explicit form of $W.$ The normalizing constant $Z^{\mathfrak{H}_1(\ell),\mathfrak{H}_1(r), \mathfrak{H}_2}$ is necessary to ensure a probability measure and it takes the form of
\begin{align*}
Z^{\mathfrak{H}_1(\ell),\mathfrak{H}_1(r), \mathfrak{H}_2} :=\E_{\free}[W(B)].
\end{align*}
See Section \ref{def:H-BGP} for a detailed description.\\

Let us emphasize that the the normalizing constant $Z^{\mathfrak{H}_1(\ell),\mathfrak{H}_1(r), \mathfrak{H}_2}$ inherits randomness from the boundary data $\mathfrak{H}_1(\ell),\mathfrak{H}_1(r), \mathfrak{H}_2$. This causes one of the main difficulties towards understanding the marginal law of $\mathfrak{H}_1$ on $[0,s]$. Intuitively, a quantitative control of $Z^{\mathfrak{H}_1(\ell),\mathfrak{H}_1(r), \mathfrak{H}_2}$ will be very helpful for a quantitative comparison between the law of $\mathfrak{H}_1$ with respect to a Brownian bridge. 

We now illustrate the above idea in detail. Take $[\ell,r]=[-2,2]$. Recall that $A\in{C}_{0,0}([0,s])$ is an arbitrary Borel subset with probability $\varepsilon$ under a standard Brownian bridge. We abbreviate $Z^{\mathfrak{H}_1(\ell),\mathfrak{H}_1(r), \mathfrak{H}_2}$ as $Z_{bdd}$ to emphasize that $\mathfrak{H}_1(\ell),\mathfrak{H}_1(r), \mathfrak{H}_2$ serve as the boundary data. Using \eqref{eq:RN}, we have
\begin{align*}
\mathbb{P}_{\mathfrak{H}}(\mathfrak{H}\in A )=\E_{\mathfrak{H}} [\mathbbm{1}\{\mathfrak{H}\in A \}]=\E_{\textup{bdd}}\left[ \frac{\E_{\free} [\mathbbm{1}\{B\in A \}\cdot W]}{Z_{\textup{bdd}}} \right].
\end{align*} 
Here $\E_{\textup{bdd}}$ means taking expectation over all possible boundary data.  We may simply bound the numerator using $W\leq 1$,
$$\E_{\free}[\mathbbm{1}\{B\in A \}\cdot W]\leq \bP_{\free}(B\in A)=:\varepsilon .$$ 
As long as  $Z_{\textup{bdd}}$ has a lower bound, we may find some upper bound for $\mathbb{P}_{\mathfrak{H}}(\mathfrak{H}\in A)$.\\

The following quantitative tail bound on $Z_{\textup{bdd}}$ was established in \cite{Wu21}.
\begin{align*}
\mathbb{P}_{\mathfrak{H}} \left(Z_{\textup{bdd}}<  e^{-O(1)K^2}  \right)< e^{- K^{3/2}}.
\end{align*} 
See Proposition \ref{pro:Z_lowerbound_k} for a precise statement. Therefore, 
\begin{align*}
\mathbb{P}_{\mathfrak{H}}(\mathfrak{H}\in A  )\leq \varepsilon  e^{O(1)K^2}+e^{- K^{3/2}}
\end{align*}
Choosing $K=C^{-1}\left(\log\varepsilon^{-1}\right)^{1/2}$ with a large constant $C$, we obtain
\begin{align}\label{try1}
\mathbb{P}_{\mathfrak{H}}(\mathfrak{H}\in A  )\leq e^{  -O(1)(\log\varepsilon^{-1} )^{3/4}}.
\end{align}
This bound is far from satisfactory. In particular, $e^{-O(1)(\log\varepsilon^{-1} )^{3/4}}>>\varepsilon^{\alpha}$ for all $\alpha>0$. \\

An possible remedy which naturally occurs is to optimalize our choice of the resampling domain $[\ell,r]$. By allowing $T=r-\ell$ to vary, we obtained the following estimate about $Z_{bdd}$,
\begin{align}\label{T3}
\mathbb{P}_{\mathfrak{H}} \left(Z_{\textup{bdd}}<  e^{-O(1)T^{3} }  \right)< e^{- T^{3}}.
\end{align} 
See Proposition \ref{pro:Z_lowerbound_k} for a precise statement. Therefore,  
\begin{align*}
\mathbb{P}_{\mathfrak{H}}(\mathfrak{H}\in A  )\leq \varepsilon e^{O(1) T^3}+e^{- T^{3 }}
\end{align*}
By setting $e^{O(1) T^3}=\varepsilon^{-1/2}$, we may deduce that for some large $C\in\mathbb{R}$,
\begin{align}\label{try2}
\mathbb{P}_{\mathfrak{H}}(\mathfrak{H}\in A  )\leq \varepsilon^{1/C}.
\end{align}
This estimate is an improvement of the one in \eqref{try1} but still not satisfactory enough. In this paper, we seek for a bound of the order ${\bf{\varepsilon^{1-o(1)}}},$ which implies that the Radon-Nikodym in \eqref{eq:RN} is $L^p$ integrable for any $p>0$, known as the {\em{$ L^{\infty-}$ regularity}} Brownian regularity.\\

Let us re-examine the estimates we did in the above argument. We expect that the bound on the normalizing constant is sharp. Let us briefly explain why. The scaled KPZ line ensemble is stationary around a parabola $-2^{-1}x^2$, i.e. $\mathfrak{H}_k(x)+2^{-1}x^2$ is a stationary process in $x$ for each $k\in\mathbb{N}$. When we resampling a curve over a large interval of size $T$, having in mind that the penalization ($W$) is extremely high for going out of order, the curve tries to rise up and jump over the parabola $-2^{-1}x^2$. For a free Brownian motion, such movement consumes a kinetic cost of order $e^{-T^3}$. Therefore, we expect that typically $Z_{\textup{bdd}}\approx e^{-T^3}$, which matches with what we have in \eqref{T3}. We admit that there is possible room to improve the estimates in \eqref{T3} due to the presence of the term $O(1)$. However, it seems quite difficult to pursue in this direction and we do not attempt to do so in this paper. 

Expecting that the estimate on $Z_{\textup{bdd}}$ is sharp, the overestimate in above argument comes mainly from the step where we estimated the numerator, i.e.
\begin{align}\label{error}
\E_{\free}[\mathbbm{1} \{B\in A \}\cdot W ]\leq \bP_{\free}(B\in A).
\end{align} 

\vspace{0.5cm}

In the next subsection, we explain how to improve this step by introducing the "soft" jump ensemble method. Let conclude this section by comparing the resampling methods developed in \cite{Wu21} and this paper, see Table \ref{3stages}. 

\begin{center}
\begin{table}[H]
\begin{tabular}{ccc} 
\hline
 Level & Method& Results \\\hline\\[0.1cm]  
 Basic & Simple resampling &Uniform tail bound  \\[0.3cm] 
 Intermediate & Inductive middle reconstruction &Local fluctuation, Tightness  \\[0.3cm]
 Advanced & "Soft" jump ensemble & $L^{\infty-}$ Brownian regularity\\[0.3cm]
 \hline 
\end{tabular}
\caption{A brief summary of the different stages of the resampling methods developed in \cite{Wu21} and in this paper. The first two stages were developed in \cite{Wu21} and the more delicate stage is developed in this paper. The inputs for the first stage are the one point tail estimates (\cite{CG1, CG2}) of the top layer $\mathfrak{H}^t_1$ and stationarity of $\mathfrak{H}^t$ around parabola $\frac{x^2}{2}$. Then the results of previous layers will serve as inputs for the next one. In the first stage, the resampling is straightforward and uniform tail estimates are established for each layer $\mathfrak{H}^t_k, k\in\mathbb{N}$. In the second stage, an inductive middle reconstruction is used to separate curves to obtain quantitative control of the normalizing constant. A more advanced resampling technique, the soft jump ensemble, is an adaption of the jump ensemble method developed in \cite{Ham1} for non-intersecting line ensembles.}\label{3stages}
\end{table}
\end{center} 

\subsection{The soft jump ensemble}
In order to sharpen the estimate in \eqref{error}, a new ensemble $J$, called the "soft" jump ensemble, is introduced. This method is heavily inspired by the jump ensemble introduced in \cite{Ham1} in which the author studies non-intersecting Brownian Gibbsian ensembles. Since curves in the scaled KPZ line ensemble do not avoid each other, the soft jump ensemble is invented to adapt this intersecting nature. We will just call it the jump ensemble in the rest of the paper to save us some effort.

The jump ensemble $J$ interpolates between the scaled KPZ line ensemble $\mathfrak{H}$ and the Brownian bridge ensemble and seeks a good balance between those two.  The idea is to decompose $W$ into $$W=W_{\textup{jump}}\cdot W_{\textup{rest}}. $$
We will choose both $W_{\textup{jump}}$ and $W_{\textup{rest}}$ to be less than or equal to $1$. There are two extremal cases, $W_{\textup{jump}}\equiv 1$ or $W_{\textup{jump}}\equiv W$. The first case reduces to the resampling method discussed in the previous section. The second case is informative as the jump ensemble measure is equivalent to the conditional measure of the line ensemble, which is the measure we need to study. $W_{\textup{jump}}$ will be deigned carefully to provide a meaningful resampling strategy.\\[-0.25cm]

Let us first explain the argument framework. We may rewrite
\begin{align*}
\frac{\E_{\free}[\mathbbm{1} \{B\in A \} \cdot W ]}{\E_{\free}[W]} =&\frac{\E_{\free}[\mathbbm{1} \{B\in A \} \cdot W_{\textup{rest}}  \cdot W_{\textup{jump}} ]}{\E_{\free}[\cdot W_{\textup{rest}}  \cdot W_{\textup{jump}} ]}\bigg/\frac{\E_{\free}[  W_{\textup{jump}}]}{\E_{\free}[  W_{\textup{jump}}]}\\
=&\frac{\E_J[\mathbbm{1} \{J\in A \} \cdot W_{\textup{rest}}   ]}{\E_J[  W_{\textup{rest}}   ]}.
\end{align*}
Here the law of the jump ensemble $J$ is specified through the Radon-Nikodym relation
\begin{align*}
\frac{d \bP_J}{ d\bP_{\free}}(B)=\frac{W_{\textup{jump}}(B)}{\E_{\free}[W_\textup{jump}]}.
\end{align*}
Because  $W_{\textup{rest}}\leq 1$, it holds that
\begin{align}\label{error2}
\frac{\E_{\free}[\mathbbm{1} \{B\in A \} \cdot W]}{\E_{\free}[W]} 
\leq &\frac{\bP_J(J\in A)}{\E_J[W_{\textup{rest}}]}.
\end{align}  
The above computation reduces the one in Subsection \ref{sec:jump difficulty} when $W_\textup{jump}=1$. See Section \ref{sec:framework} for a detailed account of this framework.\\
 
Since $W\leq W_{\textup{rest}}\leq 1$, \eqref{error2} is sharper than \eqref{error}.  Unfortunately in \eqref{error2}, bounding $\bP_J(J\in A)$ from above and $\E_J[W_{\textup{rest}}]$ from below oppose each other. To achieve a bound of order $\varepsilon^{1-o(1)}$ as we wish, we seek for the following balance when estimating $\bP_J(J\in A)$ and $\E_J[W_{\textup{rest}}]$ in \eqref{error2}.

\begin{align}\label{JA}
\bP_J( J\in A )= \varepsilon^{1-o(1)},
\end{align}
and
\begin{align}\label{JW}
\E_J[ W_{\textup{rest}} ]=\varepsilon^{o(1)}.
\end{align}
In light of $\varepsilon= \bP_{\free}( B\in A )$, \eqref{JA} indicates that the jump ensemble $J$ is comparable to the Brownian bridge ensemble. On the other hand, \eqref{JW} shows that the jump ensemble $J$ is a good approximation to the scaled KPZ line ensemble $\mathfrak{H}$. In particular, $W_{\textup{rest}}$ needs not to be small. The two requirements are opposing each other since \eqref{JA} wants to make $W_{\textup{jump}}$ close to $1$ but \eqref{JW} wants to make $W_{\textup{jump}}$ close to $W$. 

Our main contribution when designing the soft jump ensemble $J$ is to carefully balance $W_{\textup{jump}}$ and $W_{\textup{rest}}$ and to find a common ground between these two Boltzmann weights. This framework is further explained in more detail in Section \ref{sec:framework}. We refer the readers to Section \ref{sec:jump ensemble} for more explanation about the design concept of $W_{\textup{jump}}$ and how it is constructed precisely as it requires significantly heavier notation.

\subsection{Outline} Section \ref{def:basics of line ensembles} contains various definitions necessary to describe Gibbsian line ensembles. Section \ref{sec:tools} records lemmas about stochastic monotonicity and strong Markov property. Section \ref{sec:framework} provides a framework of the main argument and splits the proof into three key propositions. The propositions are further proved in Section \ref{sec:Fav}, Section \ref{sec:denominator} and Section \ref{sec:numerator} respectively. The main tool, the soft jump ensemble is designed in Section \ref{sec:jump ensemble}.

\subsection{Notation}
We would like to explain some notation here. The natural numbers are defined to be $\N = {1, 2, . . .} .$ Events are denoted in a special font $\mathsf{E}$, their indicator function is written either as $\mathbbm{1}_{\mathsf{E}}$  and their complement is written as $\mathsf{E}^c$.

Universal constants will be generally denoted by $C$ and constants depending only on $k$ will be denoted as $D(k)$. We label the ones in statements of theorems, lemmas and propostions with subscripts (e.g. $D_1(k), D_2(k), \cdots$ based on their order of occurrence) but the constants in proofs may change value from line to line.

\subsection{Acknowledgment} The author extends thanks to Alan Hammond for valuable discussions and the three Minerva lectures that Alan Hammond has given at Columbia in Spring 2019 about Gibbsian resampling techniques. The author also thanks Ivan Corwin, Alan Hammond and Milind Hegde for valuable suggestions on an earlier draft of this paper.

\section{$\Ham$-Brownian Gibbs property}
In this section we introduce the notion of $\Ham$-Brownian Gibbsian line ensembles. We first introduce the basic definitions necessary to define Gibbsian line ensembles in Subsection~\ref{def:basics of line ensembles} and then provide properties of such Gibbsian line ensembles.

\subsection{Line ensembles and the $\Ham$-Brownian Gibbs property}\label{def:basics of line ensembles}
\begin{definition}[Line ensembles]\label{def:line-ensemble}
Let $\Sigma$ be an interval of $\Z$ and let $\Lambda$ be a subset of $\R$. Consider the set $C(\Sigma\times \Lambda,\R)$ of continuous functions $f:\Sigma\times\Lambda \to \R$ endowed with the topology of uniform convergence on compact subsets of $\Sigma\times \Lambda$, and let $\mathcal{C}(\Sigma\times \Lambda,\R)$ denote the sigma-field generated by Borel sets in $C(\Sigma\times \Lambda,\R)$. A $\Sigma\times \Lambda$-indexed line ensemble $\mathcal{L}$ is a random variable on a probability space $(\Omega,\mathcal{B},\mathbb{P})$, taking values in $C(\Sigma\times \Lambda,\R)$ such that $\mathcal{L}$ is a measurable function from $\mathcal{B}$ to $\mathcal{C}(\Sigma\times \Lambda,\R)$.
\end{definition}

We think of such line ensembles as multi-layer random curves. For integers $k_1<k_2$, let $[k_1,k_2]_{\Z} := \{k_1,k_1+1,\ldots,k_2\}$. We will generally write $\mathcal{L}:\Sigma\times \Lambda\to \R$ even though it is not $\mathcal{L}$, but rather $\mathcal{L}(\omega)$ for each $\omega\in \Omega$ which is such a function. We will also sometimes specify a line ensemble by only giving its law without reference to the underlying probability space.  We write $\mathcal{L}_i(\cdot):= \big(\mathcal{L}(\omega)\big)(i,\cdot)$ for the label $i\in \Sigma$ curve of the ensemble $\mathcal{L}$. 

We now start to formulate the $\Ham$-Brownian Gibbs property, a key property of the KPZ line ensemble, \cite{CH16}. We adopt the convention that all Brownian motions and bridges have diffusion parameter one.

\begin{definition}[$\Ham$-Brownian bridge line ensemble]\label{def:H_Brownian}
Fix $k_1\leq k_2$ with $k_1,k_2 \in \mathbb{Z}$, an interval $[a,b]\subset \mathbb{R}$ and two vectors $\vec{x},\vec{y}\in \mathbb{R}^{k_2-k_1+1}$. A $[k_1,k_2]_{\Z}\times [a,b]$-indexed line ensemble $\mathcal{L} = (\mathcal{L}_{k_1},\ldots,\mathcal{L}_{k_2})$ is called a free Brownian bridge line ensemble with entrance data $\vec{x}$ and exit data $\vec{y}$ if its law $\mathbb{P}^{k_1,k_2,[a,b],\vec{x},\vec{y}}_{\free}$ is that of $k_2-k_1+1$ independent standard Brownian bridges starting at time $a$ at the points $\vec{x}$ and ending at time $b$ at the points $\vec{y}$.

A Hamiltonian $\Ham$ is defined to be a measurable function $\Ham:\mathbb{R}\to [0,\infty]$.
Given a  Hamiltonian $\Ham$ and two measurable function $f,g:(a,b)\to \mathbb{R}\cup\{\pm\infty\}$, we define the $\Ham$-Brownian bridge line ensemble with entrance data $\vec{x}$, exit data $\vec{y}$ and boundary data $(f,g)$ to be a $[k_1,k_2]_{\Z}\times (a,b)$-indexed line ensemble $\mathcal{L} = (\mathcal{L}_{k_1},\ldots, \mathcal{L}_{k_2})$ with law $\mathbb{P}^{k_1,k_2,(a,b),\vec{x},\vec{y},f,g}_{\Ham}$ given according to the following Radon-Nikodym derivative relation:
\begin{equation*}
\frac{\textup{d}\mathbb{P}_{\Ham}^{k_1,k_2,(a,b),\vec{x},\vec{y},f,g}}{\textup{d}\mathbb{P}_{\free}^{k_1,k_2,(a,b),\vec{x},\vec{y}}}(\mathcal{L}) := \frac{W_{\Ham}^{k_1,k_2,(a,b),\vec{x},\vec{y},f,g}(\mathcal{L})}{Z_{\Ham}^{k_1,k_2,(a,b),\vec{x},\vec{y},f,g}}.
\end{equation*}
Here we adopt conventions that $\mathcal{L}_{k_1-1}=f$, $\mathcal{L}_{k_2+1}=g$ and define the Boltzmann weight
\begin{equation}\label{def:Boltzmann_Brownian}
W_{\Ham}^{k_1,k_2,(a,b),\vec{x},\vec{y},f,g}(\mathcal{L}):= \exp\Bigg\{-\sum_{i=k_1-1}^{k_2}\int_a^b \Ham\Big(\mathcal{L}_{i+1}(u)-\mathcal{L}_{i}(u)\Big)du\Bigg\},
\end{equation}
and the normalizing constant
\begin{equation}\label{def:normalcont_Brownian}
Z_{\Ham}^{k_1,k_2,(a,b),\vec{x},\vec{y},f,g} :=\mathbb{E}^{k_1,k_2,(a,b),\vec{x},\vec{y}}_{\free}\Big[W_{\Ham}^{k_1,k_2,(a,b),\vec{x},\vec{y},f,g}(\mathcal{L})\Big],
\end{equation}
where $\mathcal{L}$ in the above expectation is distributed according to the measure $\mathbb{P}^{k_1,k_2,(a,b),\vec{x},\vec{y}}_{\free}$.
\end{definition}

$\Ham$-Brownian Gibbs property could be viewed as a spatial Markov property, more specifically, it provides a description of the conditional law inside a compact set.
\begin{definition}[$\Ham$-Brownian Gibbs property]\label{def:H-BGP}
A $\Sigma \times \Lambda$-indexed line ensemble $\mathcal{L}$ satisfies the $\Ham$-Brownian Gibbs property if for all $K=\{k_1,\ldots, k_2\}\subset \Sigma$ and $(a,b)\subset \Lambda$, its conditional law inside $K\time \Lambda$ takes the following form,
\begin{equation*}
\textrm{Law}\left(\mathcal{L} \left\vert_{K \times (a,b)} \textrm{conditional on } \mathcal{L}\right\vert_{(\Sigma \times \Lambda) \setminus ( K \times (a,b) )} \right) =\bP
\end{equation*}
Here $f:=\mathcal{L}_{k_1-1}$ and $g:=\mathcal{L}_{k_2+1}$ with the convention that if $k_1-1\notin\Sigma$ then $f\equiv +\infty$ and likewise if $k_2+1\notin \Sigma$ then $g\equiv -\infty$; we have also set $\vec{x}=\big(\mathcal{L}_{k_1}(a),\cdots ,\mathcal{L}_{k_2}(a)\big)$ and $\vec{y}=\big(\mathcal{L}_{k_1}(b),\ldots ,\mathcal{L}_{k_2}(b)\big)$. 

This following description of Gibbs property using conditional expectation is equivalent and is convenient for computations sometimes. For $K\subset \Z$ and $S\subset \R$, let $C(K\times S, \R)$ be the space of continuous functions from $K\times S\to \R$. Then
a $\Sigma \times \Lambda$-indexed line ensemble $\mathcal{L}$
enjoys the $\Ham$-Brownian Gibbs property if and only if for any $K=\{k_1,\ldots, k_2\}\subset \Sigma$ and $(a,b)\subset \Lambda$, and any Borel function $F$ from $C\left(K\times(a,b), \R\right)\to \R$, $\mathbb{P}$-almost surely
\begin{equation*}
\E\left[F(\mathcal{L}\left\vert_{K\times(a,b)}\right. )\, \right\vert\, \Fext\big(K\times (a,b)\big)\Big] =  \mathbb{E}_{\Ham}^{k_1,k_2,(a,b),\vec{x},\vec{y},f,g}\left[F(\mathcal{L}_{k_1},\ldots, \mathcal{L}_{k_2})\right],
\end{equation*}
where $\vec{x}$,$\vec{y}$,$f$ and $g$ are defined in the previous paragraph and where \glossary{$\Fext\left(K\times (a,b)\right)$, Sigma-field generated by a line ensemble outside $K\times (a,b)$}
\begin{equation}\label{starstareleven}
\Fext\left(K\times (a,b)\right) := \sigma\left(\mathcal{L}_{i}(s): (i,s)\in \Sigma\times \Lambda \setminus K\times (a,b)\right)
\end{equation}
is the exterior sigma-field generated by the line ensemble outside $K\times (a,b)$. On the left-hand side of the above equality  $\mathcal{L}\left\vert_{K\times(a,b)}\right. $ is the restriction to $K\times (a,b)$ of curves distributed according to $\mathbb{P}$, while on the right-hand side $\mathcal{L}_{k_1},\ldots, \mathcal{L}_{k_2}$ are curves on $(a,b)$ distributed according to $\mathbb{P}_{\Ham}^{k_1,k_2,(a,b),\vec{x},\vec{y},f,g}$.
\end{definition}

\begin{remark} The scaled KPZ line ensemble enjoys an $\Ham_t$-Brownian Gibbs property with  
\begin{align}
\Ham_t(x) = e^{t^{1/3}x}.
\end{align}
Throughout this paper, we only focus on this specific Hamiltonian.
\end{remark}

\subsection{Strong Gibbs property and stochastic monotonicity}\label{sec:tools}
We record some important properties, developed in \cite[Section 2]{CH16}, about $\Ham$-Gibbsian line ensembles in this section. We begin with the strong Gibbs property, which enable us to resample the trajectory within a stopping domain as opposed to a deterministic interval. 

\begin{definition}\label{def:stopdm}
Let $\Sigma$ be an interval of $\Z$, and $\Lambda$ be an interval of $\R$. Consider a $\Sigma\times\Lambda$-indexed line ensemble $\mathcal{L}$ which has the $\Ham$-Brownian Gibbs property for some Hamiltonian $\Ham$. For $K=\{k_1,\ldots, k_2\}\subseteq \Sigma$ and $(\ell,r)\subseteq \Lambda$, $\Fext\big(K\times (\ell,r)\big)$ denotes the sigma field generated by the data outside $K\times (\ell,r)$. The random variable $(\mathfrak{l},\mathfrak{r})$ \glossary{$(\mathfrak{l},\mathfrak{r})$, Stopping domain} is called a {\it $K$-stopping domain} if for all $\ell<r$,
\begin{equation*}
\big\{\mathfrak{l} \leq \ell , \mathfrak{r}\geq r\big\} \in \Fext\big(K\times (\ell,r)\big).
\end{equation*}
\end{definition}

For $K=\{k_1,\dots,k_2\}\subset \mathbb{Z}$, define
\begin{align*}
C_{K}:=\left\{ (\ell,r,f_{k_1},\dots,f_{k_2})\,:\, \ell<r,\ (f_{k_1},\dots,f_{k_2})\in C(K\times (\ell,r)) \right\}.
\end{align*}

\begin{lemma}\label{lem:stronggibbs}
Consider a $\Sigma\times\Lambda$-indexed line ensemble $\mathcal{L}$ which has the $\Ham$-Brownian Gibbs property. Fix $K=\{k_1,\ldots,k_2\}\subseteq \Sigma$. For all random variables $(\mathfrak{l},\mathfrak{r})$ which are $K$-stopping domains for $\mathcal{L}$, the following strong $\Ham$-Brownian Gibbs property holds: for all Borel functions $F: C_K\to\mathbb{R}$, $\bP$-almost surely,
\begin{equation}\label{eqn:stronggibbs}
\E\bigg[ F\Big(\mathfrak{l},\mathfrak{r}, \mathcal{L} \big\vert_{K\times (\mathfrak{l},\mathfrak{r})} \Big) \Big\vert \Fext\big(K\times (\mathfrak{l},\mathfrak{r})\big) \bigg]
=\mathbb{E}_{\mathbf{H}}^{k_1,k_2,(\mathfrak{l},\mathfrak{r}),\vec{x},\vec{y},f,g}\Big[F\big(a,b, \mathcal{L}_{k_1},\ldots, \mathcal{L}_{k_2}\big) \Big],
\end{equation}
where $a=\mathfrak{l}$, $b=\mathfrak{r}$, $\vec{x} = \{\mathcal{L}_i(\mathfrak{l})\}_{i=k_1}^{k_2}$, $\vec{y} = \{\mathcal{L}_i(\mathfrak{r})\}_{i=k_1}^{k_2}$, $f(\cdot)=\mathcal{L}_{k_{1}-1}(\cdot)$ (or $\infty$ if $k_1-1\notin \Sigma$), $g(\cdot)=\mathcal{L}_{k_2+1}(\cdot)$ (or $-\infty$ if $k_2+1\notin \Sigma$). On the left-hand side $ \mathcal{L} \big\vert_{K\times (\mathfrak{l},\mathfrak{r})}$ is the restriction of curves distributed according to $\bP$ and on the right-hand side $\mathcal{L}_{k_1},\ldots, \mathcal{L}_{k_2}$ is distributed according to $\mathbb{P}_{\mathbf{H}}^{k_1,k_2,(\mathfrak{l},\mathfrak{r}),\vec{x},\vec{y},f,g}$.
\end{lemma}

For a convex Hamiltonian $\Ham$ (such as $\Ham_t$), $\Ham$-Brownian bridge line ensembles $\bP$. 

\begin{lemma}\label{monotonicity1}
Fix $k_1\leq k_2\in \Z$, $a<b$. Consider two pairs of vectors $\vec{x}^{(i)},\vec{y}^{(i)}\in \R^{k_2-k_1+1}$ for $i\in \{1,2\}$ such that $x^{(1)}_{j}\leq x^{(2)}_{j}$ and $y^{(1)}_{j}\leq y^{(2)}_{j}$ for all $j=k_1,\ldots,k_2$. Consider two pairs of measurable functions $(f^{(i)},g^{(i)})$ for $i\in \{1,2\}$ such that $f^{(i)}:(a,b)\rightarrow \R\cup\{\infty\}$, $g^{(i)}:(a,b)\rightarrow \R\cup\{-\infty\}$ and for all $s\in (a,b)$, $f^{(1)}(s)\geq f^{(2)}(s)$ and $g^{(1)}(s)\geq g^{(2)}(s)$. For $i\in \{1,2\}$, let $\mathcal{Q}^{(i)}=\{\mathcal{Q}^{(i)}_j\}_{j=k_1}^{k_2}$ be a $\{k_1,\ldots,k_2\}\times (a,b)$-indexed line ensemble on a probability space $\big(\Omega^{(i)},\mathcal{B}^{(i)},\mathbb{P}^{(i)} \big)$ where $\mathbb{P}^{(i)}$ equals $\mathbb{P}_{\mathbf{H}}^{k_1,k_2,(a,b),\vec{x}^{(i)},\vec{y}^{(i)},f^{(i)},g^{(i)}}$
(i.e. $\mathcal{Q}^{(i)}$ has the $\Ham$-Brownian Gibbs property with entrance data $\vec{x}^{(i)}$, exit data $\vec{y}^{(i)}$ and boundary data $(f^{(i)},g^{(i)})$).

If the Hamiltonian $\Ham:\R\to [0,\infty)$ is convex then there exists a coupling of the probability measures $\mathbb{P}^{(1)}$ and $\mathbb{P}^{(2)}$ such that almost surely $\mathcal{Q}^{(1)}_j(s)\leq \mathcal{Q}^{(2)}_j(s)$ for all $j\in \{k_1,\ldots, k_2\}$ and all $s\in (a,b)$.
\end{lemma}

\section{Proof framework for Theorem \ref{thm:main}}\label{sec:framework}
In this section we attempt to illustrate the structure our soft jump ensemble resampling argument towards proving the main Theorem \ref{thm:main}. We provide the framework of the argument, containing a core inequality \eqref{eqn:framework} and three key propositions, i.e. Propositions \ref{clm:Fav}, \ref{clm:denominator} and \ref{clm:numerator}. Combining these three Propositions and the core inequality \eqref{eqn:framework}, the main Theorem \ref{thm:main} readily follows. We prove these Propositions in sections \ref{sec:Fav}, \ref{sec:denominator} and \ref{sec:numerator} respectively.\\[0.12cm]
\indent In this section we list essential properties that are necessary for deriving \eqref{eqn:framework} for the reader's ease. We postpone explicit definitions of various objects to sections \ref{sec:Fav} and \ref{sec:jump ensemble} and refer readers to those sections.
\subsection{Introducing the resampling domain} We begin with introducing the domain, a subset in $\mathbb{N}\times\mathbb{R}$, where we run the resampling for $\mathfrak{H}$, i.e. apply the Gibbs property. We also introduce some relevant subsets, in which the boundary conditions are encoded. See Figure~\ref{fig:Boxdomain} for an illustration of these region.\\[0.12cm]
\indent Fix $k\in\mathbb{N}$ and $s\geq 1$. Take $\ell_0<0$ and $r_0>s$, whose precise definitions and dependence on $k$ can be found in \eqref{def:l0r0}. Let $U\subset \mathbb{R}^k\times \mathbb{R}^k$ be the collection of $(\bar{\ell},\bar{r})=(\ell_1,\ell_2,\dots,\ell_k,r_1,r_2,\dots,r_k)$ such that 
\begin{equation}\label{lr_order}
\begin{split}
 \ell_0 < \ell_1<\ell_2<\dots<\ell_k< 0,\  
s< r_k<r_{k-1}<\dots <r_1<r_0.
\end{split} 
\end{equation}
We assume $(\bar{\ell},\bar{r})\in U$ throughout this section. Define
\begin{align}
\textup{dom}(\bar{\ell},\bar{r})\coloneqq&\bigcup_{j=1}^k\{ j \}\times   [\ell_j, r_j ]\label{def:domain},  \\
\textup{bdd}(\bar{\ell},\bar{r})\coloneqq&\{1\}\times \{\ell_1  , r_1 \}\cup \bigcup_{j=2}^k\{ j \}\times ([\ell_{j-1},\ell_{j }]\cup[r_{j} ,r_{j-1}]) \cup \{k+1\}\times [\ell_k  , r_k ], \label{def:boundary}\\
\textup{ext}(\bar{\ell},\bar{r})\coloneqq&\textup{cl}( \mathbb{N}\times \mathbb{R}\setminus \textup{dom}(\bar{\ell},\bar{r})), \label{def:ext}\\
\textup{Box}\coloneqq &[1,k+1]_{\mathbb{Z}}\times [\ell_0,r_0],\label{def:Box}\\
\textup{Jct}(\bar{\ell},\bar{r})\coloneqq&\textup{cl}( \textup{Box} \setminus \textup{dom}(\bar{\ell},\bar{r})),\label{def:Jct}
\end{align}
where $\textup{cl}$ means taking the closure. We write ${C}(\textup{dom}(\bar{\ell},\bar{r}) ), {C}(\textup{bdd}(\bar{\ell},\bar{r}) ),{C}(\textup{Jct}(\bar{\ell},\bar{r}) )$ and ${C}(\textup{Box}(\bar{\ell},\bar{r}) )$ to denote the spaces of continuous functions defined on $\textup{dom}(\bar{\ell},\bar{r}),\textup{bdd}(\bar{\ell},\bar{r}),\textup{Jct}(\bar{\ell}\bar{r})$ and $\textup{Box}(\bar{\ell},\bar{r})$ respectively.  
\begin{figure}
\includegraphics[width=10cm]{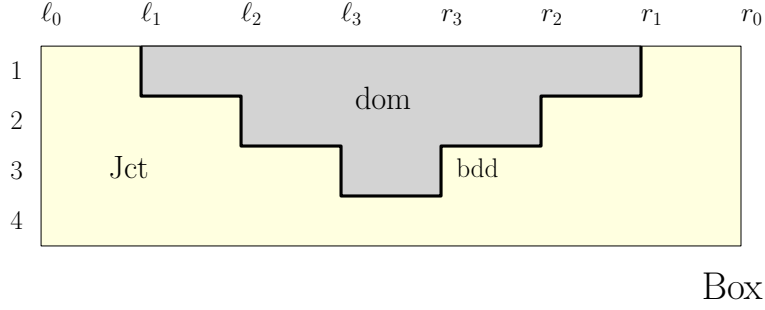}
\caption{An illustration when $k=3$. We will run the resampling for the first three curve $\mathfrak{H}_1,\mathfrak{H}_3, \mathfrak{H}_3$ inside $\textup{dom}$, which resembles the shape of an upside-down pyramid. More precisely, $\textup{dom}:= \left(\{1\}\times[\ell_1,r_1]\right) \cup \left(\{2\} \times [\ell_2,r_2]\right)\cup \left(\{3\}\times[\ell_3,r_3]\right)$. We use $\textup{bdd}$ to denote the region where the values of $\mathfrak{H}$ involves in the resampling as the boundary data. $\textup{ext}$ is the the closure of the complement of $\textup{dom}$. Box and Jct are introduced for later notational convenience.}\label{fig:Boxdomain}
\end{figure}
 We proceed to formulate the Gibbs property of $\mathfrak{H}$ restricted to $\textup{dom}(\bar{\ell},\bar{r})$. Let $f=(f_1,f_2,\dots, f_{k+1}):\textup{bdd}(\bar{\ell},\bar{r})\to \mathbb{R}$ be a continuous function. Let $\mathbb{P}^{\textup{dom}(\bar{\ell},\bar{r}),f}_{\free}$ be the law of $k$ independent Brownian bridges $(B_1,B_2,\dots,B_k)$ defined in $\textup{dom}(\bar{\ell},\bar{r})$ with $B_j(\ell_j)=f_j(\ell_j) $ and $B_j(r_j)=f_j(r_j)$. We write $\mathbb{E}^{\textup{dom}(\bar{\ell},\bar{r}),f}_{\free}$ for its expectation.\\[0.12cm]
\indent Recall that $\mathbf{H}_t(x)=e^{t^{1/3}x}$. Define the law $\mathbb{P}^{\textup{dom}(\bar{\ell},\bar{r}),f}_{\mathbf{H}_t}$ through the following Radon-Nikodym derivative relation,
\begin{equation}
\frac{{d}\mathbb{P}^{\textup{dom}(\bar{\ell},\bar{r}),f}_{\mathbf{H}_t}}{{d}\mathbb{P}^{\textup{dom}(\bar{\ell},\bar{r}),f}_{\free}}(\mathcal{L})\coloneqq\frac{ W^{\textup{dom}(\bar{\ell},\bar{r}),f}_{\mathbf{H}_t}(\mathcal{L})}{ \mathbb{E}^{\textup{dom}(\bar{\ell},\bar{r}),f}_{\free}\left[ W^{\textup{dom}(\bar{\ell},\bar{r}),f}_{\mathbf{H}_t}\right]}.
\end{equation}
The above Boltzmann weight $W_{\mathbf{H}_t}^{ \textup{dom}(\bar{\ell},\bar{r}), f }(\mathcal{L})$ represents the interaction between the curves $\mathcal{L}_1, \cdots, \mathcal{L}_k$ in $\textup{dom}(\bar{\ell},\bar{r})$ and their interaction with the boundary data $f$, defined as
\begin{equation}\label{equ:WH}
\begin{split}
W_{\mathbf{H}_t}^{ \textup{dom}(\bar{\ell},\bar{r}), f }(\mathcal{L})\coloneqq &\prod_{j=1}^{k-1} \exp\left( -\int_{[\ell_j,\ell_{j+1}]\cup [r_{j+1},r_j]}  \mathbf{H}_t(f_{j+1}(x)-\mathcal{L}_j(x))\, dx  \right)\\
&\times \prod_{j=1}^{k-1} \exp\left( -\int_{[\ell_{j+1},r_{j+1}]} \mathbf{H}_t(\mathcal{L}_{j+1}(x)-\mathcal{L}_j(x))\, dx \right)\\
&\times \exp\left( -\int_{[\ell_{k},r_k]}  \mathbf{H}_t(f_{k+1}(x)-\mathcal{L}_k(x))\, dx \right).
\end{split}
\end{equation}

\vspace{0.1cm}
Now we are ready to formulate the $\mathbf{H}_t$-Brownian Gibbs property of $\mathfrak{H}^t$ on $\textup{dom}(\bar{\ell}, \bar{r})$. Conditional on the boundary data, inside the unside-down pyramid $\textup{dom}(\bar{\ell}, \bar{r})$, the law of $\mathfrak{H}$ is equivalent to $\mathbb{P}^{\textup{dom}(\bar{\ell},\bar{r}),f}_{\mathbf{H}_t}$.
\begin{lemma}
Fix $t>0$, $k\in\mathbb{N}$ and $(\bar{\ell},\bar{r})\in U$. Let $\mathcal{F}_{\textup{ext}}(\bar{\ell},\bar{r})$ be the $\sigma$-algebra generated by $\mathfrak{H}^t|_{\textup{ext}(\bar{\ell},\bar{r})}$, i.e.
\begin{equation}\label{def:Fext}
\mathcal{F}_{\textup{ext}}(\bar{\ell},\bar{r})\coloneqq \sigma (\mathfrak{H}^t|_{\textup{ext}(\bar{\ell},\bar{r})}).
\end{equation}
Let $f=\mathfrak{H}^t|_{\textup{bdd}(\bar{\ell},\bar{r})}$. For any Borel function $F: {C}(\textup{dom}(\bar{\ell},\bar{r}))\to \mathbb{R}$, almost surely it holds that
\begin{equation}
\mathbb{E}[F(\mathfrak{H}^t|_{\textup{dom}(\bar{\ell},\bar{r})})\, |\, \mathcal{F}_{\textup{ext}}(\bar{\ell},\bar{r})]=\mathbb{E}^{\textup{dom}(\bar{\ell},\bar{r}),f}_{\mathbf{H}_t}[F(\mathcal{L})].
\end{equation}  
\end{lemma} 
Actually, we need to apply the $\mathbf{H}_t$-Brownian Gibbs property over stopping domains, known as the strong $\mathbf{H}_t$-Brownian Gibbs property. Let us first recall the definition of stopping domains in the current context. Consider a random variable $
 (\bar{\mathfrak{l}},\bar{\mathfrak{r}} )=(\mathfrak{l}_1,\mathfrak{l}_2,\dots,\mathfrak{l}_k,\mathfrak{r}_1,\mathfrak{r}_2,\dots,\mathfrak{r}_k)$ that takes value in $U$. The explicit definition of $(\bar{\mathfrak{l}},\bar{\mathfrak{r}})$ can be found in \eqref{def:fraklr}.
\begin{definition}
The random variable $ (\bar{\mathfrak{l}},\bar{\mathfrak{r}})$ is called a stopping domain if the event
\begin{align*}
\mathfrak{l}_j\leq \ell_j,\ \mathfrak{r}_j\geq r_j\ \textup{for all}\ j\in [1,k]_{\mathbb{Z}}
\end{align*}  is	 $\mathcal{F}_{\textup{ext}}(\bar{\ell},\bar{r})$-measurable with $\bar{\ell}=(\ell_1,\ell_2,\dots,\ell_k)$ and $\bar{r}=(r_1,r_2,\dots,r_k)$.
\end{definition}
Form now on, we assume $(\bar{\mathfrak{l}},\bar{\mathfrak{r}} )$ is a stopping domain. See Lemma \ref{lem:stop} for a proof. We record the strong $\mathbf{H}_t$-Brownian Gibbs property in the following lemma. Define
\begin{align*}
{C}_{\textup{dom}}\coloneqq &\amalg \{(\bar{\ell},\bar{r})\times{C}(\textup{dom}(\bar{\ell},\bar{r}))\ |\ (\bar{\ell},\bar{r})\in U \}.
\end{align*} 
We equip ${C}_{\textup{dom}}$ with the topology induced from $U\times {C}([1,k+1]_{\mathbb{Z}}\times \mathbb{R})$. 
\begin{lemma}\label{lem:sMarkov}
Let $t>0$, $k\in\mathbb{N}$ and let $(\bar{\mathfrak{l}},\bar{\mathfrak{r}})$ be a stopping domain that takes value in $U$. Let $f=\mathfrak{H}^t|_{\textup{bdd}(\bar{\mathfrak{l}},\bar{\mathfrak{r}})}$. For any function $F: {C}_{\textup{dom}} \to \mathbb{R}$, almost surely it holds that
\begin{equation}
\mathbb{E}[F(\bar{\mathfrak{l}},\bar{\mathfrak{r}}, \mathfrak{H}^t|_{\textup{dom}(\bar{\mathfrak{l}},\bar{\mathfrak{r}})})\, |\, \mathcal{F}_{\textup{ext}}(\bar{\mathfrak{l}},\bar{\mathfrak{r}})]=\mathbb{E}^{\textup{dom}(\bar{\mathfrak{l}},\bar{\mathfrak{r}}),f}_{\mathbf{H}_t}[F(\mathcal{L})].
\end{equation}  
\end{lemma}

\subsection{The core inequality} In this section we apply the strong $\mathbf{H}_t$-Brownian Gibbs property and derive the core inequality \eqref{eqn:framework}. We begin with formulating the favorable event $\mathsf{Fav}$ which confines the behavior of $\mathfrak{H}$ in $\textup{Jct}(\bar{\mathfrak{l}},\bar{\mathfrak{r}})$. 

Define
\begin{align*}
{C}_{\textup{Jct}}\coloneqq &\amalg \{(\bar{\ell},\bar{r})\times{C}(\textup{Jct}(\bar{\ell},\bar{r}))\ |\ (\bar{\ell},\bar{r})\in U \}.
\end{align*}
We equip ${C}_{\textup{Jct}}$ with the topology induced from $U\times {C}([1,k+1]_{\mathbb{Z}}\times \mathbb{R})$. Let $\mathcal{G}\subset {C}_{\textup{Jct}}$ be a Borel set which contains the realization of all nice boundary data for our purpose. The explicit form of $\mathcal{G}$ can be found in Definition~\ref{def:G}. See Lemma \ref{lem:measG} for the measurability of $\mathcal{G}$. Define the favorable event
\begin{equation}\label{def:Fav}
\mathsf{Fav}\coloneqq\{ (\bar{\mathfrak{l}},\bar{\mathfrak{r}},\mathfrak{H}^t|_{\textup{Jct}(\mathfrak{\bar{l}},\mathfrak{\bar{r}})})\in\mathcal{G} \}.
\end{equation}  
$\mathsf{Fav}$ is measurable because $\mathsf{Fav}=\{\pi\circ(\bar{\mathfrak{l}},\bar{\mathfrak{r}},\mathfrak{H} )\in \mathcal{G}\}$ with $\pi:U\times {C}([1,k+1]_{\mathbb{Z}}\times \mathbb{R}) \to {C}_{\textup{Jct}}$ being the restriction map. Moreover, $\mathsf{Fav}$ is $\mathcal{F}_{\textup{ext}}(\bar{\mathfrak{l}},\bar{\mathfrak{r}})$-measurable since $\textup{Jct}(\bar{\ell},\bar{r})\subset \textup{ext}(\bar{\ell},\bar{r})$.\\[0.12cm]
\indent Next, we discuss the jump ensemble $J$. The jump ensemble $J$ is a proxy of $\mathfrak{H}^t$. The jump ensemble $J$ is obtained by replacing the Boltzmann weight $W^{\textup{dom}(\bar{\ell},\bar{r}),f}_{\mathbf{H}_t}$ by $W_{\textup{jump}}$. The new weight $W_{\textup{jump}}$ is less restrictive in the sense that $W^{\textup{dom}(\bar{\ell},\bar{r}),f}_{\mathbf{H}_t}\leq W_{\textup{jump}}\leq 1$. Therefore, compared to $\mathfrak{H}^t$, $J$ is closer to Brownian bridges. The extra price to pay is to compare the proxy $J$ to $\mathfrak{H}^t$. \\[0.12cm]
\indent We prescribe the law of $J$ below. Given $(\bar{\ell},\bar{r},f_{\textup{J}})\in \mathcal{G} $, we consider the Boltzmann weight $W^{\textup{dom}(\bar{\ell},\bar{r}),f}_{\mathbf{H}_t}  $ with $f=f_{\textup{J}}|_{\textup{bdd}(\bar{\ell},\bar{r})}$. We decompose $W^{\textup{dom}(\bar{\ell},\bar{r}),f}_{\mathbf{H}_t}  $ into two parts,
\begin{align*}
W^{\textup{dom}(\bar{\ell},\bar{r}),f}_{\mathbf{H}_t}(\mathcal{L})=W_{\textup{jump}}(\mathcal{L};f_{\textup{J}})\cdot W_{\textup{rest}}(\mathcal{L};f_{\textup{J}}).
\end{align*}
The explicit form of $W_{\textup{jump}}(\mathcal{L};f_{\textup{J}})$ and $W_{\textup{rest}}(\mathcal{L};f_{\textup{J}})$ can be found in \eqref{def:Wjump} and \eqref{def:Wrest}. The only property we use here is that $0< W_{\textup{rest}}(\mathcal{L};f_{\textup{J}})\leq 1$, see Lemma \ref{lem:Wrest1}. The jump ensemble $J$ consists of continuous random curves defined on $\textup{dom}(\bar{\ell},\bar{r})$. The law of $J$, denoted by $\mathbb{P}_J$, is defined through the following Radon-Nikodym derivative relation,
\begin{align}\label{def:J}
\frac{{d} \mathbb{P}_{J}}{{d} \mathbb{P}_{\free}^{\textup{dom}(\bar{\ell},\bar{r}),f }}(\mathcal{L})\coloneqq \frac{W_{\textup{jump}} (\mathcal{L})  }{\mathbb{E}_{\free}^{\textup{dom}(\bar{\ell},\bar{r}),f } \left[W_{\textup{jump}}\right]}.
\end{align}
We note that even though $W^{\textup{dom}(\bar{\ell},\bar{r}),f}_{\mathbf{H}_t}$ depends only on $f$, $W_{\textup{jump}}$ and $\mathbb{P}_J$ do depend on $ f_{\textup{J}} $.\\

We are now ready to formulate and to prove the core inequality. Let $A\subset {C}_{0,0}([0,s])$ be an arbitrary Borel set. It holds that
\begin{align}\label{eqn:framework}
\mathbb{P} (\mathfrak{H}_k^{t,[0,s]}\in A )\leq \sup\left\{ \frac{ \mathbb{P}_{J}(  J_k^{[0,s]}\in A  )}{  \mathbb{E}_{J}[ W_{\textup{rest}}]}\ \bigg|\ (\bar{\ell},\bar{r},f_{\textup{J}})\in \mathcal{G}  \right\} +\mathbb{P}(\mathsf{Fav}^{\textup{c}} ).
\end{align}
\begin{proof}[{\bf Proof of \eqref{eqn:framework}}]
Note that
\begin{align*}
\mathbb{P}(\mathfrak{H}_k^{t,[0,s ]}\in A)\leq&\mathbb{P}(\mathfrak{H}_k^{t,[0,s]}\in A ,\mathsf{Fav} )+\mathbb{P}(\mathsf{Fav}^{\textup{c}} )=\mathbb{E}[\mathbbm{1}_{\mathsf{Fav}}\cdot \mathbb{E}[\mathbbm{1} \{\mathfrak{H}_k^{[0,s]}\in A\}\, |\, \mathcal{F}_{\textup{ext}}(\bar{\mathfrak{l}},\bar{\mathfrak{r}})]]+\mathbb{P}(\mathsf{Fav}^{\textup{c}} ).
\end{align*}
From \eqref{lr_order} and the strong $\mathbf{H}_t$-Brownian Gibbs property, it holds that\\[-0.3cm]
\begin{align*}
\mathbb{E}[\mathbbm{1}\{\mathfrak{H}_k^{t,[0,s]}\in A\}\, |\, \mathcal{F}_{\textup{ext}}(\bar{\mathfrak{l}},\bar{\mathfrak{r}})]=&\frac{\mathbb{E}^{\textup{dom}( \bar{\mathfrak{l}},\bar{\mathfrak{r}})  ,f}_{\free}\left[\mathbbm{1}\{B_k^{[0,s]}\in A\}\cdot  W^{\textup{dom}(\bar{\mathfrak{l}},\bar{\mathfrak{r}}),f}_{\mathbf{H}_t}\right]}{ \mathbb{E}^{\textup{dom}(\bar{\mathfrak{l}},\bar{\mathfrak{r}}),f}_{\free}\left[ W^{\textup{dom}(\bar{\mathfrak{l}},\bar{\mathfrak{r}}),f}_{\mathbf{H}_t}\right]}.
\end{align*}
Here $f=\mathfrak{H}^t|_{\textup{bdd}(\bar{\mathfrak{l}},\bar{\mathfrak{r}})}$. The event $\mathsf{Fav}$ implies $(\bar{\mathfrak{l}},\bar{\mathfrak{r}},\mathfrak{H}^t|_{\textup{Jct}(\bar{\mathfrak{l}},\bar{\mathfrak{r}})})\in\mathcal{G}$. Hence the term above is bounded by
\begin{align*}
\sup\left\{ \left. \frac{\mathbb{E}^{\textup{dom}( \bar{\mathfrak{l}},\bar{\mathfrak{r}})  ,f}_{\free}\left[\mathbbm{1}\{B_k^{[0,s]}\in A\}\cdot  W^{\textup{dom}(\bar{\mathfrak{l}},\bar{\mathfrak{r}}),f}_{\mathbf{H}_t}\right]}{ \mathbb{E}^{\textup{dom}(\bar{\mathfrak{l}},\bar{\mathfrak{r}}),f}_{\free}\left[ W^{\textup{dom}(\bar{\mathfrak{l}},\bar{\mathfrak{r}}),f}_{\mathbf{H}_t}\right]}\, \right|\ (\bar{\ell},\bar{r},f_\textup{J})\in \mathcal{G}, f=f_{\textup{J}}|_{\textup{bdd}(\bar{\ell},\bar{r})}  \right\}. 
\end{align*}
In short, we obtain
\begin{equation*}
\begin{split}
\mathbb{P}(\mathfrak{H}_k^{t,[0,s]}\in A)\leq&\sup\left\{ \left. \frac{\mathbb{E}^{\textup{dom}( \bar{\mathfrak{l}},\bar{\mathfrak{r}})  ,f}_{\free}\left[\mathbbm{1}\{B_k^{[0,s]}\in A\}\cdot  W^{\textup{dom}(\bar{\mathfrak{l}},\bar{\mathfrak{r}}),f}_{\mathbf{H}_t}\right]}{ \mathbb{E}^{\textup{dom}(\bar{\mathfrak{l}},\bar{\mathfrak{r}}),f}_{\free}\left[ W^{\textup{dom}(\bar{\mathfrak{l}},\bar{\mathfrak{r}}),f}_{\mathbf{H}_t}\right]}\, \right|\ (\bar{\ell},\bar{r},f_\textup{J})\in \mathcal{G}, f=f_{\textup{J}}|_{\textup{bdd}(\bar{\ell},\bar{r})}  \right\}  \\
&+\mathbb{P}(\mathsf{Fav}^{\textup{c}} ).
\end{split}
\end{equation*}

Because $W_{\textup{rest}}(\mathcal{L};f_{\textup{J}})\leq 1$, 
\begin{align*}
 \frac{\mathbb{E}^{\textup{dom}( \bar{\mathfrak{l}},\bar{\mathfrak{r}})  ,f}_{\free}\left[\mathbbm{1}\{B_k^{[0,s]}\in A\}\cdot  W^{\textup{dom}(\bar{\mathfrak{l}},\bar{\mathfrak{r}}),f}_{\mathbf{H}_t}\right]}{ \mathbb{E}^{\textup{dom}(\bar{\mathfrak{l}},\bar{\mathfrak{r}}),f}_{\free}\left[ W^{\textup{dom}(\bar{\mathfrak{l}},\bar{\mathfrak{r}}),f}_{\mathbf{H}_t}\right]}=&  \frac{ \mathbb{E}_{J}[\mathbbm{1}\{ J_k^{[0,s]}\in A\} \cdot  W_{\textup{rest}}]}{  \mathbb{E}_{J}[ W_{\textup{rest}}]} \leq  \frac{ \mathbb{P}_{J}(  J_k^{[0,s]}\in A  )}{  \mathbb{E}_{J}[ W_{\textup{rest}}]}.
\end{align*}
It immediately follows that
\begin{align*} 
\mathbb{P} (\mathfrak{H}_k^{t,[0,s]}\in A )\leq \sup\left\{ \frac{ \mathbb{P}_{J}(  J_k^{[0,s]}\in A  )}{  \mathbb{E}_{J}[ W_{\textup{rest}}]}\ \bigg|\ (\bar{\ell},\bar{r},f_{\textup{J}})\in \mathcal{G}  \right\} +\mathbb{P}(\mathsf{Fav}^{\textup{c}} ).
\end{align*}
\end{proof}

\subsection{Three key propositions}
We state the three key propositions, based on which we prove the main Theorem \ref{thm:main}. We first briefly explain the definitions necessary to state these propositions. \\[-0.25cm]

$\mathsf{Fav}$ is the favorable event defined in \eqref{def:Fav}, which confines the boundary data, i.e. triples $(\bar{\ell},\bar{r},f_{\textup{J}})$ for the resampling. $\mathcal{G}$ is the Borel subset consisting of good triples $(\bar{\ell},\bar{r},f_{\textup{J}})$.  Given $(\bar{\ell},\bar{r},f_{\textup{J}})\in \mathcal{G}$,  $W_{\textup{jump}}$ and $\ W_{\textup{rest}}$ are Boltzmann weights which are carefully designed, see in \eqref{def:Wjump} and \eqref{def:Wrest}.

Let $J$ be the jump ensemble defined in \eqref{def:J}. Furthermore, recall that $\varepsilon=\mathbb{P}_{\free}(B_k^{[0,s]}\in A)$ and $D$ is used to denote a large constant that depends only on $k$ and $s$. The exact value of $D$ may change from line to line.\\[-0.25cm]

The following proposition shows that the favorable event is typical.
\begin{proposition}\label{clm:Fav}
\begin{align}
\mathbb{P}(\mathsf{Fav}^{\textup{c}} )\leq D\varepsilon.
\end{align}
\end{proposition}

For any realization of the favorable event, i.e. a triple  $(\bar{\ell},\bar{r},f_{\textup{J}})$, we are able to estimate the denominator and numerator in the core inequality, \eqref{eqn:framework}, respectively in Proposition \ref{clm:denominator} and \ref{clm:numerator}. The estimate of the denominator, Proposition \ref{clm:denominator} gauges how the jump ensemble $J$ differs from free Brownian bridges while the estimate of the numerator, Proposition \ref{clm:numerator} measures the cost to pay for free Brownian bridges to become the jump ensemble.

\begin{proposition}\label{clm:denominator}
\begin{align}
\inf \{\mathbb{E}_{J}[ W_{\textup{rest}}]\, |\,  (\bar{\ell},\bar{r},f_{\textup{J}})\in \mathcal{G}   \}  \geq D^{-1} \exp\left(-D\left(\log\varepsilon^{-1}\right)^{5/6}\right).
\end{align}
\end{proposition}

\begin{proposition}\label{clm:numerator}
\begin{align*}
\sup \{\mathbb{P}_{J}(  J^{[0,s]}\in A  )\, |\,  (\bar{\ell},\bar{r},f_{\textup{J}})\in \mathcal{G}  \}  \leq D\varepsilon \exp\left( D\left(\log\varepsilon^{-1}\right)^{5/6}\right).
\end{align*}
\end{proposition}

Assuming Propositions \ref{clm:Fav}, \ref{clm:denominator} and \ref{clm:numerator} hold, we may complete the proof of the main Theorem \ref{thm:main}. The rest of this paper is devoted to prove these three key propositions.

\begin{proof}[{\bf Proof of Theorem~\ref{thm:main}}]
Combining \eqref{eqn:framework} and Propositions \ref{clm:Fav}, \ref{clm:denominator} and \ref{clm:numerator}, we conclude
\begin{align*}
\mathbb{P} (\mathfrak{H}_k^{t,[0,s]}\in A )\leq  D \varepsilon \exp\left( D\left(\log\varepsilon^{-1}\right)^{5/6} \right).
\end{align*}
\end{proof}

\section{Favorable Event}\label{sec:Fav}
In this section we explain in detail the definition of the favorable event 
$$\mathsf{Fav}\coloneqq\{ (\bar{\mathfrak{l}},\bar{\mathfrak{r}},\mathfrak{H}^t|_{\textup{Jct}(\mathfrak{\bar{l}},\mathfrak{\bar{r}})})\in\mathcal{G} \}.$$   
$\mathcal{G}$ consists of nice boundary data that meets the four assumptions {\bf C1-C4} in Definition \ref{def:G}. We begin with introducing the parameters involved in the assumptions. The final goal of this section is to prove Proposition \ref{clm:Fav}, which bounds the probability of the complement of the favorable event. \\

Fix $k\in\mathbb{N}$ and $s\geq 1$. Let $E=E(k,s)\geq 10$ be a large constant to be determined later. Take
\begin{align}\label{def:T}
T= E\max\{\left(\log \varepsilon^{-1}\right)^{1/3},s\},
\end{align}
and
\begin{equation}\label{def:l0r0}
\ell_0=-(k+1)T,\ r_0=(k+1)T.
\end{equation}
In other words, $T$ determines the width of the Box region. For $\varepsilon$ small, $T$ is comparable to $\left(\log \varepsilon^{-1}\right)^{1/3}$. This choice of $T$ was pioneered by Hammand \cite{Ham1} for non-intersecting Brownian Gibbsian ensembles. In \cite{Wu21}, the same author studies the uniform tail bound of the normalizing constant for the scaled KPZ line ensemble $\mathfrak{H}^t$, recorded in this paper as Proposition \ref{pro:Z_lowerbound_k} for the reader's convenience. Roughly speaking, with probability $1-\varepsilon$, the normalizing constant is bounded from below by $\varepsilon^{O(1)}$ if it is evaluated in an interval of length $\left(\log \varepsilon^{-1}\right)^{1/3}$. This bound allows us to pick $T\approx \left(\log \varepsilon^{-1}\right)^{1/3}$ as the width of Box.

From now on, we will use $T$ instead of $\varepsilon$ to define various quantities. Define
\begin{equation}\label{dd'D}
\Delta\coloneqq 8\log T,\ d\coloneqq 64T^{-4}(\log T)^2.
\end{equation} 
We will explain how these parameters are chosen when it becomes clear.\\

Next, we discuss how to pick the stopping domains $(\bar{\mathfrak{l}},\bar{\mathfrak{r}})$.  $\mathfrak{H}^t_j$ will be resampled in the interval $[\mathfrak{l}_j,\mathfrak{r}_j]$. In order to ensure $\mathfrak{H}^t_{j}$ lies above $\mathfrak{H}^t_{j+1}$, we need to control the growth of $\mathfrak{H}^t_{j+1}$. A crucial observation of Hammond \cite{Ham1} is the following. If $\mathfrak{H}^t_{j+1}(x)$ is close to a parabola $-2^{-1}x^2$, then the concave majorant of $\mathfrak{H}^t_{j+1}$ is also close to $-2^{-1}x^2$ with derivative close to $-x$. See Lemma \ref{lem:slope}. Therefore, $\mathfrak{H}^t_{j+1}$ can be bounded by a linear function by choosing $\mathfrak{l}_j$ to be an extreme point of the concave majorant. See Figure~\ref{fig:ellj} for an illustration.

\begin{figure}
\includegraphics[width=10cm]{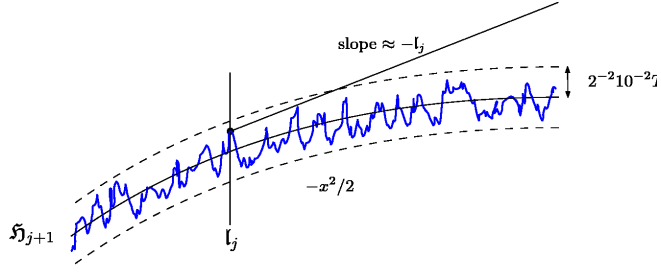}
\caption{Illustration of the choice of $\mathfrak{l}_j$. $\mathfrak{l}_j$ is chosen to be an extreme point of the concave majorant such that $\mathfrak{H}_{j+1}$ can be bounded by a linear function. The key observation is that $\mathfrak{H}_{j+1}(x)$ stays close to a parabola $-2^{-1}x^2$ and therefore the concave majorant of $\mathfrak{H}_{j+1}$ also stays close to $-2^{-1}x^2$ with derivative close to $-x$.}
\label{fig:ellj}
\end{figure}

Now we start to construct $(\bar{\mathfrak{l}},\bar{\mathfrak{r}})$. We first define $(\bar{\mathbf{l}},\bar{\mathbf{r}})$ for continuous functions $f_{\textup{B}}\in{C}(\textup{Box})$ and then define $(\bar{\mathfrak{l}},\bar{\mathfrak{r}})$ by
\begin{equation}\label{def:fraklr} 
(\bar{\mathfrak{l}},\bar{\mathfrak{r}})\coloneqq (\bar{\mathbf{l} },\bar{\mathbf{r} })(\mathfrak{H}^t|_{\textup{Box}}) .
\end{equation}
For $j\in [1,k]_{\mathbb{Z}}$, let $$R_{j+1,\pm}\coloneqq [\pm (k+1-j)T-2^{-1}T,\pm (k+1-j)T+2^{-1}T].$$
Given $f_{\textup{B}}=(f_1,f_2,\dots,f_{k+1})$ in ${C}(\textup{Box})$, we write $\mathfrak{f}_{j+1}$ for the concave majorant of $f_{j+1}$ in $R_{j+1,\pm}$. Let 
$$\bar{R}_{j+1,\pm}\coloneqq [\pm (k+1-j)T-5^{-1}T,\pm (k+1-j)T+5^{-1}T].$$ Define 
\begin{equation*}
\begin{split}
I_{j+1,-}&\coloneqq\{x\in \bar{R}_{j+1,-} \ |\ (\mathfrak{f}_{j+1})'_+(x)\leq (k+1-j)T\},\\
I_{j+1,+}&\coloneqq\{x\in \bar{R}_{j+1,+} \ |\ (\mathfrak{f}_{j+1})'_-(x)\geq -(k+1-j)T\}.
\end{split}
\end{equation*}
and
\begin{align*}
\mathbf{l}_j \coloneqq\left\{ \begin{array}{cc}
\inf  I_{j+1,-} &\ \textup{if}\ I_{j+1,-}\neq \phi,\\
(-(k+1-j)+5^{-1})T   &\ \textup{if}\ I_{j+1,-}= \phi,
\end{array} \right.
\end{align*}
\begin{align*}
\mathbf{r}_j \coloneqq\left\{ \begin{array}{cc}
\sup  I_{j+1,+} &\ \textup{if}\ I_{j+1,+}\neq \phi,\\
 ( (k+1-j)-5^{-1})T   &\ \textup{if}\ I_{j+1,+}= \phi,
\end{array} \right.
\end{align*}

\begin{remark}\label{rmk:ljrj}
Suppose $I_{j+1,-}$ is non-empty and $I_{j+1,-}\subset\textup{int}\bar{R}_{j+1,-}$, where $\textup{int}$ means taking the interior. Then $\mathbf{l}_j$ is an extreme point of $\mathfrak{f}_{j+1}$. Therefore, $\mathfrak{f}_{j+1}(\mathbf{l}_j)={f}_{j+1}(\mathbf{l}_j)$ and $(\mathfrak{f}_{j+1})'_+(\mathbf{l}_j)\leq (k+1-j)T$. See Lemma~\ref{lem:contact} and Lemma \ref{lem:onesidecontinuous}. By the concavity of $\mathfrak{f}_{j+1}$, it holds that for $x\in R_{j+1,-}$ and $x\geq \mathbf{l}_j$,
\begin{align*}
f_{j+1}(x)\leq  {f}_{j+1}(\mathbf{l}_j)+(k+1-j)T(x-\mathbf{l}_j).
\end{align*}
See Figure \ref{fig:ellj} for an illustration.
\end{remark}
Putting $\mathbf{l}_j$ and $\mathbf{r}_j$ together, we define 
\begin{align*}
 (\bar{\mathbf{l}},\bar{\mathbf{r}} )\coloneqq& (\mathbf{l}_1,\mathbf{l}_2\dots,\mathbf{l}_k,\mathbf{r}_1,\mathbf{r}_2\dots,\mathbf{r}_k).
\end{align*}
\begin{remark} \label{rmk:semiconti}
To make sure $(\bar{\mathfrak{l}},\bar{\mathfrak{r}})$ defined in \eqref{def:fraklr} is a random variable, we check $(\bar{\mathbf{l}} ,\bar{\mathbf{r}} )$ is Borel measurable. The map from $f_{\textup{B}}$ to $(\mathfrak{f}_2,\dots,\mathfrak{f}_{k+1})$ is continuous under the uniform topology. From the view of Lemma~\ref{lem:Alsc},  $(\bar{\mathbf{l}} ,\bar{\mathbf{r}} )$ is semi-continuous and hence Borel measurable.
\end{remark}

\begin{remark}\label{rmk:U}
For any continuous function $f_{\textup{B}}$, it holds that $(\bar{\mathbf{l}} ,\bar{\mathbf{r}} )\in U$. Actually, it always holds that $|\mathbf{l}_j+(k+1-j)T|\leq 5^{-1}T$ and $|\mathbf{r}_j-(k+1-j)T|\leq 5^{-1}T$. 
\end{remark}

\begin{lemma}\label{lem:stop}
$(\bar{\mathfrak{l}},\bar{\mathfrak{r}})$ is a stopping domain which take values in $U$.
\end{lemma}

\begin{proof}
From Remark \ref{rmk:U}, $(\bar{\mathfrak{l}},\bar{\mathfrak{r}})$ takes value in $U$. From the construction, $(\bar{\mathbf{l}},\bar{\mathbf{r}})$ depends only on $f_{j+1}$ in $R_{j+1,\pm}$. From Remark \eqref{rmk:U}, it holds that $\{j+1\}\times R_{j+1,\pm}\subset \textup{Jct}(\bar{\mathbf{l}},\bar{\mathbf{r}})$. Because $\textup{Jct}(\bar{\ell},\bar{r})\subset \textup{ext}(\bar{\ell},\bar{r})$, we conclude that $(\bar{\mathfrak{l}},\bar{\mathfrak{r}})$ is a stopping domain. 
\end{proof} 

Now we are ready to define  $\mathcal{G}$.
\begin{definition}\label{def:G}
$\mathcal{G}$ is defined to be the collections of $(\bar{\ell},\bar{r},f_{\textup{J}})\in{C}_{\textup{Jct}}$ such that the following conditions hold.
\begin{description}[style=multiline, labelwidth=1cm]
\item[\namedlabel{itm:req1}{C1}] There exists $f_{\textup{B}}\in{C}(\textup{Box})$ with $f_{\textup{B}}|_{\textup{Jct}(\bar{\ell},\bar{r})}=f_{\textup{J}}$ such that $(\bar{\ell},\bar{r})=(\bar{\mathbf{l}},\bar{\mathbf{r}})(f_{\textup{B}})$.
    \item[\namedlabel{itm:req2}{C2}] $|f_j(x)+2^{-1}x^2|\leq 2^{-2}10^{-2}T^2$ for all $(j,x)\in \textup{Jct}(\bar{\ell},\bar{r}).$
    \item[\namedlabel{itm:req3}{C3}]  $ |f_j(x)-f_j(y)|\leq \Delta $ for all $|x-y|\leq d$ and $(j,x), (j,y)\in \textup{Jct}(\bar{\ell},\bar{r}). $   
    \item[\namedlabel{itm:req5}{C4}]  $f_j(x) \geq f_{j+1}(x)-3\Delta$ for all $(j,x), (j+1,x)\in \textup{Jct}(\bar{\ell},\bar{r}). $
\end{description}
\end{definition}

Condition \textbf{C1} actually holds automatically for the triple $(\bar{\mathfrak{l}},\bar{\mathfrak{r}},\mathfrak{H}|_{\textup{Jct}(\bar{\mathfrak{l}},\bar{\mathfrak{r}})})$. We include it still as we need to make use of it later. Condition \textbf{C2} requires each layer in $f_{\textup{J}}$ to be close to the parabola $-2^{-1}x^2$, which further ensures that $f_{j+1}$ is bounded above by a linear function as in Figure \ref{fig:ellj}. See the discussion in Remark \ref{rmk:ljrj}. Conditions \textbf{C3} controls the local fluctuation of $f_{\textup{J}}$ within distance $d$. Condition \textbf{C4} ensures that layers in $f_{\textup{J}}$ may go out of order by an amount of at most $3\Delta$.

\begin{lemma}\label{lem:measG}
$\mathcal{G}$ is a Borel subset of ${C}_{\textup{Jct}}$.
\end{lemma}
\begin{proof}
Let $\pi^{-1}(\mathcal{G} )$ be the pre-image of $\mathcal{G}$ in $U\times {C}( \textup{Box} )$. $\pi^{-1}(\mathcal{G} )$ is the collection of triples $(\bar{\ell},\bar{r},f_{\textup{B}})$ that satisfy Conditions \textbf{C2}-\textbf{C4} and $(\bar{\ell},\bar{r})=(\bar{\mathbf{l}},\bar{\mathbf{r}})(f_{\textup{B}})$. Conditions \textbf{C2}-\textbf{C4} induce a closed subset of $U\times {C}( \textup{Box} )$. Together with Remark \ref{rmk:semiconti}, $\pi^{-1}(\mathcal{G} )$ is a Borel subset of $U\times {C}(\textup{Box})$. Hence $\mathcal{G}$ is a Borel subset of ${C}_{\textup{Jct}}$.
\end{proof}

\begin{proposition}\label{pro:Fav2}
There exist $E_1(k,s)$ and $D=D(k,s)$ such that the following statement holds. For all $t\geq 1$ and $E\geq E_1(k,s)$, we have
\begin{align*}
\mathbb{P}(\mathsf{Fav}^{\textup{c}})\leq De^{-D^{-1}T^3}.
\end{align*}
\end{proposition}

We are ready to prove Proposition \ref{clm:Fav}.
\begin{proof}[{\bf Proof of Proposition \ref{clm:Fav}}]
If $\varepsilon\geq e^{-s^3}$, the assertion holds easily by taking $D$ large enough. Assume $\varepsilon<e^{-s^3}$. In particular, $T=E\left(\log\varepsilon^{-1} \right)^{1/3}$. By taking $E$ large enough, we obtain from Proposition \ref{pro:Fav2} that
\begin{align*}
\mathbb{P}(\mathsf{Fav}^{\textup{c}})\leq D \varepsilon.
\end{align*}
\end{proof}

To prove Proposition \ref{pro:Fav2}, we record \cite[Proposition 3.1, Corollary 3.2 and Corollary B.3]{Wu21} here respectively.

\begin{proposition}\label{pro:Z_lowerbound_k}
 There exists a constant $D=D(k)$ such that the following statement holds. For all $t\geq 1$, $L\geq 1$ and $K\geq L^2$, we have
\begin{align*} 
\mathbb{P} \left( \mathbb{E}_{\free}^{1,k,(-L,L), \vec{x}, \vec{y}}[ W^{1,k,(-L,L), \vec{x}, \vec{y},+\infty,\mathfrak{H}^t_{k+1}}_{\mathbf{H}_t}]<  D^{-1}  e^{- DL^{-1}K^2}  \right)< e^{-K^{3/2}},
\end{align*} 
where $\vec{x}=\big(\mathfrak{H}^t_i(-L)\big)_{i=1}^k$ and $\vec{y}=\big(\mathfrak{H}^t_i(L)\big)_{i=1}^k$.
\end{proposition}

\begin{proposition}\label{cor:osc}
 There exists a constant $D=D(k)$ such that the following statement holds. For all $t\geq 1$, $L\geq 1$, $d_0\in (0,1]$ and $K\geq 4Ld_0^{1/2}$, we have 
\begin{align*}
\mathbb{P}\left( \sup_{x,y\in [-L,L],\ |x-y|\leq d_0}\left|\mathfrak{H}_k^t(x)-\mathfrak{H}_k^t(y) \right|\geq K d_0^{1/2} \right)\leq  d_0^{-1}L D  e^{-D^{-1} K^{3/2} }.
\end{align*} 
\end{proposition}

\begin{proposition}\label{cor:tail}
There exists a constant $D=D(k)$ such that the following statement holds. For all $t\geq 1$, intervals $I\subset \mathbb{R}$ and $K\geq 0$, we have
\begin{align*} 
\mathbb{P}\left( \sup_{x\in I}\left|\mathfrak{H}^t_k(x)+2^{-1} {x^2}   \right|\geq K \right)\leq  (|I|/2+1)D e^{-D^{-1}K^{3/2}}.
\end{align*}
\end{proposition}

Proposition~\ref{pro:Z_lowerbound_k} bounds the normalizing constant for $\mathfrak{H}$ and is used to obtain Condition \textbf{C4}. Proposition~\ref{cor:osc} bounds the local fluctuation of $\mathfrak{H}$ and is used to ensure Condition \textbf{C3}. Proposition~\ref{cor:tail} implies $\mathfrak{H}$ are close to a parabola and guarantees Condition \textbf{C2}.

We would like to briefly explain how the parameters are chosen. The numbers $\Delta$ and $d$ are chosen to ensure that
\begin{align}\label{dD1}
\exp(-d\, \mathbf{H}_t(\Delta))<<e^{-T^3}
\end{align} 
and that
\begin{align}\label{dD2}
d^{-1/2}\, \Delta=T^2. 
\end{align}
\eqref{dD1} implies that it is costly for $\mathfrak{H}^t_j$ to always stay below $\mathfrak{H}^t_{j+1}-\Delta$ for an interval of length $d$. \eqref{dD2} controls the local fluctuation of $\mathfrak{H}^t_j$ within distance $d$. More precisely, its modulus of continuity over an interval of length $d$ is bounded above by $\Delta$, with probability $1-e^{-T^3}$, see Proposition \ref{cor:osc}. We will employ \eqref{dD1} and \eqref{dD2} to show that $\mathfrak{H}^t_j\geq \mathfrak{H}^t_{j+1}-3\Delta$ with high probability.

\begin{proof}[{\bf Proof of Proposition \ref{pro:Fav2}}]
Let $ \mathcal{G}_2,\mathcal{G}_3$ and $\mathcal{G}_4 $ be the subsets of ${C}(\textup{Box})$ such that \textbf{C2}, \textbf{C3} and \textbf{C4} hold in $\textup{Box}$ respectively. To be explicit,
\begin{align*}
\mathcal{G}_2\coloneqq &\{f_{\textup{B}}\in{C}(\textup{Box})\,|\, |f_j(x)+2^{-1}x^2|\leq 2^{-2}10^{-2}T^2\ \textup{ for all}\  (j,x)\in \textup{Box} \},\\
\mathcal{G}_3\coloneqq &\{f_{\textup{B}}\in{C}(\textup{Box})\,|\, |f_j(x)-f_j(y)|\leq \Delta \textup{ for all}\  (j,x), (j,y)\in \textup{Box}\ \textup{with}\ |x-y|\leq d \},\\\mathcal{G}_4\coloneqq &\{f_{\textup{B}}\in{C}(\textup{Box})\,|\, f_j(x) \geq f_{j+1}(x)-3\Delta\ \textup{ for all}\  (j,x), (j+1,x)\in \textup{Box} \}.
\end{align*}
For $i=2,3,4$, denote by $  \mathsf{Fav}_i$ the event $\{\mathfrak{H}|_{\textup{Box}}\in \mathcal{G}_i\}$. From the definition of $(\bar{\mathfrak{l}},\bar{\mathfrak{r}})$, condition \textbf{C1} holds automatically. Therefore $$  \bigcap_{i=2}^4\mathsf{Fav}_i \subset \mathsf{Fav}.$$
In the following, we bound $\mathbb{P}(\mathsf{Fav}_i^{\textup{c}})$.\\

To bound $\mathbb{P}(\mathsf{Fav}_2^{\textup{c}})$, we apply Proposition \ref{cor:tail}. Set $ I =[\ell_0,r_0]$ and $K=2^{-2}10^{-2}T^2$, we deduce that for all $j\in [1,k+1]_{\mathbb{Z}}$,
\begin{align*}
\mathbb{P}\left(\sup_{x\in [\ell_0,r_0]}|\mathfrak{H}^t_j+2^{-1}x^2|> 2^{-2}10^{-2}T^2\right)\leq ((r_0-\ell_0)/2+1)De^{-D^{-1}2^{-3}10^{-3}T^2}. 
\end{align*} 
From $r_0-\ell_0=2(k+1)T$ and $T\geq 10$, we deduce that\\[-0.25cm]
\begin{align*}
\mathbb{P}(\mathsf{Fav}_2^{\textup{c}})\leq De^{-D^{-1}T^3}
\end{align*}

To bound $\mathbb{P}(\mathsf{Fav}_3^{\textup{c}})$, we apply Proposition \ref{cor:osc}. Set $L=(k+1)T$, $K =T^2$ and $d_0=d$. We verify that $K \geq 4Ld_0^{1/2}$ as long as $T$ is large enough. We deduce that for all $j\in[1,k+1]_{\mathbb{Z}}$,\\[-0.25cm]
\begin{align*}
\mathbb{P}\left( \sup_{x,y\in [\ell_0,r_0],\ |x-y|\leq d}|\mathfrak{H}^t_j(x)-\mathfrak{H}^t_j(y)|\geq \Delta \right)\leq d_0^{-1}(2k+2)TD e^{-D^{-1}T^3}
\end{align*}

Together with $T\geq 1 $,
\begin{align*}
\mathbb{P}(\mathsf{Fav}_3^{\textup{c}})\leq De^{-D^{-1}T^3}. 
\end{align*}

Lastly, we bound $\mathbb{P}(\mathsf{Fav}_4^{\textup{c}})$. Let $\bar{\ell}^*=(\ell_0,\ell_0,\dots, \ell_0)$ and $\bar{r}^*=(r_0,r_0,\dots, r_0)$. Apply Proposition~\ref{pro:Z_lowerbound_k}. Setting $L=(k+1)T$ and $K=(k+1)^2T^2$, we have\\[-0.25cm]
\begin{align*}
\mathbb{P}\left(\mathbb{E}_{\free}^{\textup{dom}(\bar{\ell}^*,\bar{r}^*),f^*} \left[W_{\mathbf{H}_t}^{\textup{dom}(\bar{\ell}^*,\bar{r}^*),f^*}\right]<D^{-1}e^{-D(k+1)^3 T^3}\right)\leq  e^{-  (k+1)^3T^3}.
\end{align*}
Here $f^*=\mathfrak{H}^t|_{\textup{bdd}(\bar{\ell}^*,\bar{r}^*)}$. Denote this event by $\mathsf{Z}^{\textup{c}}$. Note that $\mathsf{Fav}_3\cap\mathsf{Fav}_4^{\textup{c}}$ implies there exists an interval with length $d$ in which $\mathfrak{H}_j<\mathfrak{H}_{j+1}-\Delta$ for some $j\in [1,k]_{\mathbb{Z}}$. Because $d\mathbf{H}_t(\Delta)\geq 64T^4(\log T)^2$, we have
$$\mathsf{Fav}_3\cap\mathsf{Fav}_4^{\textup{c}}\subset \left\{W_{\mathbf{H}_t}^{\textup{dom}(\bar{\ell}^*,\bar{r}^*),f^*}(\mathfrak{H}^t|_{\textup{dom}(\bar{\ell}^*,\bar{r}^*)})\leq e^{-64T^4(\log T)^2}\right\}.$$   

We then deduce from the $\mathbf{H}_t$-Brownian Gibbs property that\\[-0.3cm]
\begin{align*}
\mathbb{P}(\mathsf{Fav}_3\cap\mathsf{Fav}_4^{\textup{c}})\leq &\mathbb{P}(\mathsf{Fav}_3\cap\mathsf{Fav}_4^{\textup{c}}\cap\mathsf{Z} )+\mathbb{P}(\mathsf{Z}^{\textup{c}})\\
=&\mathbb{E}[\mathbbm{1}_{\mathsf{Z}}\cdot \mathbb{E}[\mathbbm{1}_{\mathsf{Fav}_3}\cdot \mathbbm{1}_{\mathsf{Fav}_4^{\textup{c}}}\,|\,\mathcal{F}_{\textup{ext}}(\bar{\ell}^*,\bar{r}^*) ] ]+\mathbb{P}(\mathsf{Z}^{\textup{c}})\\
\leq &De^{-64T^4(\log T)^2+D(k+1)^3 T^3}+e^{-(k+1)^3T^3}\\
\leq &De^{-D^{-1}T^3}.
\end{align*}
Therefore,
\begin{align*}
\mathbb{P}(\mathsf{Fav}_4^{\textup{c}})\leq \mathbb{P}(\mathsf{Fav}_3\cap\mathsf{Fav}_4^{\textup{c}})+\mathbb{P}( \mathsf{Fav}_3^{\textup{c}})\leq De^{-D^{-1}T^3}.
\end{align*}
In conclusion, 
\begin{align*}
\mathbb{P}(\mathsf{Fav}^{\textup{c}})\leq \sum_{i=2}^4\mathbb{P}(\mathsf{Fav}_i^{\textup{c}})\leq &De^{-D^{-1}T^3}.
\end{align*}
\end{proof}

\section{Curve separation over a finite set}
The goal of this section is to prove Theorem \ref{thm:separation}. Theorem \ref{thm:separation} shows that, in loose terms, conditioned on jumping over a Lipschitz function at finite many locations, the Brownian bridges may remain separated provided they are separated on the boundary. Theorem \ref{thm:separation} will be used in Section~\ref{sec:denominator} to separate the curves in the jump ensemble $J$ over the resampling domain. The details in the proof will not be needed in the remaining part of this paper. Skipping those will not affect the reader's ability to follow the rest of the arguments. \\
 
Fix $k\in\mathbb{N}$ and $\ell<r$. Let $a^+= (a^+_1,a_2^+,\dots, a_k^+)$ and $a^-= (a^-_1,a_2^-,\dots, a_k^-)$ be two vectors in $\mathbb{R}^k$. Recall that $\mathbb{P}^{1,k,(\ell,r),a^+,a^-}_{\free}$ is the law of $k$ independent Brownian bridges defined on $ [\ell,r]$ with boundary values given by $a^{\pm}$. Let $g\in{C}([\ell,r]) $ be a Lipschitz function with $g(\ell)=g(r)=0$. $g(x)$ will serve as the lower boundary. Let $P \subset (\ell,r )$ be a finite subset. Let $\mathbb{P}_{\mathcal{L}}$ be the law of $\mathbb{P}^{1,k,(\ell,r),a^+,a^-}_{\free}$ conditioned on jumping over $g(x)$ at pole points in $P$. In terms of the Radon-Nikodym derivative relation, we have
\begin{align*}
\frac{{d}\mathbb{P}_{\mathcal{L}}}{{d}\mathbb{P}^{1,k,(\ell,r),a^+,a^-}_{\free}}(\mathcal{L})\coloneqq \frac{W(\mathcal{L} )}{\mathbb{E}^{1,k,(\ell,r),a^+,a^-}_{\free}[W   ]}.
\end{align*}
The Boltzmann weight $W$ is defined through \\[-0.25cm] 
\begin{align*}
W(\mathcal{L} )\coloneqq \left\{ \begin{array}{cc}
1 & \mathcal{L}_j(p)> g(p)\ \textup{for all}\  p\in P\ \textup{and}\ j\in [1,k]_\mathbb{Z},\\
0 & \textup{otherwise}.
\end{array}\right.
\end{align*}

Next, we introduce two assumptions \eqref{assumptionC} and \eqref{assumptionC2}. We assume \eqref{assumptionC} and \eqref{assumptionC2} hold throughout this section. Fix $T\geq 10$,
\begin{align}\label{assumption_C}
b_0,b_1 ,\lambda_0, \lambda_1>0,
\end{align}
and
\begin{align}\label{assumption_C2}
b_2>0.
\end{align}
We assume that
\begin{equation}\label{assumptionC}
\begin{split}
&r-\ell\leq b_0 T,\ |P|\leq b_0T,\ \sup_{x,y\in [\ell,r],\ x\neq y}\frac{|g(x)-g(y)|}{|x-y|}\leq b_1 T,\\
&a^\pm_{j}-a^\pm_{j+1}\geq \lambda_0 T^{1/2}\ \textup{for all}\ j\in [1,k-1]_{\mathbb{Z}} \ \\ &a_k^{-}-g(\ell)\geq \lambda_1 T \textup{and}\ a_k^{+}-g(r)\geq \lambda_1 T. 
\end{split}
\end{equation}
We also assume that
\begin{equation}\label{assumptionC2}
a_{1}^--g(\ell)\leq b_2 T^2\ \textup{and}\  a_{1}^+-g(r)\leq b_2 T^2.
\end{equation}
The assumption \eqref{assumptionC} bounds the distant between $\ell$ and $r$, the number of elements in $P$, the Lipschitz constant of $g$ and the gaps in $a^\pm$. The assumption \eqref{assumptionC2} further gives an upper bound for $a_1^\pm$. \\

For $\mu\in (0,1)$, Consider the event
\begin{equation*}
\begin{split}
\mathsf{H}\coloneqq &\{  \mathcal{L}_j(p)-\mathcal{L}_{j+1}(p)\geq \mu \lambda_0 T^{1/2}\ \textup{for all}\  p\in P\ \textup{and}\ j\in [1,k-1]_\mathbb{Z} \}\\
&\cap \{ \mathcal{L}_{k}(p)-g(p)\geq  \mu \lambda_1 T\ \textup{for all}\  p\in P \}.
\end{split}
\end{equation*}

\begin{theorem}\label{thm:separation}
There exists a constant $C$ depending on constants in \eqref{assumption_C}, \eqref{assumption_C2} and $\mu$ such that
\begin{align*}
\mathbb{P}_{\mathcal{L}}(\mathsf{H})\geq C^{-1}e^{-CT^{5/2}}.
\end{align*}
\end{theorem}
The rest of the section is devoted to prove Theorem \ref{thm:separation}. In subsection~\ref{sec:C1}, we first prove Theorem \ref{thm:separation} for a special case. The general case is proved in subsection \ref{sec:C2} up to Lemma~\ref{lem:middle}. Lemma~\ref{lem:middle} is then proved in subsection~\ref{sec:C3}.  
\subsection{A special case}\label{sec:C1}
In this subsection, we prove Theorem \ref{thm:separation} for the special case $$P\subset (\ell,\ell+T^{1/2})\cup (r-T^{1/2},r).$$ 
 We begin by considering the situation that $P\subset (\ell,\ell+T^{1/2})$.
\begin{lemma}\label{lem:1side_free}
Assume that $P\subset (\ell,\ell+T^{1/2})$. Then there exists a constant $C$ depending on constants in \eqref{assumption_C} and $\mu$ such that
\begin{align*}
\mathbb{P}^{1,k,(\ell,r),a^+,a^-}_{\free}(\mathsf{H})\geq C^{-1}e^{-CT^{5/2}}.
\end{align*}
\end{lemma}
From Lemma \ref{lem:1side_free}, we immediately get the following two corollaries. Note that by the time reversal symmetry of the Brownian motion, the same results in Corollary~\ref{cor:1side_W} and Corollary~\ref{cor:1side} also hold if $P\subset (r-T^{1/2},r)$.
\begin{corollary}\label{cor:1side_W}
Assume that $P\subset (\ell,\ell+T^{1/2})$. Then there exists a constant $C$ depending on constants in \eqref{assumption_C} and $\mu$ such that
\begin{align*}
\mathbb{E}^{1,k,(\ell,r),a^+,a^-}_{\free}[W ]\geq C^{-1}e^{-CT^{5/2}}.
\end{align*}
\end{corollary}
\begin{proof}
This follows easily from Lemma \ref{lem:1side_free} because $ \mathsf{H} \subset\{W =1\}$.
\end{proof}
\begin{corollary}\label{cor:1side}
Assume that $P\subset (\ell,\ell+T^{1/2})$. Then there exists a constant $C$ depending on constants in \eqref{assumption_C} and $\mu$ such that
\begin{align*}
\mathbb{P}_{\mathcal{L}}(\mathsf{H})\geq C^{-1}e^{-CT^{5/2}}.
\end{align*}
\end{corollary}
\begin{proof}
By $W\leq 1 $ and $\mathsf{H} \subset\{W =1\} $, we have\\[-0.25cm]
\begin{align*}
\mathbb{P}_{\mathcal{L}}(\mathsf{H}) =\frac{\mathbb{E}^{1,k,(\ell,r),a^+,a^-}_{\free }[\mathbbm{1}_{\mathsf{H}}\cdot W ]}{\mathbb{E}^{1,k,(\ell,r),a^+,a^-}_{\free }[W] }=\frac{\mathbb{E}^{1,k,(\ell,r),a^+,a^-}_{\free }[\mathbbm{1}_{\mathsf{H}} ]}{\mathbb{E}^{1,k,(\ell,r),a^+,a^-}_{\free }[W] }\geq \mathbb{P}^{1,k,(\ell,r),a^+,a^-}_{\free }(\mathsf{H}) .
\end{align*}
Then the assertion follows from Lemma \ref{lem:1side_free}.
\end{proof}

\begin{proof}[{\bf Proof of Lemma \ref{lem:1side_free}}]
Note that without loss of generality, we may assume $g(\ell)=g(r)=0$. The reason is that the law of Brownian bridges are invariant under affine transformations. The event $\mathsf{H}$ is also unaffected by affine transformations. The only thing differs is the Lipschitz bound of $g(x)$. When we replace $g(x)$ by $g(x)-\frac{x-\ell}{r-\ell}g(r)-\frac{r-x}{r-\ell}g(\ell)$, the Lipschitz bound of $g(x)$ may be doubled. Therefore, it is sufficient to prove the assertion for the case $g(\ell)=g(r)=0$. From now on we assume $g(\ell)=g(r)=0$.\\

The proof involves direct computations of Brownian bridges. We separate the discussion into two cases depending on the length of the interval $[\ell,r]$.\\

\noindent\underline{Case 1: $r-\ell\geq 2T^{1/2}$.} Let $q=\ell+T^{1/2}$. Under the law $\mathbb{P}^{1,k,(\ell,r),a^+,a^-}_{\free }$, $B_j(q)$ are Gaussian random variables with mean and variance
\begin{align*}
m_j=\frac{(r-q)a_j^-+(q-\ell)a^+_j}{r-\ell},\ \sigma^2= \frac{(q-\ell)(r-q)}{r-\ell}.
\end{align*}
For $x\in [\ell,q]$, we have
\begin{align*}
B_j(x)=\frac{(q-x)a_j^-+(x-\ell)B_j(q)}{q-\ell}+B_j^{[\ell,q]}(x). 
\end{align*}
In particular,
\begin{align*}
B_k(x)-g(x)\geq &\left(\frac{(q-x)a_k^-+(x-\ell)B_k(q)}{q-\ell}+B_k^{[\ell,q]}(x) \right)-  (x-\ell)b_1T  \\
= &\frac{(q-x)a_k^-+(x-\ell)m_k  }{q-\ell}+\frac{(x-\ell) (B_k(q)-m_k-b_1 T^{3/2})}{q-\ell}+B_k^{[\ell,q]}(x)\\
\geq & \lambda_1 T-  |B_k(q)-(m_k+b_1T^{3/2})|-\sup_{x\in [\ell,q]}|B_k^{[\ell,q]}(x)|.
\end{align*}
We have used the assumption \eqref{assumptionC} in the above inequalities.\\

Similarly, for $x\in [\ell,q]$ and $j\in [1,k-1]_{\mathbb{Z}}$, we have
\begin{align*}
B_j(x)-B_{j+1}(x)\geq  & \lambda_0 T^{1/2} -  |B_j(q)-(m_j+b_1T^{3/2})|-\sup_{x\in [\ell,q]}|B_j^{[\ell,q]}(x)|\\
&-  |B_{j+1}(q)-(m_{j+1}+b_1T^{3/2})|-\sup_{x\in [\ell,q]}|B_{j+1}^{[\ell,q]}(x)|.
\end{align*}
We have used the assumption \eqref{assumptionC} in the above inequality.\\

In order to ensure $\mathsf{H}$, it is sufficient to require that for all $j\in [1,k]_{\mathbb{Z}}$, it holds that\\[-0.2cm]
$$|B_j(q)-(m_j+b_1T^{3/2})|\leq 4^{-1}(1-\mu)\lambda T^{1/2},\ \sup_{x\in [\ell,q]}|B_j^{[\ell,q]}(x)|\leq 4^{-1}(1-\mu) \lambda T^{1/2},$$ 
 where $\lambda=\min\{\lambda_0,\lambda_1\}$. Let $N$ be a Gaussian random variable with mean $0$ and variance $1$. It holds that
\begin{align*}
&\mathbb{P}^{1,k,(\ell,r),a^+,a^-}_{\free}\bigg(|B_j(q)-(m_j+b_1T^{3/2})|\leq 4^{-1}(1-\mu)\lambda T^{1/2}\bigg)\\
=& \mathbb{P}\bigg(|N-\sigma^{-1}b_1T^{3/2}|\leq 4^{-1}\sigma^{-1}(1-\mu)\lambda T^{1/2}\bigg).
\end{align*}
Since $r-\ell\geq 2T^{1/2}$ and $q=\ell+T^{1/2}$, $   \sigma^2\in  [2^{-1}T^{1/2}, T^{1/2})$. It holds that
\begin{align*}
\sigma^{-1}b_1T^{3/2}\leq &2^{1/2}b_1T^{5/4},
\end{align*}
and
\begin{align*}
4^{-1}\sigma^{-1}(1-\mu)\lambda T^{1/2}\geq &4^{-1} (1-\mu)\lambda T^{1/4}.
\end{align*}
As a result,
\begin{align*}
&\mathbb{P}(|N-\sigma^{-1}b_1T^{3/2}|\leq 4^{-1}\sigma^{-1}(1-\mu)\lambda T^{1/2}).\\
 \geq & \mathbb{P}(|N-2^{1/2}b_1T^{5/4} |\leq 4^{-1} (1-\mu)\lambda T^{1/4} )\\
 \geq &C^{-1}e^{-CT^{5/2}}.
\end{align*}

Next, we bound the event $\displaystyle\sup_{x\in [\ell,q]}|B_j^{[\ell,q]}(x)|\leq 4^{-1}(1-\mu) \lambda T^{1/2}$. Recall that $q-\ell=T^{1/2}$. Through a direct computation,
\begin{align*}
 &\mathbb{P}^{1,k,(\ell,r),a^+,a^-}_{\free}( \sup_{x\in [\ell,q]}|B_j^{[\ell,q]}(x)|\leq 4^{-1}(1-\mu) \lambda T^{1/2}) \\
= &\mathbb{P}(\sup_{x\in [0,1]}|B^{[0,1]}(x)|\leq 4^{-1}(1-\mu) \lambda T^{1/4})\geq C^{-1}. 
\end{align*}

Combining the above estimates, we conclude that  $$\mathbb{P}^{1,k,(\ell,r),a^+,a^-}_{\free}(\mathsf{H})\geq C^{-1}e^{-CT^{5/2}}. $$

\noindent\underline{Case 2: $r-\ell<2T^{1/2}$.} Set $q=2^{-1}(\ell+r)$. Under the law $\mathbb{P}^{1,k,(\ell,r),a^+,a^-}_{\free}$, $B_j(q)$ are Gaussian random variables with mean and variance
\begin{align*}
m_j=2^{-1}(a^-_j+a^+_j),\ \sigma^2=4^{-1}(r-\ell).
\end{align*} 

For $x\in [\ell,q]$ and $y\in [q,r]$, it holds that
\begin{align*}
B_k(x)-g(x)\geq \lambda_1 T-  |B_k(q)-(m_k+2^{-1}(r-\ell)b_1T )|-\sup_{x\in [\ell,q]}|B_k^{[\ell,q]}(x)|,\\
B_k(y)-g(y)\geq \lambda_1 T-  |B_k(q)-(m_k+2^{-1}(r-\ell)b_1T )|-\sup_{y\in [ q,r]}|B_k^{[q,r]}(y)|.
\end{align*}
We have used the assumption \eqref{assumptionC} in the above inequality.\\

Similarly, for $x\in [\ell,q]$, $y\in [q,r]$ and $ j\in [1,k-1]_{\mathbb{Z}}$, it holds that
\begin{align*}
B_j(x)-B_{j+1}(x)\geq  & \lambda_0 T^{1/2} -  |B_j(q)-(m_j+2^{-1}(r-\ell)b_1T )|-\sup_{x\in [\ell,q]}|B_j^{[\ell,q]}(x)|\\
&-  |B_{j+1}(q)-(m_{j+1}+2^{-1}(r-\ell)b_1T )|-\sup_{x\in [\ell,q]}|B_{j+1}^{[\ell,q]}(x)|.\\
B_j(y)-B_{j+1}(y)\geq  & \lambda_0 T^{1/2} -  |B_j(q)-(m_j+2^{-1}(r-\ell)b_1T )|-\sup_{y\in [q,r]}|B_j^{[q,r]}(y)|\\
&-  |B_{j+1}(q)-(m_{j+1}+2^{-1}(r-\ell)b_1T )|-\sup_{y\in [q,r]}|B_{j+1}^{[q,r]}(y)|.
\end{align*}
We have used the assumption \eqref{assumptionC} in the above inequality.\\

In order to ensure $\mathsf{H}$, it is sufficient to require that for $j\in [1,k]_{\mathbb{Z}}$, it holds that $$|B_j(q)-(m_j+2^{-1}(r-\ell)b_1T )|\leq 4^{-1}(1-\mu)\lambda T^{1/2},$$ 
$$\sup_{x\in [\ell,q]}|B_j^{[\ell,q]}(x)|\leq 4^{-1}(1-\mu) \lambda T^{1/2},$$
and
$$ \sup_{x\in [q,r]}|B_j^{[q,r]}(x)|\leq 4^{-1}(1-\mu) \lambda T^{1/2},$$
 where $\lambda=\min\{\lambda_0,\lambda_1\}$. Let $N$ be a Gaussian random variable with mean $0$ and variance $1$. It holds that
\begin{align*}
&\mathbb{P}^{1,k,(\ell,r),a^+,a^-}_{\free}(|B_j(q)-(m_j+2^{-1}(r-\ell)b_1T )|\leq 4^{-1}(1-\mu)\lambda T^{1/2})\\
 =& \mathbb{P}(|N-(r-\ell)^{1/2}b_1T |\leq 2^{-1}(1-\mu)\lambda (r-\ell)^{-1/2}T^{1/2}).
\end{align*}
Because $r-\ell <2T^{1/2}$, it holds that $$(r-\ell)^{1/2}b_1T< 2^{1/2}b_1T^{5/4}$$ and $$  2^{-1}(1-\mu)\lambda (r-\ell)^{-1/2}T^{1/2}> 2^{-3/2}(1-\mu)\lambda T^{1/4}.$$ Therefore,
\begin{align*}
& \mathbb{P}(|N-(r-\ell)^{1/2}b_1T |\leq 2^{-1}(1-\mu)\lambda (r-\ell)^{-1/2}T^{1/2})\\
 \geq & \mathbb{P}(|N-2^{1/2}b_1T^{5/4} |\leq 2^{-3/2}(1-\mu)\lambda T^{1/4} )\\
 \geq &C^{-1}e^{-CT^{5/2}}.
\end{align*}

Next, we bound the event $ \sup_{x\in [\ell,q]}|B_j^{[\ell,q]}(x)|\leq 4^{-1}(1-\mu) \lambda T^{1/2}$. By a direct computation,
\begin{align*}
&\mathbb{P}^{1,k,(\ell,r),a^+,a^-}_{\free} (\sup_{x\in [\ell,q]}|B_j^{[\ell,q]}(x)|\leq 4^{-1}(1-\mu) \lambda T^{1/2})\\
=&\mathbb{P}(\sup_{x\in [0,1]}|B^{[0,1]}(x)|\leq 2^{-3/2}(r-\ell)^{-1/2}(1-\mu) \lambda T^{1/2})\\
\geq &\mathbb{P}(\sup_{x\in [0,1]}|B^{[0,1]}(x)|\leq 4^{-1} (1-\mu) \lambda T^{1/4})\geq C^{-1}.
\end{align*}
We have used $r-\ell< 2T^{1/2}$. Similarly,
\begin{align*}
&\mathbb{P}^{1,k,(\ell,r),a^+,a^-}_{\free}(\sup_{x\in [q,r]}|B_j^{[q,r]}(x)|\leq 2^{-2}(1-\mu) \lambda T^{1/2})\geq C^{-1}.
\end{align*}

Combining the above estimates, we conclude
$$\mathbb{P}^{1,k,(\ell,r),a^+,a^-}_{\free}(\mathsf{H})\geq C^{-1}e^{-CT^{5/2} } .$$ 
\end{proof}
We are ready to prove Theorem~\ref{thm:separation} for a special case $P\subset (\ell,\ell+T^{1/2})\cup (r-T^{1/2},r)$. 
\begin{lemma}\label{lem:2side}
Assume that $P\subset (\ell,\ell+T^{1/2})\cup (r-T^{1/2},r)$. Then there exists a constant $C$ depending on constants in \eqref{assumption_C} and $\mu$ such that
\begin{align*}
\mathbb{P}_{\mathcal{L}}(\mathsf{H})\geq C^{-1}e^{-CT^{5/2}}.
\end{align*}
\end{lemma}
\begin{proof} The idea is to decompose the Boltzmann weight into $W=\widetilde{W}\cdot \hat{W}$. Roughly speaking, $\widetilde{W}$ and $\hat{W}$ activate the interaction at the pole points in $(r-T^{1/2},r)$ and $(\ell,\ell+T^{1/2})$ respectively. Then we apply Corollary \ref{cor:1side} twice to finish the argument.\\

Let $\tilde{P} =P\cap (r-T^{1/2},r)$ and $\hat{P}=P\setminus \tilde{P}$. If $\tilde{P}$ is empty, the assertion of Lemma~\ref{lem:2side} follows from Corollary \ref{cor:1side}. From now on we assume $\tilde{P}$ is non-empty.\\

Let $\mathbb{P}_{\tilde{\mathcal{L}}}$ be the law of $\mathbb{P}^{1,k,(\ell,r),a^+,a^-}_{\free}$ conditioned on jumping over $g(x)$ at pole points in $\tilde{P}$. In terms of the Radon-Nikodym derivative relation, it holds that
 \begin{align*}
 \frac{{d} \mathbb{P}_{\tilde{\mathcal{L}}} }{{d}\mathbb{P}^{1,k,(\ell,r),a^+,a^-}_{\free} }(\mathcal{L})=\frac{\widetilde{W}(\mathcal{L})}{ \mathbb{E}^{1,k,(\ell,r),a^+,a^-}_{\free}[\widetilde{W}] } .
\end{align*}
Here the weight $\widetilde{W}$ is defined by
\begin{align*}
\widetilde{W}(\mathcal{L}) \coloneqq \left\{ \begin{array}{cc}
1, & \mathcal{L}_j(p)> g(p)\ \textup{for all}\  p\in \tilde{P}\ \textup{and}\ j\in [1,k]_\mathbb{Z},\\
0, & \textup{otherwise}.
\end{array}\right.
\end{align*}

Correspondingly, we define the remaining weight
\begin{align*}
\hat{W} (\mathcal{L})\coloneqq \left\{ \begin{array}{cc}
1, & \mathcal{L}_j(p)> g(p)\ \textup{for all}\  p\in \hat{P}\ \textup{and}\ j\in [1,k]_\mathbb{Z},\\
0, & \textup{otherwise}.
\end{array}\right.
\end{align*}

Since $W=\widetilde{W}\cdot \hat{W} $, we have
 \begin{align*}
\frac{{d} \mathbb{P}_{  \mathcal{L} } }{{d} \mathbb{P}_{\tilde{\mathcal{L}}} }(\mathcal{L}) =\frac{\hat{W}(\mathcal{L})}{ \mathbb{E}_{\tilde{\mathcal{L}}}[\hat{W}] } .
\end{align*}

Consider the following events $\tilde{\mathsf{G}}$ and $\hat{\mathsf{G}}$,
\begin{equation*}
\begin{split}
\tilde{\mathsf{G}}\coloneqq &\{  \mathcal{L}_j(p)-\mathcal{L}_{j+1}(p)\geq 2^{-1}(1+\mu)  \lambda_0 T^{1/2}\ \textup{for all}\  p\in \tilde{P}\ \textup{and}\ j\in [1,k-1]_\mathbb{Z} \}\\
&\cap \{ \mathcal{L}_{k}(p)-g(p)\geq  2^{-1}(1+\mu) \lambda_1 T\ \textup{for all}\  p\in \tilde{P}  \},
\end{split}
\end{equation*}
and
\begin{equation*}
\begin{split}
\hat{\mathsf{G}}\coloneqq &\{  \mathcal{L}_j(p)-\mathcal{L}_{j+1}(p)\geq  \mu \lambda_0 T^{1/2}\ \textup{for all}\  p\in \hat{P}\ \textup{and}\ j\in [1,k-1]_\mathbb{Z} \}\\
&\cap \{ \mathcal{L}_{k}(p)-g(p)\geq     \mu \lambda_1 T\ \textup{for all}\  p\in \hat{P}  \}.
\end{split}
\end{equation*}
Roughly speaking, $\tilde{\mathsf{G}}$ and $\hat{\mathsf{G}}$ require conditions in $\mathsf{H}$ to hold in $\tilde{P}$ and $\hat{P}$ respectively. The requirements in $\tilde{\mathsf{G}}$ is stronger than the ones in $\mathsf{H}$. The reason is that we will apply Corollary \ref{cor:1side} and Corollary \ref{cor:1side_W} to $\hat{P}$ conditioned on the occurrence of $\mathsf{\tilde{G}}$. \\

Because $\hat{W}\leq 1$, we have
\begin{align*}
&\mathbb{P}_{\mathcal{L}}(\mathsf{H})\geq \mathbb{P}_{\mathcal{L}}(\mathsf{H}\cap\tilde{\mathsf{G}})\\
=&\mathbb{P}_{\mathcal{L}}(\mathsf{H} | \tilde{\mathsf{G}}  )\times \mathbb{P}_{\mathcal{L}}( \tilde{\mathsf{G}} ) =\mathbb{P}_{\mathcal{L}}(\mathsf{H} | \tilde{\mathsf{G}}  )\times \frac{\mathbb{E}_{\tilde{\mathcal{L}}}[\mathbbm{1}_{ \tilde{\mathsf{G}} }\cdot \hat{W}]}{\mathbb{E}_{\tilde{\mathcal{L}}}[  \hat{W}]}\\
\geq& \mathbb{P}_{\mathcal{L}}(\mathsf{H} | \tilde{\mathsf{G}}  )\times  \mathbb{E}_{\tilde{\mathcal{L}}}[\mathbbm{1}_{ \tilde{\mathsf{G}} }\cdot \hat{W}]  \\
= &\mathbb{P}_{\mathcal{L}}(\mathsf{H} | \tilde{\mathsf{G}}  )\times \mathbb{P}_{\tilde{\mathcal{L}}}(  \tilde{\mathsf{G}} ) \times \mathbb{E}_{\tilde{\mathcal{L}}}[ \hat{W} |\tilde{\mathsf{G}} ].   
\end{align*}

Since $\tilde{P}\subset (r-T^{1/2},r)$, we can apply Corollary \ref{cor:1side} to obtain $$\mathbb{P}_{\tilde{\mathcal{L}}}( \tilde{\mathsf{G}} )\geq C^{-1}e^{-CT^{5/2}}.$$ 
 
Next, we turn to $\mathbb{P}_{\mathcal{L}}(\mathsf{H} | \tilde{\mathsf{G}}   )$.  Conditioned on a realization of $\tilde{\mathsf{G}}$, the law of $\mathbb{P}_{\mathcal{L}}$ restricted to $[\ell, r-T^{1/2}]$ is given by the law of free Brownian bridges conditioned on jumping over $g$ at those pole points in $\hat{P}$. Therefore, we can apply Corollary \ref{cor:1side} to bound $\mathbb{P}_{\mathcal{L}}(\mathsf{H} | \tilde{\mathsf{G}})$.

We proceed to establish estimates that are uniform in any realizations of $\tilde{\mathsf{G}}$. To simplify the notation, we fix a realization of $\tilde{\mathsf{G}}$. That is, we fix $\mathcal{L}_j(p)$ for $j\in [1,k]_{\mathbb{Z}}$ and $p\in\tilde{P}$ which satisfy the requirements in $\tilde{\mathsf{G}}$. We denote the event of such realization by $\tilde{\mathsf{G}}_0$. Because $ \hat{\mathsf{G}}\cap \tilde{\mathsf{G}}_0 \subset \mathsf{H}$, it holds that $ \mathbb{P}_{\mathcal{L}}(\mathsf{H} | \tilde{\mathsf{G}}_0  )\geq  \mathbb{P}_{\mathcal{L}}(\hat{\mathsf{G}} | \tilde{\mathsf{G}}_0  )  $.

In order to apply Corollary~\ref{cor:1side}, we define the following quantities. Let $p_0$ be the smallest element of $\tilde{P}$. Define 
\begin{align*}
\hat{\ell}\coloneqq \ell,\ \hat{r}\coloneqq p_0,\ \hat{a}^-_j\coloneqq a^-_j,\ \hat{a}^+_j\coloneqq \mathcal{L}_j(p_0).
\end{align*}
Conditioned on $\tilde{\mathsf{G}}_0$, the law of ${\mathcal{L}}|_{[\hat{\ell},\hat{r}]}$, denoted by $\hat{\mathbb{P}}_{\mathcal{L}}$, is the law of $\mathbb{P}_{\free}^{1,k,(\hat{\ell},\hat{r}),\hat{a}^+,\hat{a}^-}$ conditioned on jumping over $ {g}(x)$ at pole points in $\hat{P}$. In terms of the Radon-Nikodym derivative relation, it holds that
\begin{align*}
\frac{d \hat{\mathbb{P}}_{\mathcal{L}} }{ d \mathbb{P}_{\free}^{1,k,(\hat{\ell},\hat{r}),\hat{a}^+,\hat{a}^-}}(\mathcal{L})=\frac{\hat{W}(\mathcal{L})}{\mathbb{E}_{\free}^{1,k,(\hat{\ell},\hat{r}),\hat{a}^+,\hat{a}^-}[\hat{W}]} .\ 
\end{align*}

To apply Corollary \ref{cor:1side}, we check the assumption \eqref{assumptionC}. $\tilde{\mathsf{G}}_0 $ implies that
\begin{equation*} 
\begin{split}
&\hat{a}^\pm_{j}-\hat{a}^\pm_{j+1}\geq 2^{-1}(1+\mu)\lambda_0 T^{1/2}\ \textup{for all}\ j\in [1,k-1]_{\mathbb{Z}},\\
& \hat{a}_k^{-}-g(\hat{\ell})\geq 2^{-1}(1+\mu)\lambda_1 T,\ \textup{and}\ \hat{a}_k^{+}-g(\hat{r})\geq 2^{-1}(1+\mu)\lambda_1 T.\\
\end{split}
\end{equation*}
Therefore, the assumption \eqref{assumptionC} holds for these data with slightly different bounds. By Corollary \ref{cor:1side}, it holds that $\mathbb{P}_{\mathcal{L}}(\hat{\mathsf{G}} | \tilde{\mathsf{G}}_0  ) \geq C^{-1}e^{-CT^{5/2}}.$ This ensures 
$$\mathbb{P}_{\mathcal{L}}(\mathsf{H} | \tilde{\mathsf{G}}_0  )\geq  \mathbb{P}_{\mathcal{L}}(\hat{\mathsf{G}} | \tilde{\mathsf{G}}_0  ) \geq C^{-1}e^{-CT^{5/2}}.$$
Because this holds for any realization $\tilde{\mathsf{G}}_0$ of $\mathsf{\tilde{G}}$, we obtain
$$\mathbb{P}_{\mathcal{L}}(\mathsf{H} | \tilde{\mathsf{G}} )\geq C^{-1}e^{-CT^{5/2}}.$$ 

Next, we turn to $\mathbb{E}_{\tilde{\mathcal{L}}}[ \hat{W}|\mathsf{G}_0]$. Conditioned on $\tilde{\mathsf{G}}_0$, The law of $\tilde{\mathcal{L}}|_{[\hat{\ell},\hat{r}]}$ is identical to $\mathbb{P}_{\free}^{1,k,(\hat{\ell},\hat{r}),\hat{a}^+,\hat{a}^-}$. The assumption \eqref{assumptionC} is already checked above. Applying Corollary \ref{cor:1side_W}, it holds that
$$\mathbb{E}_{\tilde{\mathcal{L}}}[ \hat{W}|\mathsf{G}_0]=\mathbb{E}_{\free}^{1,k,(\hat{\ell},\hat{r}),\hat{a}^+,\hat{a}^-}[ \hat{W}] \geq C^{-1}e^{-CT^{5/2}}.$$ 
Because this holds for any realization $\tilde{\mathsf{G}}_0$ of $\mathsf{\tilde{G}}$, we obtain
$$\mathbb{E}_{\tilde{\mathcal{L}}}[ \hat{W}|\tilde{\mathsf{G}}]\geq  C^{-1}e^{-CT^{5/2}}.$$

Putting the above estimates together, we conclude that
$$\mathbb{P}_{\mathcal{L}}(\mathsf{H})\geq C^{-1}e^{-CT^{5/2}}. $$
\end{proof}
\subsection{The general case.}\label{sec:C2} In this subsection, we consider the general case of $P$ and prove Theorem \ref{thm:separation}. The approach is similar to the one in proving Lemma \ref{lem:2side}. We decompose the Boltzmann weight into $W= {W}'\cdot \check{W}$. The weights $ {W}'$ and $\check{W}$ activate the interaction at the pole points in $[\ell+T^{1/2},r-T^{1/2}]$ and $(\ell, \ell+T^{1/2})\cup (r-T^{1/2},r)$ respectively. Lemma \ref{lem:2side} is then applicable to $\check{W}.$\\

Let $ {P}'=P\cap [\ell+T^{1/2},r-T^{1/2}]$ and $\check{P} =P\setminus P'$.\ Suppose $ {P}'$ is empty, then Theorem \ref{thm:separation} follows from Lemma \ref{lem:2side}. From now on we assume $P' $ is non-empty. 

Let $ \mathbb{P}_{\mathcal{L}}'$ be the law of $\mathbb{P}^{1,k,(\ell,r),a^+,a^-}_{\free}$ conditioned on jumping over $g(x)$ at pole points in ${P}'$. In terms of the Radon-Nikodym derivative relation, it holds that
\begin{align*}
\frac{{d}\mathbb{P}_{ \mathcal{L}'}}{{d}\mathbb{P}^{1,k,(\ell,r),a^+,a^-}_{\free} }(\mathcal{L})=  \frac{{  W'} (\mathcal{L})}{ \mathbb{E}^{1,k,(\ell,r),a^+,a^-}_{\free}[{  W'} ] } .
\end{align*}
Here the weight $W'$ is defined by
\begin{align*}
 {W}' (\mathcal{L}) &\coloneqq\left\{ \begin{array}{cc}
1, & \mathcal{L}_j(p)> g(p)\ \textup{for all}\  p\in  P'\ \textup{and}\ j\in [1,k]_\mathbb{Z},\\
0, & \textup{otherwise}
\end{array}\right.
\end{align*}

Correspondingly, we define
\begin{align*}
\check{W}(\mathcal{L}) &\coloneqq\left\{ \begin{array}{cc}
1, & \mathcal{L}_j(p)> g(p)\ \textup{for all}\  p\in  \check{P}\ \textup{and}\ j\in [1,k]_\mathbb{Z},\\
0, & \textup{otherwise}
\end{array}\right.
\end{align*}
Since $ {W}'\cdot \check{W}=W$, we have 
\begin{align*}
\frac{{d}\mathbb{P}_{ \mathcal{L}}}{{d}\mathbb{P}_{ \mathcal{L}'} }(\mathcal{L}) =\frac{\check{W}(\mathcal{L})}{ \mathbb{E}_{  \mathcal{L }'}[\check{W}] }   .
\end{align*}

We further define $\mathcal{L}'_+$ to be ordered at pole points in $P'$. Let $ \mathbb{P}_{\mathcal{L}_+}'$ be the law of $\mathbb{P}_{\mathcal{L}'}$ conditioned being ordered at pole points in ${P}'$. In terms of the Radon-Nikodym derivative relation, it holds that 
\begin{align*}
\frac{{d}\mathbb{P}_{ \mathcal{L}'_+}}{{d}\mathbb{P}^{1,k,(\ell,r),a^+,a^-}_{\free} }(\mathcal{L})=  \frac{{  W'_+} (\mathcal{L})}{ \mathbb{E}^{1,k,(\ell,r),a^+,a^-}_{\free}[{  W'_+} ] } .
\end{align*}
Here the weight $W'_+$ is defined by
\begin{align*}
 {W}'_+ (\mathcal{L}) &\coloneqq\left\{ \begin{array}{cc}
1, & \mathcal{L}_k(p)> g(p),\,\ \mathcal{L}_{j+1}(p)> \mathcal{L}_{j}(p) \ \textup{for all}\  p\in  P'\ \textup{and}\ j\in [1,k]_\mathbb{Z},\\
0, & \textup{otherwise}
\end{array}\right. 
\end{align*}
Since $W'_+=W'\cdot W'_+$, we have
\begin{align*}
\frac{{d}\mathbb{P}_{ \mathcal{L}'_+}}{{d}\mathbb{P}_{ \mathcal{L}'} }(\mathcal{L}) =\frac{{W}'_+ (\mathcal{L})}{ \mathbb{E}_{  \mathcal{L }'}[ {W}'_+] } .
\end{align*}

\begin{lemma}\label{lem:order}
There exists a constant $C$ depending on constants in \eqref{assumption_C} such that
\begin{align*}
\mathbb{E}_{ {\mathcal{L}'}}[ {W}'_+]\geq C^{-1}e^{-CT}.
\end{align*}
\end{lemma}
\begin{proof}
If $a^\pm_j=a^\pm_{j+1}$ for $j\in [1,k-1]_{\mathbb{Z}}$, $\{\mathcal{L}'_j(p), j\in[1,k]_{\mathbb{Z}}\, \}$ are identically distributed and it follows that
\begin{align*}
\mathbb{P}_{ {\mathcal{L}'}}(\mathcal{L}'_j(p)>\mathcal{L}'_{j+1}(p)\ \textup{for}\ j\in [1,k-1]_{\mathbb{Z}})= \frac{1}{k!}.
\end{align*}
Since $a^\pm_j>a^\pm_{j+1}$ for $j\in [1,k-1]_{\mathbb{Z}}$, we have for any $p\in P'$,
\begin{align*}
\mathbb{P}_{ {\mathcal{L}'}}(\mathcal{L}'_j(p)>\mathcal{L}'_{j+1}(p)\ \textup{for}\ j\in [1,k-1]_{\mathbb{Z}})\geq \frac{1}{k!}.
\end{align*}
Because $|P'|\leq b_0T$, we inductively get
\begin{align*}
\mathbb{E}_{ {\mathcal{L}'}}[ {W}'_+]\geq  (k!)^{-b_0T}.
\end{align*}
\end{proof}

Consider the event $\mathsf{G}$,
\begin{equation*}
\begin{split}
\mathsf{G}\coloneqq  &\{  \mathcal{L}_j(p)-\mathcal{L}_{j+1}(p)\geq    2^{-1}(1+\mu) \lambda_0 T^{1/2}\ \textup{for all}\  p\in   P' \ \textup{and}\ j\in [1,k-1]_\mathbb{Z} \}\\
&\cap \{ \mathcal{L}_{k}(p)-g(p)\geq    2^{-1}(1+\mu)  \lambda_1 T\ \textup{for all}\  p\in P' \}.
\end{split}
\end{equation*}
Compared to $\mathsf{H}$, the event $\mathsf{G}$ requires a stronger gap at pole points in $P'$. 

\begin{lemma}\label{lem:middle}
There exists a constant $C$ depending on constants in \eqref{assumption_C}, \eqref{assumption_C2} and $\mu$ such that
\begin{align*}
\mathbb{P}_{ \mathcal{L}'_+}(   \mathsf{G}  )\geq C^{-1}e^{-CT^{5/2}}.
\end{align*}
\end{lemma}

We postpone the proof of Lemma \ref{lem:middle} and prove Theorem \ref{thm:separation} below.
\begin{proof}[{\bf Proof of Theorem \ref{thm:separation}}]
Because $\check{W} \leq 1$, we have
\begin{align*}
&\mathbb{P}_{\mathcal{L}}(\mathsf{H})\geq\mathbb{P}_{\mathcal{L}}(\mathsf{H}\cap\mathsf{G}) \\
=&\mathbb{P}_{\mathcal{L}}(\mathsf{H} |\mathsf{G}  )\times \mathbb{P}_{\mathcal{L}}(\mathsf{G} ) =\mathbb{P}_{\mathcal{L}}(\mathsf{H} |\mathsf{G}  )\times \frac{\mathbb{E}_{ \mathcal{L}'}[\mathbbm{1}_{\mathsf{G} }\cdot  \check{W}]}{\mathbb{E}_{ \mathcal{L}' }[  \check{W}]}\\
\geq &\mathbb{P}_{\mathcal{L}}(\mathsf{H} |\mathsf{G}  )\times  \mathbb{E}_{ \mathcal{L}' }[\mathbbm{1}_{\mathsf{G} }\cdot   \check{W}]\\
  = &\mathbb{P}_{\mathcal{L}}(\mathsf{H} | \mathsf{G}  )\times \mathbb{P}_{  \mathcal{L}'}( \mathsf{G} ) \times \mathbb{E}_{ \mathcal{L}' }[    \check{W}|\mathsf{G}].   
\end{align*}

$\mathbb{P}_{\mathcal{L}}(\mathsf{H} | \mathsf{G}  )$ and $\mathbb{E}_{ {\mathcal{L}}'}[ \check{W} |  \mathsf{G} ]$ could be bounded similarly as in the proof of Lemma \ref{lem:2side}. Conditioned on a realization of $\mathsf{G}$, the law $\mathbb{P}_{\mathcal{L}' }$ behaves like the law of free Brownian bridges. Furthermore, the law $\mathbb{P}_{ {\mathcal{L}}}$ is given by the law of those free Brownian bridges conditioned on jumping over $g$ at pole points in $\check{P}$. Therefore, we can apply Corollary \ref{cor:1side_W} and Corollary \ref{cor:1side} to bound $\mathbb{E}_{\tilde{\mathcal{L}}}[ \check{W} | \mathsf{G} ]$ and $\mathbb{P}_{\mathcal{L}}(\mathsf{H} | \mathsf{G})$ respectively. From Corollary \ref{cor:1side}, we have  $\mathbb{P}_{\mathcal{L}}(\mathsf{H} | \mathsf{G}  )\geq C^{-1}e^{-CT^{5/2}} .$ From Corollary \ref{cor:1side_W}, we have  $\mathbb{E}_{ \mathcal{L}' }[\check{W}|\mathsf{G}] \geq C^{-1}e^{-CT^{5/2}}. $  We omit the detail here and refer the readers to the proof of Lemma \ref{lem:2side} for a similar argument.

It suffices to show $\mathbb{P}_{ \mathcal{L}'}( \mathsf{G})\geq C^{-1}e^{-CT^{5/2}}$. Note that $W'_+=1$ when $\mathsf{G}$ holds. Therefore, we have 
\begin{align*}
\mathbb{P}_{ \mathcal{L}'}( \mathsf{G} )=\mathbb{E}_{ \mathcal{L}'}[\mathbbm{1}_{\mathsf{G}}\cdot  W'_+]=\frac{\mathbb{E}_{ \mathcal{L}'}[\mathbbm{1}_{\mathsf{G}}\cdot  W'_+]}{\mathbb{E}_{ \mathcal{L}'}[ W'_+]}\times \mathbb{E}_{ \mathcal{L}'}[ W'_+]= \mathbb{P}_{ \mathcal{L}'_+}(\mathsf{G})\times \mathbb{E}_{ \mathcal{L}'}[ W'_+]
\end{align*}
Lemmas \ref{lem:order} and \ref{lem:middle} implies that there exists $C>0$ such that
\begin{align*}
\mathbb{P}_{ \mathcal{L}'}( \mathsf{G} )\geq  C^{-1}e^{-CT^{5/2}}.
\end{align*}

Putting the above estimates together, we conclude that
$$\mathbb{P}_{\mathcal{L}}(\mathsf{H})\geq C^{-1}e^{-CT^{5/2}}.$$
\end{proof}
\subsection{Proof of Lemma~\ref{lem:middle}}\label{sec:C3} 
The proof of Lemma \ref{lem:middle} is based on a modification of an approached of Hammond. The key idea is the following simple calculation of normal distribution. Let $N$ be a Gaussian random variable with mean $0$ and variance $\sigma^2$ and let $M_1>>M_2>>\sigma$. The probability that $N\geq M_1+M_2$, conditioned on $N\geq M_1$ is given by
\begin{align*}
\mathbb{P}(N\geq M_1+M_2\,|\,N\geq M_1 )\approx e^{-\sigma^{-2}M_1M_2}.
\end{align*}

In the current context, the curve $\mathcal{L}'_+$ behaves like $N$ conditioned on greater than $T^2$. This is because the largest possible value of $g(x)$ is of order $T^2$.  The variance is of order $ T^{1/2}$ because $P'\subset (\ell+T^{1/2},r-T^{1/2})$. The event $\mathsf{G}$ behaves like $N\geq T^2+T$ because it  requires $\mathcal{L}_k\geq g+T$. Therefore, we obtain 
\begin{align*}
\mathbb{P}_{\mathcal{L}'_+}(\mathsf{G})\approx e^{-T^{5/2}}.
\end{align*}

We carry out the above heuristic argument in the rest of this subsection. We may assume without loss of generality that $g(\ell)=g(r)=0$. See the discussion at the beginning of the proof of Lemma~\ref{lem:1side_free}. \\

\begin{lemma}\label{lem:W+_free}
There exists a constant $C_1$ depending on constants in \eqref{assumption_C} such that
\begin{align*}
\mathbb{E}^{1,k,(\ell,r),a^+,a^-}_{\free}[ {W}'_+]\geq C_1^{-1}e^{-C_1T^3}.
\end{align*}
\end{lemma}
\begin{proof}  
The proof involves direct computations of Brownian bridges and is similar to the one for Lemma~Lemma \ref{lem:1side_free}.\\

Set $q=2^{-1}(\ell+r)$. Under the law $\mathbb{P}^{1,k,(\ell,r),a^+,a^-}_{\free}$, $B_j(q)$ are normal distributions with mean and variance
\begin{align*}
m_j=2^{-1}(a^-_j+a^+_j),\ \sigma^2=4^{-1}(r-\ell).
\end{align*} 
For all $x\in [\ell,q]$ and $y\in [q,r]$, we have
\begin{align*}
B_k(x)-g(x)\geq \lambda_1 T-  |B_k(q)-(m_k+2^{-1}(r-\ell)b_1T )|-\sup_{x\in [\ell,q]}|B_k^{[\ell,q]}(x)|,\\
B_k(y)-g(y)\geq \lambda_1 T-  |B_k(q)-(m_k+2^{-1}(r-\ell)b_1T )|-\sup_{y\in [ q,r]}|B_k^{[q,r]}(y)|.
\end{align*}
Similarly, for $1\leq j\leq k-1$, $x\in [\ell,q]$ and $y\in [q,r]$, we have
\begin{align*}
B_j(x)-B_{j+1}(x)\geq  & \lambda_0 T^{1/2} -  |B_j(q)-(m_j+2^{-1}(r-\ell)b_1T )|-\sup_{x\in [\ell,q]}|B_j^{[\ell,q]}(x)|\\
&-  |B_{j+1}(q)-(m_{j+1}+2^{-1}(r-\ell)b_1T )|-\sup_{x\in [\ell,q]}|B_{j+1}^{[\ell,q]}(x)|.\\
B_j(y)-B_{j+1}(y)\geq  & \lambda_0 T^{1/2} -  |B_j(q)-(m_j+2^{-1}(r-\ell)b_1T )|-\sup_{y\in [q,r]}|B_j^{[q,r]}(y)|\\
&-  |B_{j+1}(q)-(m_{j+1}+2^{-1}(r-\ell)b_1T )|-\sup_{y\in [q,r]}|B_{j+1}^{[q,r]}(y)|.
\end{align*}
Therefore, to ensure for all $x\in [\ell,r]$, 
$$B_1(x)> B_2(x)> \dots > B_k(x)> g(x),$$
it suffices to require that for all $j\in [1,k]_{\mathbb{Z}}$, 
$$|B_j(q)-(m_j+2^{-1}(r-\ell)b_1T )|< 4^{-1} \lambda T^{1/2},$$ 
and 
$$\sup_{x\in [\ell,q]}|B_j^{[\ell,q]}(x)|\leq 4^{-1}  \lambda T^{1/2},\ \sup_{x\in [q,r]}|B_j^{[q,r]}(x)|\leq 4^{-1}  \lambda T^{1/2}.$$
Here $\lambda=\min\{\lambda_0,\lambda_1\}$. Let $N$ be a normal distribution with mean $0$ and variance $1$.
\begin{align*}
&\mathbb{P}^{1,k,(\ell,r),a^+,a^-}_{\free}(|B_j(q)-(m_j+2^{-1}(r-\ell)b_1T )|< 4^{-1} \lambda T^{1/2})\\
 =& \mathbb{P}(|N-(r-\ell)^{1/2}b_1T |< 2^{-1} \lambda (r-\ell)^{-1/2}T^{1/2}).
\end{align*}
Because $r-\ell \leq b_0T$, $$(r-\ell)^{1/2}b_1T\leq   b_0^{1/2}b_1T^{3/2}$$ and $$ 2^{-1} \lambda (r-\ell)^{-1/2}T^{1/2}\geq  2^{-1}\lambda b_0^{-1/2} .$$ 
Thus
\begin{align*}
&\mathbb{P}^{1,k,(\ell,r),a^+,a^-}_{\free} (|B_j(q)-(m_j+2^{-1}(r-\ell)b_1T )|<  4^{-1}(1-\mu)\lambda T^{1/2})\\
\geq&  \mathbb{P}(|N-b_0^{1/2}b_1T^{3/2} |< 2^{-1}\lambda b_0^{-1/2})\geq  C^{-1}e^{-CT^{3} }.
\end{align*}

Next, we bound the event $\displaystyle \sup_{x\in [\ell,q]}|B_j^{[\ell,q]}(x)|\leq 4^{-1}  \lambda T^{1/2}$. It holds that
\begin{align*}
&\mathbb{P}^{1,k,(\ell,r),a^+,a^-}_{\free}(\sup_{x\in [\ell,q]}|B_j^{[\ell,q]}(x)|\leq 4^{-1}  \lambda T^{1/2})\\
=&\mathbb{P}(\sup_{x\in [0,1]}|B^{[0,1]}(x)|\leq 2^{-3/2}(r-\ell)^{-1/2}  \lambda T^{1/2})\\
\geq &\mathbb{P}(\sup_{x\in [0,1]}|B^{[0,1]}(x)|\leq 2^{-3/2}\lambda b_0^{-1/2}  )\geq C^{-1}.
\end{align*}
Similarly,
\begin{align*}
\mathbb{P}_{\free}(\sup_{x\in [q,r]}|B_j^{[q,r]}(x)|\leq 4^{-1}(1-\mu) \lambda T^{1/2})\geq C^{-1}.
\end{align*}

Combining the above estimates, we conclude that
\begin{align*}
\mathbb{E}^{1,k,(\ell,r),a^+,a^-}_{\free}[ {W}'_+]\geq C^{-1}e^{-CT^3}.
\end{align*}
The assertion follows by taking $C_1$ large enough.
\end{proof}

Let $m+1=|P'|$ be the number of elements in $P'$. We label them as $\{p_0<p_1<\dots< p_m\}.$ Let $C_2$ be a large number to be determined. Consider the event \\[-0.3cm]
\begin{align*}
\mathsf{Gap}\coloneqq \{ |\mathcal{L}_j(p_i)-\mathcal{L}_j(p_0)|\leq C_2T^2\ \textup{for all}\ j\in [1,k]_{\mathbb{Z}}\ \textup{and}\  i\in [1,m]_{\mathbb{Z}} \}.
\end{align*}

In the lemma below, we show that for $C_2$ large enough, it holds that $\mathbb{P}_{  \mathcal{L}'_+}(\mathsf{Gap})\geq 2^{-1}.$
\begin{lemma}\label{lem:A}
There exists a constant $C_2$ depending on constants in \eqref{assumption_C}, \eqref{assumption_C2} such that
\begin{align*}
\mathbb{P}_{  \mathcal{L}'_+}(\mathsf{Gap})\geq 2^{-1}.
\end{align*}
\end{lemma}
\begin{proof}
From the view of Lemma~\ref{lem:W+_free}, it suffices to show that for large enough $C_2$, we have $$\mathbb{P}^{1,k,(\ell,r),a^+,a^-}_{\free} (\mathsf{Gap}^{\textup{c}})\leq 2^{-1}C_1^{-1}e^{-C_1T^3},$$
where $C_1$ is the constant in Lemma~\ref{lem:W+_free}. From the assumptions \eqref{assumptionC} and \eqref{assumptionC2},
\begin{align*}
\left| B_j(p_i)-B_j(p_0)  \right|\leq &\left| B^{[\ell,r]}_j(p_i)-B^{[\ell,r]}_j(p_0)  \right|+\left| \frac{(p_i-p_0)(a_j^+-a_j^-)}{r-\ell} \right|\\
\leq & \left| B^{[\ell,r]}_j(p_i)-B^{[\ell,r]}_j(p_0)  \right|+b_2T^2. 
\end{align*}
Hence for given $j\in [1,k]_{\mathbb{Z}}$ and $i\in [1,m]_{\mathbb{Z}}$,
\begin{align*}
&\mathbb{P}^{1,k,(\ell,r),a^+,a^-}_\free (\left| B_j(p_i)-B_j(p_0)  \right|>C_2T^2)\\
\leq &\mathbb{P}^{1,k,(\ell,r),a^+,a^-}_{\free}\left(\left| B_j^{[\ell,r]}(p_i)-B_j^{[\ell,r]}(p_0)  \right|>(C_2-b_2)T^2\right)\\
\leq &\mathbb{P}^{1,k,(\ell,r),a^+,a^-}_{\free}\left( \sup_{x\in [\ell,r]}\left| B_j^{[\ell,r]}(x)   \right|>2^{-1}(C_2-b_2)T^2\right)\\
=&\mathbb{P} \left( \sup_{x\in [0,1]}\left| B^{[0,1]}(x)   \right|>2^{-1}(C_2-b_2)(r-\ell)^{-1/2}T^2\right)\\
\leq &\mathbb{P} \left( \sup_{x\in [0,1]}\left| B^{[0,1]}(x)   \right|>2^{-1} b_0^{-1/2}(C_2-b_2) T^{3/2}\right).
\end{align*}
By taking $C_2$ large enough, we can obtain
\begin{align*}
\mathbb{P}^{1,k,(\ell,r),a^+,a^-}_{\free}(\left| B_j(p_i)-B_j(p_0)  \right|>C_2T^2)\leq 2^{-1}k^{-1}m^{-1}C_1^{-1}e^{-C_1T^3}. 
\end{align*}

Taking union over $j\in [1,k]_{\mathbb{Z}}$ and $i\in [1,m]_{\mathbb{Z}}$, it holds that
$$\mathbb{P}^{1,k,(\ell,r),a^+,a^-}_{  \free}(\mathsf{Gap}^{\textup{c}})\leq 2^{-1}C_1^{-1}e^{-C_1T^3}.$$

As a result, we conclude that
\begin{align*}
\mathbb{P}_{  \mathcal{L}'_+}(\mathsf{Gap}^{\textup{c}})=\frac{  \mathbb{E}^{1,k,(\ell,r),a^+,a^-}_{  \free}[\mathbbm{1}_{\mathsf{Gap}^{\textup{c}}}\cdot W'_+]}{\mathbb{E}^{1,k,(\ell,r),a^+,a^-}_{  \free}[ W'_+] }\leq \frac{\mathbb{P}^{1,k,(\ell,r),a^+,a^-}_{  \free}(\mathsf{Gap}^{\textup{c}})}{\mathbb{E}^{1,k,(\ell,r),a^+,a^-}_{  \free}[ W'_+] } \leq 2^{-1}.
\end{align*}
 
\end{proof}
From now on we fix the choice of $C_2$. \\
 
We write $\mathbb{P}_0$ for the law of $\mathbb{P}^{1,k,(\ell,r),a^+,a^-}_{  \free} $ conditioned on $ \mathsf{Gap} $ and write $\mathbb{E}_0$ for the expectation of $\mathbb{P}_0$. To simplify the notation, we further take a realization of $\mathsf{Gap}$. That is, for $j\in [1,k]_{\mathbb{Z}}$ and $i\in [1,m]_{\mathbb{Z}}$, we fix $c_{j,i}\in\mathbb{R}$ with $|c_{j,i}|\leq C_2T^2$. We abuse the notation and still use $\mathbb{P}_0$ to denote the law of $\mathbb{P}^{1,k,(\ell,r),a^+,a^-}_{  \free} $ conditioned on
\begin{align*}
B_j(p_i)-B(p_0)=c_{j,i}.
\end{align*}
The estimates below will be uniform in any such choice.\\

Under the law of  $\mathbb{P}_{0}$, $B_j(p_0)$ are independent Gaussian random variables with mean and variance 
\begin{align*}
m_j=&\frac{a^-_j}{1+(r-p_m)^{-1}(p_0-\ell) }+\frac{a^+_j-c_{j,m}}{1+(p_0-\ell)^{-1}(r-p_m)},\\
\sigma_0^2=&\frac{1}{(p_0-\ell)^{-1}+(r-p_m)^{-1}}.
\end{align*}
To see this, by a direct computation, the p.d.f. of $B_j(p_0)$ is proportional to
\begin{align*}
\exp\left( -\frac{(x-a^-_j)^2}{2(p_0-\ell)}-\frac{(a^+_j-c_{j,m}-x)^2}{2(r-p_m)} \right)\propto  \exp\left( {-\frac{(x-m_j)^2}{2\sigma_0^2}}\right).
\end{align*}
Furthermore, $ {W}'_+=1$ if and only if
\begin{align*}
 B_{k}(p_0)&> \max_{0\leq i\leq m} \{g(p_i)-c_{k,i}\}=:h_k,\\
 B_{j}(p_0)-B_{j+1}(p_0)&> \max_{0\leq i\leq m} \{c_{j+1,i}-c_{j,i}\}=:h_j.
\end{align*}
Here we adapt the convention that $c_{j,0}=0$. Denote such event by $\mathsf{Q}$. Explicitly,
\begin{align*}
\mathsf{Q}\coloneqq \{B_{k}(p_0) >  h_k,\ \textup{and}\ 
 B_{j}(p_0)-B_{j+1}(p_0) >  h_j\ \textup{for all}\ j\in [1,k-1]_{\mathbb{Z}}\}.
\end{align*} 

We check that
\begin{align}\label{keyrelation}
\mathbb{P}_{\mathcal{L}'_+}(\mathsf{G} \,|\, \mathsf{Gap})=\mathbb{P}_{0}(\mathsf{G} \,|\, \mathsf{Q}).
\end{align}

\begin{lemma}\label{lem:Q} There exists a constant $C_3$  depending on constants in \eqref{assumption_C} and \eqref{assumption_C2} such that
$$\mathbb{P}_0(\mathsf{Q})\geq C_3^{-1}e^{-C_3\sigma_0^{-2} T^{4}}.$$
\end{lemma}
\begin{proof}
It can be checked directly that the event 
\begin{align*}
B_j(p_0)\in \sum_{i=j}^k |h_i|+2(k-j)\sigma_0+(0, \sigma_0]\ \textup{for all}\ j\in [1,k]_{\mathbb{Z}}
\end{align*}
is contained in $\mathsf{Q}$. Let $N$ be a Gaussian random variable with mean $0$ and variance $1$.
\begin{align*}
\mathbb{P}_0\left(B_j(p_0)\in \sum_{i=j}^k |h_i|+2(k-j)\sigma_0+(0, \sigma_0]\right)=\mathbb{P}\left(N\in \sigma_0^{-1}\left(\sum_{i=j}^k |h_i| -m_j\right)+2(k-j)+(0,1] \right).
\end{align*} 
Under the assumptions \eqref{assumptionC} and \eqref{assumptionC2}, for $j\in [1,k]_{\mathbb{Z}}$ it holds that 
$$ |h_j|\leq (2C_2+b_0b_1) T^2,$$  
$$|m_j|\leq (b_2+C_2)T^2.$$ 
Because $T^{-2} \leq  p_0-\ell,r-p_m\leq b_0T$, it holds that
$$2^{1/2} b^{-1/2}_0 T^{-1/2} \leq \sigma_0^{-1}\leq 2^{1/2}T^{-1/4}.$$ Therefore, 
\begin{align*}
\left|\sigma_0^{-1}\left(\sum_{i=j}^k |h_i| -m_j\right)+2(k-j)\right|\leq C\sigma_0^{-1} T^{2}. 
\end{align*}
And
\begin{align*}
\mathbb{P}_0\left(B_j(p_0)\in \sum_{i=j}^k |h_i|+2(k-j)\sigma_0+(0, \sigma_0]\right)\geq  \mathbb{P}\left(N\in C\sigma_0^{-1} T^{2} +(0,1] \right)\geq C^{-1}e^{-C\sigma_0^{-2}T^4}.
\end{align*}
Because $B_j(p_0)$ are independent under the law $\mathbb{P}_0$, we conclude that
\begin{align*}
\mathbb{P}_0(\mathsf{Q})\geq C^{-1}e^{-C\sigma_0^{-2}T^4}.
\end{align*}
The assertion follows by taking $C_3$ large enough.
\end{proof}

Let $C_4$ be a large constant to be determined. 
\begin{lemma}\label{lem:D3}
Assume that $C_4\geq 2(b_2+C_2)$. Then there exists a universal constant $C_5$ such that 
\begin{align*}
\mathbb{P}_0(B_1(p_0)>C_4T^2 )\leq C_5 e^{-C_5^{-1}C_4^2\sigma_0^{-2} T^{4}  }.
\end{align*}
\end{lemma}

\begin{proof}
Let $N$ be a Gaussian random variable with mean $0$ and variance $1$. It holds that
\begin{align*}
\mathbb{P}_0(B_1(p_0)>C_4T^2)=\mathbb{P}(N>\sigma_0^{-1}(C_4T^2-m_1)). 
\end{align*}
Using $|m_1|\leq (b_2+C_2)T^2$, we have 
\begin{align*}
\mathbb{P}_0(B_1(p_0)>C_4T^2)\leq &\mathbb{P}(N>\sigma_0^{-1}(C_4-C_2-b_2)T^{2})\\
\leq &\mathbb{P}(N>2^{-1}\sigma_0^{-1} C_4 T^{2}).
\end{align*}
Hence the assertion follows.
\end{proof}

\begin{lemma}\label{lem:D4} Assume that $C_4\geq 2(b_2+C_2)$. Then there exists a constant $C_6$  depending on constants in \eqref{assumption_C} such that
\begin{align*}
\mathbb{P}_0( \mathsf{G}^{\textup{c}}\cap\{B_1(p_0)\leq C_4T^2\}  | \mathsf{Q}  )\leq 1-C_6^{-1}e^{-C_6C_4T^{5/2}}.
\end{align*}
\end{lemma}
\begin{proof} Let $\mathsf{J}=\mathsf{G}^{\textup{c}}\cap \mathsf{Q}\cap\{B_1(p_0)\leq C_4T^2\} $ and $\Omega\subset \mathbb{R}^k$ be the set such that\\[-0.25cm]
$$\mathsf{J}=\{ (B_1(p_0),B_2(p_0),\dots, B_k(p_0))\in\Omega \}.$$ 

If $\Omega$ is of measure zero, the assertion is clearly true. From now on we assume $\Omega$ has positive measure.

Let $\alpha=\lambda_0 T^{1/2}$ and $\beta=\lambda_1 T$. Define the map $\Phi:\mathbb{R}^k\to\mathbb{R}^k$ by
\begin{align*}
\Phi_j(x_1,x_2\dots,x_k )\coloneqq x_j+\alpha (k-j)+\beta.
\end{align*} 
Define
\begin{align*}
\Phi(\mathsf{J})\coloneqq \{ (B_1(p_0),B_2(p_0),\dots, B_k(p_0))\in\Phi(\Omega) \}.
\end{align*}
It holds that $\Phi(\mathsf{J})\cup \mathsf{J}\subset \mathsf{Q}$ and $\Phi(\mathsf{J})\cap \mathsf{G}^{\textup{c}}=\phi$. Hence we have 
\begin{align*}
\mathbb{P}_0( \mathsf{G}^{\textup{c}}\cap\{B_1(p_0)\leq M\}  | \mathsf{Q}  )\leq \frac{\mathbb{P}_0( \mathsf{J} )}{\mathbb{P}_0( \mathsf{J} )+\mathbb{P}_0(\Phi( \mathsf{J}) )}.
\end{align*}

Recall that $m_j$ and $\sigma_0^2$ are the means and variance of $B_j(p_0)$ under the law $\mathbb{P}_0$. 
\begin{align*}
\mathbb{P}_0(\mathsf{J})=\int_{\Omega} \prod_{j=1}^k \rho _{\sigma_0^2, m_j} (x_j )\,  \prod_{j=1}^k{d} {x}_j.
\end{align*}
Because the Jacobian of $\Phi$ equals $1$,
\begin{align*}
\mathbb{P}_0(\Phi( \mathsf{J}))&=\int_{\Phi( \mathsf{\Omega})} \prod_{j=1}^k \rho _{\sigma_0^2, m_j} (x_j )\,  \prod_{j=1}^k{d} {x}_j =\int_{\Omega  } \prod_{j=1}^k \rho _{\sigma_0^2, m_j} (x_j+\alpha(k-j)+\beta )\, \prod_{j=1}^k{d} {x}_j.
\end{align*}
Thus
\begin{align*}
\frac{\mathbb{P}_0(\Phi( \mathsf{J}))}{\mathbb{P}_0(  \mathsf{J}) }\geq &\inf_{\Omega } \prod_{j=1}^k  \frac{\rho _{\sigma_0^2, m_j} (x_j+\alpha(k-j)+\beta ) }{\rho _{\sigma_0^2, m_j} (x_j ) }\\
= &\inf_{\Omega }\prod_{j=1}^k \exp\left( -\sigma_0^{-2}\left[ (x_j-m_j)(\alpha(k-j)+\beta)+2^{-1}(\alpha(k-j)+\beta)^2 \right] \right).
\end{align*}
For $(x_1,x_2,\dots, x_k)\in \Omega$ and $C_4\geq 2(b_2+C_2)$, we have $x_j-m_j\leq 2C_4T^2$. Together with $\sigma_0^{-2}\leq 2T^{-1/2}$, the above is bounded from below by $C^{-1}e^{-CC_4 T^{5/2}}$ for some large constant $C$. 
As a result, 
\begin{align*}
\mathbb{P}_0( \mathsf{G}^{\textup{c}}\cap\{B_1(p_0)\leq M\}  | \mathsf{Q}  )\leq \frac{1}{1+C^{-1}e^{-CC_4 T^{5/2}}}\leq 1-C^{-1}e^{-CC_4 T^{5/2}}.
\end{align*}
\end{proof}

We are ready to prove Lemma~\ref{lem:middle}.
\begin{proof}[{\bf Proof of Lemma \ref{lem:middle}}]
From Lemma \ref{lem:A},
\begin{align*}
\mathbb{P}_{ \mathcal{L}'_+}( \mathsf{G})\geq \mathbb{P}_{ \mathcal{L}'_+}( \mathsf{G}\cap \mathsf{Gap}) = \mathbb{P}_{ \mathcal{L}'_+}( \mathsf{G} |\mathsf{Gap})\times \mathbb{P}_{  \mathcal{L}'_+}( \mathsf{Gap})\geq 2^{-1}\mathbb{P}_{ \mathcal{L}'_+}( \mathsf{G} |\mathsf{Gap}).
\end{align*}
From \eqref{keyrelation}, Lemmas \ref{lem:Q}, \ref{lem:D3} and \ref{lem:D4}, we have
\begin{align*}
\mathbb{P}_{ \mathcal{L}'_+}( \mathsf{G} |\mathsf{Gap})= &\mathbb{P}_0(\mathsf{G}   | \mathsf{Q}  )\geq  \mathbb{P}_0(\mathsf{G} \cap\{B_1(p_0)\leq C_4T^2\}  | \mathsf{Q}  )\\
\geq & 1-\mathbb{P}_0( \mathsf{G}^{\textup{c}}\cap\{B_1(p_0)\leq C_4T^2\}  | \mathsf{Q}  )-\mathbb{P}_0(  B_1(p_0)> C_1T^2   | \mathsf{Q}  ) \\
\geq &1-(1-C_6 e^{-C_6^{-1} C_4T^{5/2}})-C_3C_5e^{-\sigma_0^2 T^4 (C_5^{-1}C_4^2-C_3) }.
\end{align*}
Requiring $C_4^2\geq 2^{-1}C_3C_5 $ and using $\sigma_0^{-2}\geq 2b_0^{-1}T^{-1}$, we have
\begin{align*}
\mathbb{P}_{ \mathcal{L}'_+}( \mathsf{G} |\mathsf{Gap})\geq C_6^{-1} e^{-C_6 C_4T^{5/2}}-C_3C_5e^{-b_0^{-1}C_5^{-1}C_4^2 T^3 }.
\end{align*}
Further requiring $C_4$ to be large enough, we can obtain $2^{-1}C_6^{-1} e^{-C_6 C_4T^{5/2}}\geq C_3C_5e^{-b_0^{-1}C_5^{-1}C_4^2 T^3 } $ and thus
\begin{align*}
\mathbb{P}_{ \mathcal{L}'_+}( \mathsf{G} |\mathsf{Gap})\geq 2^{-1}C_6^{-1} e^{-C_6 C_4T^{5/2}}. 
\end{align*}
For concreteness, taking
\begin{align*}
C_4\coloneqq \max\{  2(b_2+C_2), 2^{-1/2}C_3^{1/2}C_5^{1/2},b_0C_5(C_6+1), \log(2C_3C_5C_6)  \},
\end{align*}
We have
\begin{align*}
\mathbb{P}_{ \mathcal{L}'_+}( \mathsf{G})\geq 4^{-1}C_6^{-1} e^{-C_6 C_4T^{5/2}}.
\end{align*}
The proof is finished.
\end{proof} 

\section{The "soft" Jump ensemble}\label{sec:jump ensemble}
In this section, we introduce the jump ensemble $J$. Inspired by the jump ensemble \cite{Ham1},  $J$ is the intermediary between the scaled KPZ line ensemble $\mathfrak{H}^t$ and the Brownian bridge ensemble. We design $J$ to be adapted to the intersecting nature of $\mathfrak{H}^t$; moreover $J$ is built to meet the following requirements:
\begin{enumerate}
\item   The scaled KPZ line ensemble $\mathfrak{H}^t$ is well represented by $J$.
\item    $J$ is comparable to Brownian bridge ensemble on $[0,s]$. 
\end{enumerate} 
Requirement (1) is realized in Proposition \ref{clm:denominator} and requirement (2) is entailed in Proposition~\ref{clm:numerator}. These requirements are competing with each other. We find an appropriate  balance between these two. We proceed to explain the design concept of $J$ in detail.\\

Recall that $\textup{dom}(\bar{\ell},\bar{r}), \textup{bdd}(\bar{\ell},\bar{r})$ and $\textup{ext}(\bar{\ell},\bar{r})$ are regions in $\mathbb{N}\times \mathbb{R}$ defined in Section \ref{sec:framework}. We run Gibbs resampling for $\mathfrak{H}^t$ in $\textup{dom}(\bar{\ell},\bar{r})$. The law of $\mathfrak{H}_t|_{\textup{dom}(\bar{\ell},\bar{r})}$ is then given by $\mathbb{P}^{\textup{dom}(\bar{\ell},\bar{r}),f}_{\mathbf{H}_t}$ with $f=\mathfrak{H}^t|_{\textup{bdd}(\bar{\ell},\bar{r})}$. More precisely, $\mathbb{P}^{\textup{dom}(\bar{\ell},\bar{r}),f}_{\mathbf{H}_t}$ is specified via the following Radon-Nikodym derivative relation,
\begin{equation*}
\frac{{d}\mathbb{P}^{\textup{dom}(\bar{\ell},\bar{r}),f}_{\mathbf{H}_t}}{{d}\mathbb{P}^{\textup{dom}(\bar{\ell},\bar{r}),f}_{\free}}(\mathcal{L})=\frac{ W^{\textup{dom}(\bar{\ell},\bar{r}),f}_{\mathbf{H}_t}(\mathcal{L})}{ \mathbb{E}^{\textup{dom}(\bar{\ell},\bar{r}),f}_{\free}\left[ W^{\textup{dom}(\bar{\ell},\bar{r}),f}_{\mathbf{H}_t}\right]}.
\end{equation*}
Here $\mathbb{P}^{\textup{dom}(\bar{\ell},\bar{r}),f}_{\free}$ is the law of independent Brownian bridges in $\textup{dom}(\bar{\ell},\bar{r})$. The Boltzmann weight $W^{\textup{dom}(\bar{\ell},\bar{r}),f}_{\mathbf{H}_t}$ is defined in \eqref{equ:WH} as
\begin{equation*} 
\begin{split}
W_{\mathbf{H}_t}^{ \textup{dom}(\bar{\ell},\bar{r}), f }(\mathcal{L})= &\prod_{j=1}^{k-1} \exp\left( -\int_{[\ell_j,\ell_{j+1}]\cup [r_{j+1},r_j]}  \mathbf{H}_t(f_{j+1}(x)-\mathcal{L}_j(x))\, dx  \right)\\
&\times \prod_{j=1}^{k-1} \exp\left( -\int_{[\ell_{j+1},r_{j+1}]} \mathbf{H}_t(\mathcal{L}_{j+1}(x)-\mathcal{L}_j(x))\, dx \right)\\
&\times \exp\left( -\int_{[\ell_{k},r_k]}  \mathbf{H}_t(f_{k+1}(x)-\mathcal{L}_k(x))\, dx \right).
\end{split}
\end{equation*} 

The jump ensemble $J$ is obtained by replacing $W^{\textup{dom}(\bar{\ell},\bar{r}),f}_{\mathbf{H}_t}$ with a new weight $W_{\textup{jump}}$. In order to achieve the requirement (2), $W_{\textup{jump}}$ is designed in the way that 
\begin{itemize}
\item[(i)] The interaction in $W_{\textup{jump}}$ is only activated near a finite subset of $\mathbb{R}$,  (the Pole set). 
\item[(ii)] There is no self-interaction within $J$. The interaction is only present between curves in $J$ and the boundary data $f$.
\end{itemize} 
\vspace{0.1cm}
In order to achieve (1), the location where we activate the interaction in $W_{\textup{jump}}$ needs to be carefully chosen. Our choice is inspired by the following crucial observation of Hammond \cite{Ham1}.

If a function $g(x)$ is close to a parabola $-2^{-1}x^2$, its concave majorant $\mathfrak{c}(x)$ is also close to $-2^{-1}x^2$ with derivative close to $-x$, see Lemma~\ref{lem:slope}. At extreme points of $\mathfrak{c}(x)$, $\mathfrak{c}$ agrees with $g$. Let $P$ a discrete subset of extreme points of $\mathfrak{c}(x)$. A linear interpolation between points in $P$, denoted by Tent$(x)$, gives a good approximation to $\mathfrak{c}$, see Lemma~\ref{lem:Tentlb}. As a result, a Brownian bridge conditioned on jumping over $g(x)$ in an interval is well-approximated by a Brownian bridge conditioned on jumping over $g(x)$ only at points in $P$.  See Figure~\ref{fig:parabola} for an illustration.
\begin{figure}
\includegraphics[width=10cm]{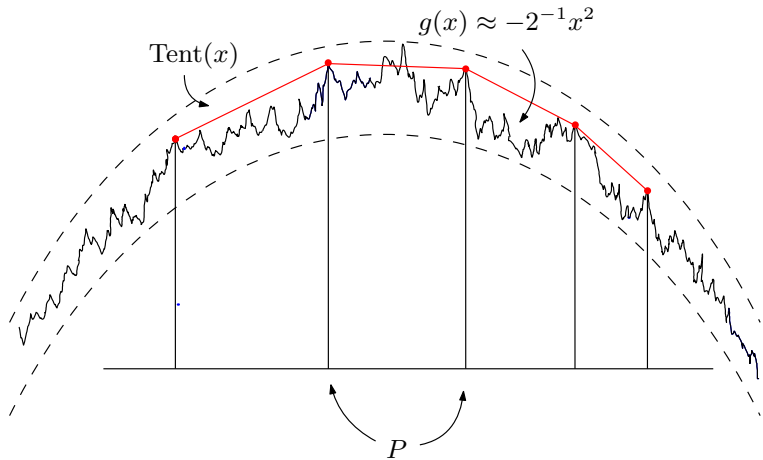}
\caption{Illustration of the Tent as a good approximation to the concave majorant $\mathfrak{c}(x)$.}
\label{fig:parabola}
\end{figure}

There is an extra difficulty. In the current context, different indexed jump curves $J_j$ travel in different intervals, e.g. $ ({\ell }_j, {r}_j)$ for $J_i$. Near $x=\ell_j$, we need to have Tent$(x)$, up to a controllable error, bounded from above by  $f_{j}(\ell_j) $. This scenario allows us to raise $J_j$ above Tent$(x)$ within a short amount of travel time. Therefore, we need to construct multiple Pole sets $P_{j}$ and multiple tent maps $\textup{Tent}_{j}(x)$, one for each $J_j$.\\

In subsection \ref{subsec:poleset} we list essential properties of $P_{j+1}$ needed in the rest of the paper. The construction of $P_{j+1}$ and the verification of those properties are postponed to subsection \ref{subsec:poleconstruct}. Skipping subsection~\ref{subsec:poleconstruct} will not affect the reader's understanding for the rest of the paper. The construction of $W_{\textup{jump}}$ and the jump ensemble $J$ are done in subsection \ref{subsection:jump}.

\subsection{Pole sets}\label{subsec:poleset}
 Throughout this section, we fix $(\bar{\ell},\bar{r},f_{\textup{J}})\in\mathcal{G}$. Denote $\bar{\ell}=(\ell_1,\ell_2\dots,\ell_k)$, $\bar{r}=(r_1,r_2\dots,r_k)$ and $f_{\textup{J}}=(f_1,f_2\dots,f_{k+1})$.  \\

Define $\ug(x)$  by
\begin{equation}\label{def:g2}
\begin{split}
\ug(x)\coloneqq\left\{ \begin{array}{cc}
f_{2}(x) & x\in [\ell_0, {\ell}_{2})\cup( {r}_{2},r_0 ],\\
f_{i+1}(x) & x\in ( {\ell}_{i}, {\ell}_{i+1})\cup({r}_{i+1}, {r}_{i}),\ i\in [2 ,k-1]_{\mathbb{Z}},\\
f_{k+1}(x) & x\in ({\ell}_{k}, {r}_k),\\
\max\{f_i(x),f_{i+1}(x)\} & x\in \{  {\ell}_{i  }, {r}_i \},\ i\in [2,k]_{\mathbb{Z}}.
\end{array}\right.
\end{split}
\end{equation}
The function $\ug$ consists of parts in $f_{\textup{J}}$ that interact $\mathfrak{H}$. At $x=\ell_j$ or $x=r_j$, the value of $\ug(x)$ is chosen such that $\ug(x)$ is upper semi-continuous.  See Figure \ref{fig:ug} for an illustration.   
\begin{figure}[H]
\includegraphics[width=6cm]{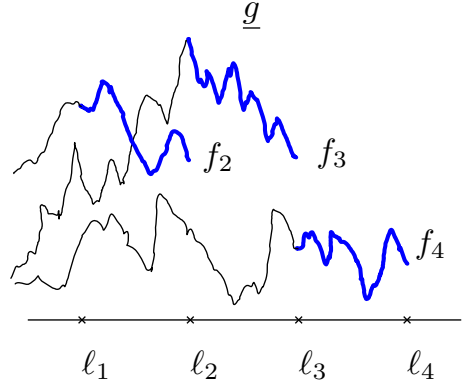}
\caption{An illustration of $\ug(x)$ when $k=3$. Recall that the resampling domain resembles an upside-down pyramid and the the resampled curves only interact with values on the boundary of the resampling domain, i.e. $f_2$ on $[\ell_1, \ell_2]$, $f_3$ on $[\ell_2, \ell_3]$, $f_3$ on $[\ell_3, \ell_4]$. }
\label{fig:ug}
\end{figure}
For all $j\in [1,k]_{\mathbb{Z}}$ and $x\in (\ell_j,r_j)$, we have
\begin{equation}\label{ug_bdd}
\begin{split}
\ug(x)\leq &f_{j+1}(\ell_j)+(k+1-j)T(x-\ell_j),\\
\ug(x)\leq &f_{j+1}(r_j)+(k+1-j)T(r_j-x).
\end{split}
\end{equation}
See Figure~\ref{fig:gT} for an illustration. The proof can be found in the next subsection.\\
\begin{figure}[H]
\includegraphics[width=10cm]{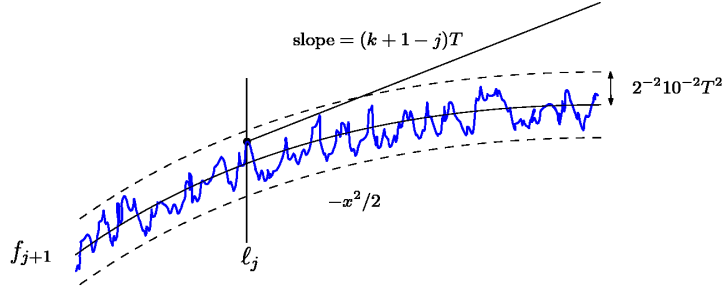}
\caption{This picture illustrates \eqref{ug_bdd} near $\ell_j$. Note that for $x\in (\ell_j,\ell_{j+1}),\ \ug(x)=f_{j+1}(x)$. }
\label{fig:gT}
\end{figure}
In the next subsection, we will construct Pole sets $P_{j+1}$ for $j\in [1,k]_{\mathbb{Z}}$. $P_{j+1}$ consists of some extreme points of a concave function which captures $f_{j+1}$ well near $x=\ell_j$ and $x=r_j$. $P_{j+1}$ satisfies the following requirements.\\

 $P_{j+1}$ is a subset of $[\ell_j,r_j]$ which includes $\ell_j$ and $r_j$. The number of pole points in $P_{j+1}$ is at most of order $T$. 
\begin{align}\label{Pole}
P_{j+1}\subset [\ell_j,r_j],\ \ell_j,r_j\in P_{j+1}\ \textup{and}\ |P_{j+1}|\leq |r_j-\ell_j|+1.
\end{align}

We require different indexed $P_{j+1}$ agree each other in $(-2^{-1}T,2^{-1}T)$. For $i,j\in [1,k]_{\mathbb{Z}}$,
\begin{align}\label{Pole_agree}
P_{i+1}\cap (-2^{-1}T,2^{-1}T)=P_{j+1}\cap (-2^{-1}T,2^{-1}T).
\end{align}

In order to make $J$ close to Brownian bridges in the interval $[0,s]$, the number of pole points in $[0,s]$ is limited to $1$. Suppose $p_0\in P_{j+1}\cap [0,s]$. Let $p_-=\max\{p\in P_{j+1}\,|\,p<p_0\}$ and $p_+=\min\{p\in P_{j+1}\,|\,p>p_0\}$ be the elements in $P_{j+1}$ next to $p_0$. We require that
\begin{align}\label{Pole_three}
s\leq |p_0-p_\pm|<4^{-1}T.
\end{align}

Near $\ell_j$ and $r_j$, we need to keep the full interaction $-\mathbf{H}_t(f_{j+1}(x)-\mathcal{L}_j(x))$ unchanged in $W_{\textup{jump}}$. See subsection \ref{subsection:jump} for more explanation. We want to make sure this full interaction is not interrupted by the pole points. To do so, we require $P_{j+1}$ to avoid $\ell_i$ and $r_i$. Define 
\begin{equation}\label{d'}
d'\coloneqq T^{-3/2}.
\end{equation}  
$d'$ determines the width of intervals near $\ell_j$ and $r_j$ in which $-\mathbf{H}_t(f_{j+1}(x)-\mathcal{L}_j(x))$ is preserved in $W_{\textup{jump}}$. We will explain the choice of scaling in more detail in Section \ref{sec:denominator}. For all $j\in [1,k]_{\mathbb{Z}}$ and $i\in [j,k]_{\mathbb{Z}}$, we ask 
\begin{align}\label{Pole_i}
P_{j+1}\cap (\ell_i,\ell_i+2d')=P_{j+1}\cap (r_i-2d',r_i )=\phi.
\end{align} 

Define the tent maps
\begin{equation}\label{def:Tentj}
\textup{Tent}_{j+1}(x)\coloneqq \left\{ \begin{array}{cc}
f_{j+1}(x) & x=\ell_j\ \textup{or}\ x=r_j,\\
\ug(x) & x\in P_{j+1}\setminus \{\ell_j,r_j\},\\
\textup{linear inteperlation} & x\in [\ell_j,r_j]\setminus P_{j+1}.
\end{array} \right.
\end{equation}
We need $\textup{Tent}_{j+1}(x)$ to be a good approximation of $\ug(x)$ in the following sense. For all $x\in (\ell_j,r_j)$,\\[-0.25cm]
\begin{equation}\label{Tentj_lowerbound}
\textup{Tent}_{j+1}(x)\geq \ug(x)-3(k+1-j)sT.
\end{equation}
Moreover, we require that the slope of $\textup{Tent}_{j+1}(x)$ is of order $T$.\\[-0.25cm]
\begin{equation}\label{Tentj_slope}
\sup_{x\neq y\in [\ell_{j },r_{j }]}\left| \frac{\textup{Tent}_{j +1}(x)-\textup{Tent}_{j+1 }(y)}{x-y} \right| \leq 2(k+1-j)T.
\end{equation}

The construction of $P_{j+1}$ and the verification of above properties can be found in the next subsection. 
\subsection{Construction of pole sets}\label{subsec:poleconstruct}
 For $j\in [1,k]_{\mathbb{Z}}$, define

\begin{equation}\label{def:g}
\begin{split}
g_{j+1}(x)\coloneqq\left\{ \begin{array}{cc}
f_{j+1}(x) & x\in [\ell_0, {\ell}_{j}]\cup [ {r}_{j},r_0 ],\\
\ug(x) & x\in (\ell_j,r_j).
\end{array}\right.
\end{split}
\end{equation}
Let $\mathfrak{c}_{j+1}$ be the concave majorant of $g_{j+1}$ in $[\ell_0,r_0]$.
\begin{align*}
\mathfrak{c}_{j+1}(x_0)\coloneqq\inf\{ax_0+b\ |\ ax+b\geq g_{j+1}(x)\ \textup{for all}\ x\in  [\ell_0,r_0]\}.
\end{align*} 
Define
\begin{equation*}
\begin{split}
\textup{xExt}_{j+1 }\coloneqq   &\left\{ x\in (\ell_0 ,r_0  )\ |\  \mathfrak{c}_{j+1 }\  \textup{is not linear near }\ x \right\}\cup\{\ell_0,r_0\}.
\end{split}
\end{equation*}
The reason to use $g_{j+1}(x)$ instead of $\ug(x)$ to construct the concave majorant is to obtain a better fit to $f_{j+1}$ near $x=\ell_j$ and $x=r_j$. Especially, we have the following lemma.
\begin{lemma}
For all $j\in [1,k]_{\mathbb{Z}}$,
$  \ell_j,r_j \in \textup{xExt}_{j+1}$ and \eqref{ug_bdd} holds.
\end{lemma}

\begin{proof}
From Condition \textbf{C1} and the construction of $(\mathbf{l},\mathbf{r})$ in Section \ref{sec:Fav}, we have
\begin{align*}
|\ell_{j }+(k-j+1)T|\leq 5^{-1}T,|r_{j }-(k-j+1)T|\leq 5^{-1}T. 
\end{align*} 
 
From \eqref{def:g}, $g_{j+1}$ and $f_{j+1}$ agree on $R_{j+1,\pm}=[\pm (k+1-j)T-2^{-1}T,\pm (k+1-j)T +2^{-1}T ]$. Recall that $\mathfrak{f}_{j+1}$, defined in section \ref{sec:Fav}, is the concave majorant of $f_{j+1}$ on $R_{j+1,\pm}$. By Condition \textbf{C2}, we can apply Corollary~\ref{cor:majorant} with $\delta=10^{-1}T$ to show that $\mathfrak{c}_{j+1}$ and $\mathfrak{f}_{j+1}$ agree on $\bar{R}_{j+1,\pm}=[\pm (k+1-j)T-5^{-1}T,\pm (k+1-j)T +5^{-1}T ]$. In particular, $\ell_j$ and $r_j$ are extreme points of $\mathfrak{c}_j$. See Lemmas~\ref{lem:liri} and \ref{lem:slope}.\\

Next, we prove \eqref{ug_bdd}. From Lemma \ref{lem:contact} and \eqref{def:g}, 
$$\mathfrak{c}_{j+1}(\ell_j)=g_{j+1}(\ell_j)=f_{j+1}(\ell_j).$$
Moreover, Condition \textbf{C2} and Lemma~\ref{lem:slope} imply 
\begin{align*}
|\ell_j+(k+1-j)T|\leq 10^{-1}T,|r_j-(k+1-j)T|\leq 10^{-1}T. 
\end{align*}
In particular, $(\mathfrak{c}'_{j+1})_+(\ell_j)\leq (k+1-j)T$. By the concavity of $\mathfrak{c}_{j+1}$, we have for all $x\in (\ell_j,r_j)$
\begin{align*}
\ug (x)= &g_{j+1}(x) \leq \mathfrak{c}_{j+1}(x)\\
\leq &\mathfrak{c}_{j+1}(\ell_j)+(k+1-j)T(x-\ell_j)\\
=&f_{j+1}(\ell_j)+(k+1-j)T(x-\ell_j).
\end{align*}
This proves the first part of \eqref{ug_bdd}. The other part can be proved similarly.
\end{proof}
 
\begin{remark}
From Condition \textbf{C2} and Corollary \ref{cor:majorant}, $\mathfrak{c}_{j+1}(x)$ agrees each other on $ [-2^{-1}T,2^{-1}T] $ for all $j\in [1,k]_{\mathbb{Z}}$. Hence $\textup{xExt}_{j+1}\cap (-2^{-1}T,2^{-1}T)$ are also the same.
\end{remark}
For $j\in [1,k ]_{\mathbb{Z}}$, let $  \mathring{P}_{j+1 }\subset \textup{xExt}_{j+1 }\cap [\ell_{j },r_{j }]$ be a subset satisfying the following properties.
\begin{itemize}
\item $\mathring{P}_{j+1}\cap (-2^{-1}T,2^{-1}T)$ are identical for all  $j\in [1,k]_{\mathbb{Z}}$.
\item $\ell_{j },r_{j }\in \mathring{P}_{j+1 }$
\item For any $x_1\neq x_2\in \mathring{P}_{j+1 }$, $|x_1-x_2|\geq s$
\item For any $y\in \textup{xExt}_{j+1 }\cap [\ell_{j },r_{j }]$, there exists $x\in \mathring{P}_{j+1 }$ such that $|x-y|<s$.
\end{itemize}
The sets $\mathring{P}_{j+1}$ satisfy the conditions \eqref{Pole}, \eqref{Pole_agree} and \eqref{Pole_three}. To get \eqref{Pole_i}, we need further modification. Let 
\begin{align*}
A_{j+1}^-\coloneqq\{i\in [j+1,k]_{\mathbb{Z}}\, |\, P_{j+1}\cap (\ell_i,\ell_i+2d')\neq \phi\},\\
A_{j+1}^+\coloneqq\{i\in [j+1,k]_{\mathbb{Z}}\, |\, P_{j+1}\cap  (r_i-2d',r_i) \neq \phi\}. 
\end{align*}
The sets $A_{j+1}^\pm$ record the indices of $i$ for which $P_{j+1}$ has a pole point near $x=\ell_i$ or $x=r_i$ respectively. If $i\in A_{j+1}^-$, we replace the pole points in $\mathring{P}_{j+1}\cap (\ell_i,\ell_i+2d')$ by $\ell_i$. The same replacement is also done for $i\in A^+_{j+1}$. Concretely, we define the Pole set by
\begin{equation}\label{def:Pj}
\begin{split}
 {P}_{j+1 }\coloneqq &\mathring{P}_{j+1 }\cup \{\ell_i\, |\, i\in A^-_{j+1}  \} \cup \{r_i\, |\, i\in A^+_{j+1} \} \setminus\bigcup_{i=j+1}^k\bigg( (\ell_i,\ell_i+2d')\cup (r_i-2d',r_i ) \bigg).
\end{split}
\end{equation}
By the construction, \eqref{Pole_i} holds.\\

From the view of Lemma \ref{lem:slope} and condition \textbf{C2}, two consecutive elements in $\textup{xExt}_{j+1}$ have distance less than $10^{-1}T$. Together with $T\geq 10$, \eqref{Pole_three} holds.\\

Next, we aim to show \eqref{Tentj_lowerbound} and \eqref{Tentj_slope}. Define the tent map from $\mathring{P}_{j+1}$ as
\begin{equation*} 
\mathring{\textup{T}}\textup{ent}_{j+1}(x)\coloneqq \left\{ \begin{array}{cc}
f_{j+1}(x) & x=\ell_j\ \textup{or}\ x=r_j,\\
\ug(x) & x\in \mathring{P}_{j+1}\setminus \{\ell_j,r_j\},\\
\textup{linear inteperlation} & x\in [\ell_j,r_j]\setminus \mathring{P}_{j+1}.
\end{array} \right.
\end{equation*}
Because $|(\mathfrak{c}_{j+1 })'_{\pm}(x)|\leq (k+1-j)T$ for $x\in [\ell_{j },r_{j }]$, we have from Lemma \ref{lem:Tentlb}, 
\begin{align*}
\mathring{\textup{T}}\textup{ent}_{j+1}(x)\geq &\mathfrak{c}_{j+1}(x)-2(k+1-j)sT     \geq g_{j+1}(x)-2(k+1-j)sT.
\end{align*}
In particular, for $x\in (\ell_j,r_j)$, we have
\begin{align*}
\mathring{\textup{T}}\textup{ent}_{j+1}(x)\geq   \ug(x)-2(k+1-j)sT.
\end{align*}
Moreover,
\begin{equation*} 
\sup_{x\neq y\in [\ell_{j },r_{j }]}\left| \frac{\mathring{\textup{T}}\textup{ent}_{j +1}(x)-\mathring{\textup{T}}\textup{ent}_{j+1 }(y)}{x-y} \right| \leq (k+1-j)T.
\end{equation*}

The strategy of proving \eqref{Tentj_lowerbound} and \eqref{Tentj_slope} is to compare $ {\textup{T}}\textup{ent}_{j +1}$ with $\mathring{\textup{T}}\textup{ent}_{j +1}$ using Lemma \ref{lem:Tentmod}. 
\begin{lemma}\label{lem:gap}
Suppose $p\in \mathring{P}_{j+1}\cap (\ell_i,\ell_i+2d']$ for some $j\in [1,k]_{\mathbb{Z}}$ and $i\in [j+1,k]_{\mathbb{Z}}$. We have
\begin{equation*}
|\ug (\ell_i)-\ug (p)| \leq 2(k+1-j)T^{-1/2}.
\end{equation*}
Similarly, suppose  $p\in \mathring{P}_j\cap [r_i-2d',r_i)$ for some $j\in [1,k]_{\mathbb{Z}}$ and $i\in [j+1,k]_{\mathbb{Z}}$. Then
\begin{equation*}
|\ug(r_i)-\ug (p)| \leq  2(k+1-j)T^{-1/2}.
\end{equation*}
\end{lemma}
We postpone the proof of Lemma~\ref{lem:gap} and show \eqref{Tentj_lowerbound} and \eqref{Tentj_slope} first. For \eqref{Tentj_lowerbound}, it remains to prove
\begin{align*}
  \textup{Tent}_{j+1}(x)\geq \mathring{\textup{T}}\textup{ent}_{j+1}(x) - (k+1-j)sT.
\end{align*}
Note that $\textup{Tent}_{j+1}(x) $ agrees with $ \mathring{\textup{T}}\textup{ent}_{j+1}(x)$ except for $x$ near $\ell_i$ and $r_i $.\\

We give the proof for the case $x$ close to $\ell_i$. We may assume that $\mathring{P}_j\cap (\ell_i,\ell_i+2d')\neq \phi$. Otherwise, $\textup{Tent}_{j+1}(x)=\mathring{\textup{T}}\textup{ent}_{j+1}(x)$ near $\ell_i$. Let $p=\mathring{P}_j\cap (\ell_i,\ell_i+2d')$, $q=\ell_i$,  $p_-<p<p_+$ be the elements in $\mathring{P}_{j+1}$ next to $p $. In the interval $[p_-,p_+]$, the tent maps $\mathring{\textup{T}}\textup{ent}_{j+1}(x)$ and $\mathring{\textup{T}}\textup{ent}_{j+1}(x)$ are piecewise linear functions given by
\begin{equation*} 
\mathring{\textup{T}}\textup{ent}_{j+1}(x)= \left\{ \begin{array}{cc}
\ug(x) & x=p\ \textup{or}\ x=p_\pm ,\\
\textup{linear inteperlation} & x\in (p_-,p)\cup(p,p_+) .
\end{array} \right.
\end{equation*}
\begin{equation*} 
 {\textup{T}}\textup{ent}_{j+1}(x)= \left\{ \begin{array}{cc}
\ug(x) & x=q\ \textup{or}\ x=p_\pm ,\\
\textup{linear inteperlation} & x\in (p_-,q)\cup(q,p_+) .
\end{array} \right.
\end{equation*}

To compare $\mathring{\textup{T}}\textup{ent}_{j+1}(x)$ and ${\textup{T}}\textup{ent}_{j+1}(x)$, we apply Lemma~\ref{lem:Tentmod}. From Lemma~\ref{lem:gap},
\begin{align*}
|\ug ( \ell_i)-\ug(p)| \leq 2 (k+1-j)T^{-1/2}.
\end{align*}
Note that $|p-q|\leq 2d'=2T^{-3/2}$ and $|p_\pm-p|\geq s\geq 1$. Applying Lemma~\ref{lem:Tentmod} yields
\begin{align*}
  \textup{Tent}_{j+1}(x)-\mathring{\textup{T}}\textup{ent}_{j+1}(x)\geq -2 (k+1-j)T^{-1/2}  - 2T^{-3/2}\max\{|\mathring{m}_-|,|\mathring{m}_+|\} .
\end{align*}
Here $\mathring{m}_\pm$ is the slope of $\mathring{\textup{T}}\textup{ent}_{j+1}$ in $[p_-,p]$ and $[p,p_+]$. Because the slope of $\mathring{\textup{T}}\textup{ent}_{j+1}$ is bounded by $(k+1-j)T$, we get
$$\textup{Tent}_{j+1}(x)\geq  \mathring{\textup{T}}\textup{ent}_{j+1}(x) -4 (k+1-j)T^{-1/2}.$$ 
The proof of \eqref{Tentj_lowerbound} near $\ell_i$ is finished.\\

Next, we turn to \eqref{Tentj_slope}. We again apply Lemma~\ref{lem:Tentmod}. The difference between the slopes of $\mathring{\textup{T}}\textup{ent}_{j+1}$ and $ {\textup{T}}\textup{ent}_{j+1}$ is bounded by $8 (k+1-j)T^{-1/2}.$ Therefore, \eqref{Tentj_slope} holds true. The argument for the case near $r_i$ is similar.

\begin{proof}[{\bf Proof of Lemma~\ref{lem:gap}}]
By symmetry, it suffices to prove the case $p\in \mathring{P}_{j+1}\cap (\ell_i,\ell_i+2d')$. We start by showing $\ug(\ell_i)\leq \ug(p)$. By the concavity of $\mathfrak{c}_{j+1}$,
\begin{align*}
\ug(\ell_i)=&g_{j+1}(\ell_i)\leq \mathfrak{c}_{j+1}(\ell_i) \leq  \mathfrak{c}_{j+1}(p)+(\mathfrak{c}'_{j+1})_+(p)(\ell_i-p).
\end{align*}
Because $p$ is an extreme point of $\mathfrak{c}_{j+1}$,   
\begin{align*}
\mathfrak{c}_{j+1}(p)+(\mathfrak{c}'_{j+1})_+(p)(\ell_i-p)=g_{j+1}(p)-(\mathfrak{c}'_{j+1})_+(p)(p-\ell_i).
\end{align*}
From the definitions of $\ug$ and and $g_{j+1}$, see \eqref{def:g2} and \eqref{def:g},
\begin{align*}
g_{j+1}(p)-(\mathfrak{c}'_{j+1})_+(p)(p-\ell_i) =&\ug(p)-(\mathfrak{c}'_{j+1})_+(p)(p-\ell_i) .
\end{align*}
By Lemma \ref{lem:slope}, $(\mathfrak{c}'_{j+1})_+(p) \geq -p-10^{-1}T\geq 0$. We conclude $\ug(\ell_i)\leq \ug(p)$.\\

Next, we show that $\ug(p)\leq \ug(\ell_i)+2(k+1-j)T^{-1/2}.$ Because $p\in (\ell_i,\ell_i+2d')$,
\begin{align*}
\ug(p)=f_{i+1}(p)=g_{i+1}(p).
\end{align*} 
By the concavity of $\mathfrak{c}_{i+1}$,
\begin{align*}
g_{i+1}(p)\leq \mathfrak{c}_{i+1}(p)\leq \mathfrak{c}_{i+1}(\ell_i)+(\mathfrak{c}_{i+1}')_+(\ell_i)(p-\ell_i).
\end{align*}
Because $\ell_i$ is an extreme point of $\mathfrak{c}_{i+1}$, it holds that
$$\mathfrak{c}_{i+1}(\ell_i)=g_{i+1}(\ell_i)=f_{i+1}(\ell_i)\leq \ug(\ell_i).$$ 
From Lemma \ref{lem:slope}, we have $$\mathfrak{c}_{i+1}'(\ell_i)\leq -\ell_i+10^{-1}T\leq (k+2-i)T\leq (k+1-j)T.$$ 
Together with $p-\ell_i\leq 2T^{-3/2}$, we conclude that
\begin{align*}
\ug(p)\leq \ug(\ell_i)+2(k+1-j)T^{-1/2}.
\end{align*}
\end{proof}

\subsection{Jump ensemble}\label{subsection:jump}
 In this subsection, we use the Pole set $P_{j+1}$ to construct the jump ensemble.  We put all $P_{j+1}$ together to form the joint Pole set.
\begin{equation}\label{def:Pstar}
P^*\coloneqq\bigcup_{j=1}^{k }P_{j+1}.
\end{equation}

Recall that $d=64T^{-4}(\log T)^2$. For $j\in [1,k-1]_\mathbb{Z}$, define
\begin{equation*}
\textup{Int}_j\coloneqq \bigg(\bigcup_{p\in P^*\cap (\ell_j,\ell_{j+1}]} [p-d,p]\bigg) \cup \bigg(\bigcup_{p\in P^*\cap  [r_{j+1},r_j)} [p ,p+d]\bigg)
\end{equation*}
and
\begin{equation*}
\textup{Int}_k\coloneqq \bigg(\bigcup_{p\in P^*\cap (\ell_k,0) } [p-d,p]\bigg) \cup \bigg(\bigcup_{p\in P^*\cap  [0,r_k)} [p ,p+d]\bigg)
\end{equation*}
From \eqref{Pole_i}, we have $\textup{Int}_j\subset (\ell_j,\ell_{j+1}]\cup [r_{j+1},r_j) $ for $j\in [1,k-1]_{\mathbb{Z}}$ and $\textup{Int}_k\subset (\ell_k,r_k)$. $\textup{Int}_j$ and $\textup{Int}_k$ indicates the region where $J$ interacts with the boundary. The index $j$ in $\textup{Int}_j$ records the layer of the boundary data. For instance, in region $\textup{Int}_j$, the lower boundary is given by $f_{j+1}$. 
\begin{figure}[H]
\includegraphics[width=10cm]{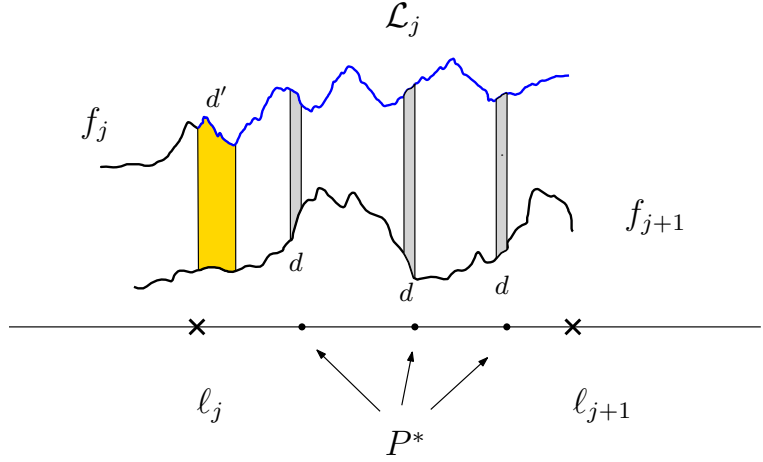}
\caption{This picture illustrates $W^j_{\textup{jump}}$ near $\ell_j$. In the thick pillar with width $d'$, we use the Hamiltonian $\textbf{H}_t$ as in the scaled KPZ line ensemble. In the thin pillars with width $d$, a weaker Hamiltonian $\textbf{H}_t^{j,j+1}$ is used. There is no interaction on other places.  }
\end{figure}
Now we are ready to define the jump weight $W_{\textup{jump}}$. Recall that we fix the triple $(\bar{\ell},\bar{r},f_{\textup{J}})\in\mathcal{G}$. Fix $j\in [1,k]_{\mathbb{Z}}$. Let
\begin{equation*}
\begin{split}
W^{j}_{\textup{jump}}(\mathcal{L}_j ;f|_{\textup{J}})\coloneqq &\exp\left( -\int_{[\ell_j,\ell_j+d']\cup [r_j-d' ,r_j] } \mathbf{H}_t(\ug (x)-\mathcal{L}_j(x)) dx \right)\\
& \times \exp\left( -\sum_{i=j}^k \int_{\textup{Int}_i} \mathbf{H}_t^{j,i+1}(\ug (x)-\mathcal{L}_j(x)) dx \right) .
\end{split}
\end{equation*}
Here $\ug(x)$ is the lower boundary defined in \eqref{def:g2} and $\mathbf{H}_t^{i,{j+1}}(x)$ is the Hamiltonian defined in \eqref{def:Hk}. $\mathbf{H}_t^{i,{j+1}}(x)$ is chosen to guarantee Lemma \ref{lem:Wrest1}.

For any $\ell<r$, we define the truncated version \\[-0.25cm]
\begin{equation}\label{def:Wjump_j}
\begin{split}
W^{j, [\ell,r]}_{\textup{jump}}(\mathcal{L}_j ;f|_{\textup{J}})\coloneqq &\exp\left( -\int_{([\ell_j,\ell_j+d']\cup [r_j-d' ,r_j])\cap[\ell,r] } \mathbf{H}_t(\ug (x)-\mathcal{L}_j(x)) dx \right)\\
& \times \exp\left( -\sum_{i=j}^k \int_{\textup{Int}_i\cap[\ell,r]} \mathbf{H}_t^{j,i+1}(\ug (x)-\mathcal{L}_j (x)) dx \right) .
\end{split}
\end{equation}
The jump weight $W_{\textup{jump}} $ is defined by
\begin{align}
W_{\textup{jump}}(\mathcal{L} ;f_{\textup{J}})\coloneqq &\prod_{j=1}^k W^{j}_{\textup{jump}}(\mathcal{L}_j;f_{\textup{J}}).\label{def:Wjump} 
\end{align}

In $W_{\textup{jump}} $, the random curves don't interact with each other. In particular, the jump curves $J_i$ and $J_j$, defined in \eqref{def:J}, are independent of each other if $i\neq j$. 

The $j$-th jump curve $J_j$ interacts with $\ug$ near $\ell_j$ and $r_j$ through $\mathbf{H}_t$. This is the same Hamiltonian in $W_{\mathbf{H}_t}^{ \textup{dom}(\bar{\ell},\bar{r}), f }$. See \eqref{equ:WH}. The reason to keep the Hamiltonian function comes from the possibility that  $f_{\textup{J}}$ may go out of order. If $f_{j+1}>f_j$ near $\ell_j$, the interaction in $-\mathbf{H}_t(f_{j+1}(x)-\mathcal{L}_j(x))$ near $\ell_j$ becomes so drastic such that any mollification would make a huge difference. Thus we preserve the original Hamiltonian. 

The other location where $J_j$ interacts with $\ug(x)$ is determined by the Pole set $P^*$, which depends on $f_{\textup{J}}$.  \\

Having the jump weight $W_{\textup{jump}}$, $W_{\textup{rest}}$ is defined by
\begin{align}
W_{\textup{rest}}(\mathcal{L} ;f_{\textup{J}})\coloneqq &W^{\textup{dom}(\bar{\ell},\bar{r}),f}_{\mathbf{H}_t}(\mathcal{L} )/ W_{\textup{jump}}(\mathcal{L} ;f_{\textup{J}}).\label{def:Wrest}
\end{align}
Here $f=f_{\textup{J}}|_{\textup{bdd}(\bar{\ell},\bar{r})}$.\\

We finish this section by the following lemma.
\begin{lemma}\label{lem:Wrest1}
It holds that
\begin{align*}
W_{\textup{rest}}(\mathcal{L} ;f_{\textup{J}})\leq 1.
\end{align*}
\end{lemma}
\begin{proof}
This is equivalent to $W^{\textup{dom}(\bar{\ell},\bar{r}),f}_{\mathbf{H}_t}(\mathcal{L} )\leq   W_{\textup{jump}}(\mathcal{L} ;f_{\textup{J}})$.
From \eqref{def:g2}, we can rewrite $W_{\mathbf{H}_t}^{ \textup{dom}(\bar{\ell},\bar{r}), f }$ in \eqref{equ:WH} as
\begin{equation*}
\begin{split}
W_{\mathbf{H}_t}^{ \textup{dom}(\bar{\ell},\bar{r}), f }(\mathcal{L})= &\prod_{j=1}^{k-1} \exp\left( -\int_{[\ell_j,\ell_{j+1}]\cup [r_{j+1},r_j]}  \mathbf{H}_t(\ug (x)-\mathcal{L}_j(x))\, dx  \right)\\
&\times \prod_{j=1}^{k-1} \exp\left( -\int_{[\ell_{j+1},r_{j+1}]} \mathbf{H}_t(\mathcal{L}_{j+1}(x)-\mathcal{L}_j(x))\, dx \right)\\
&\times \exp\left( -\int_{[\ell_{k},r_k]}  \mathbf{H}_t(\ug (x)-\mathcal{L}_k(x))\, dx \right).
\end{split}
\end{equation*}
Suppose $x_0\in (\ell_j,\ell_{j+1})$ for some $j\in [1,k]_{\mathbb{Z}}$, the integrand of $\log W_{\mathbf{H}_t}^{ \textup{dom}(\bar{\ell},\bar{r}), f }$ at $x_0$ is
\begin{align}\label{compare}
-\mathbf{H}_t(\ug (x_0)-\mathcal{L}_j(x_0))-\sum_{i=1}^{j-1}\mathbf{H}_t(\mathcal{L}_{i+1} (x_0)-\mathcal{L}_i(x_0)).
\end{align}
We claim that it is smaller than the integrand of $\log W_{\textup{jump}}(\mathcal{L} ;f_{\textup{J}})$. Suppose $x_0\in (\ell_j,\ell_j+d']$. Because of \eqref{Pole_i}, for any $i \in [1,k]_{\mathbb{Z}}$,
$$\textup{Int}_i\cap (\ell_j,\ell_j+d']=\phi.$$
Hence the integrand of $\log W_{\textup{jump}}(\mathcal{L} ;f_{\textup{J}})$ at $x_0$ is $-\mathbf{H}_t(\ug(x_0)-\mathcal{L}_j(x_0)))$, which is larger than \eqref{compare}. Suppose $x_0\in \textup{Int}_j $, the integrand of $\log W_{\textup{jump}}(\mathcal{L} ;f_{\textup{J}})$ at $x_0$ is 
\begin{align*}
-\sum_{i=1 }^j \mathbf{H}_t^{i,j+1}(\ug  (x_0)-\mathcal{L}_i(x_0)).
\end{align*}
In light of Lemma \ref{lem:Hk}, it is larger than \eqref{compare}. Suppose $x_0\in ( \ell_j+d',\ell_{j+1})\setminus \textup{Int}_j$. Then the integrand of $\log W_{\textup{jump}}(\mathcal{L} ;f_{\textup{J}})$ at $x_0$ is zero, which is also larger than \eqref{compare}. This completes the discussion for the case $x_0\in (\ell_j,\ell_{j+1})$. Other cases, $x_0\in (r_{j+1},r_{j })$ and $x_0\in (\ell_{k+1},r_{k+1})$ can be treated similarly and we omit the detail. The remaining end points $x_0= \ell_i$ or $x_0= r_i$ have zero contribution. The proof is finished.
\end{proof}
\section{Proof of Proposition \ref{clm:denominator}}\label{sec:denominator}
In this section, we aim to show that 
\begin{align*}
\inf \{\mathbb{E}_{J}[ W_{\textup{rest}}]\, |\,  (\bar{\ell},\bar{r},f_{\textup{J}})\in \mathcal{G}   \}  \geq D^{-1} e^{-DT^{5/2}}.
\end{align*}
Here $\mathcal{G}$ is defined in Definition~\ref{def:G} and $W_{\textup{rest}}$ is defined in \eqref{def:Wrest}. Throughout this section, we fix $(\bar{\ell},\bar{r},f_{\textup{J}})\in\mathcal{G}$. In below we explain the strategy of the proof.\\

 Consider the event $\mathsf{NoTouch}$ that requires the jump ensemble $J$ to stay ordered and to lie above $\ug$ except in neighborhoods near $\ell_j$ or $r_j$. Concretely,
\begin{align*}
\mathsf{NoTouch}\coloneqq \{J_j(x)\geq \ug(x)\ \textup{for}\ j\in [1,k]_{\mathbb{Z}}\ \textup{and}\ x\in [\ell_j+d',r_j-d'],\\
J_j(x)\geq J_{j+1}(x)\ \textup{for}\ j\in [1,k-1]_{\mathbb{Z}}\ \textup{and}\ x\in [\ell_{j+1},r_{j+1}]\}.
\end{align*} 
When $\mathsf{NoTouch}$ occurs, the weight $W_{\textup{rest}}$ is bounded from below by $e^{-T}$. See Lemma \ref{lem:Westbdd}. Therefore Proposition~\ref{clm:denominator} is reduced to showing $\mathbb{P}_J(\mathsf{NoTouch})\geq e^{-T^{5/2}}$.   

To begin, we show that each curve $J_j$ can jump over $\ug$ in the interval $[\ell_j+d',r_j-d']$ with probability $e^{-T^{5/2}}$. Moreover,a gap of size $T$ between $J_j$ and $\ug$ can be created in the interval $[\ell_j+1,r_j-1]$. This event is denoted by $\mathsf{A}_j$.

Next, we apply induction on the number of curves. Take $j\in [2,k]_{\mathbb{Z}}$. We write $\mathsf{F}_{j+1}$ for the event that $J_1, J_2,\dots, J_{j-1}$ stay ordered and that $J_{j-1}\geq \ug+T $ in $[\ell_{j-1}+1,r_{j-1}-1]$. Suppose that $\mathbb{P}(\mathsf{F}_{j-1})\geq e^{-T^{5/2}}$. Because the $j$-th curve $J_{j}$ is independent of other curves, $\mathbb{P}(\mathsf{F}_{j-1}\cap\mathsf{A}_j)\geq e^{-T^{5/2}}$. To get bounds on $\mathbb{P}_J(\mathsf{F}_j)$, we need to make sure $J_{j}$ does not raise too high and does not touch $J_{j-1}$. 

To do so, we apply a stopping domain argument. Consider the leftest and the rightest locations where $J_{j-1}=J_j+2^{-1}T$. Denote them by $\mathfrak{l}''_j$ and $\mathfrak{r}''_j$ respectively. At these two points, we know $J_1\geq J_2\geq \dots \geq J_{j-1}>J_j>\ug.$ Furthermore, $J_{j-1}-J_j=2^{-1}T$ and $J_j-\ug\geq 2^{-1}T$. Applying the strong $\mathbf{H}_t$-Brownian Gibbs property, we resample $J_1,\dots J_j$ in $[\mathfrak{l}''_j, \mathfrak{r}''_j]$. Because the boundary data is now well spaced, we can show that the $J_1,\dots, J_j$, with probability $e^{-T^{5/2}}$, remain ordered and lie beyond $\ug$ in $[\mathfrak{l}''_j, \mathfrak{r}''_j]$. This finishes the proof.\\

\begin{lemma}\label{lem:Westbdd}
There exists $D=D(k,s)$ such that for all $t\geq 1$, it holds that
\begin{align*}
\mathbbm{1}_{\mathsf{NoTouch}}\cdot W_{\textup{rest}}(;f_{\textup{J}})\geq D^{-1}e^{-DT }\cdot \mathbbm{1}_{\mathsf{NoTouch}}.
\end{align*}
\end{lemma}

\begin{proof}
From the definition of $W_{\textup{rest}}$ in \eqref{def:Wrest} and $\mathbf{H}_t^{i,j}\geq 0$, we have
\begin{equation*}
\begin{split}
W_{\textup{rest}}(\mathcal{L};f_{\textup{J}}) \geq  &\prod_{j=1}^{k-1} \exp\left( -\int_{[\ell_j+d',\ell_{j+1}]\cup [r_{j+1},r_j-d']}  \mathbf{H}_t(\ug (x)-\mathcal{L}_j(x))\, dx  \right)\\
&\times \prod_{j=1}^{k-1} \exp\left( -\int_{[\ell_{j+1},r_{j+1}]} \mathbf{H}_t(\mathcal{L}_{j+1}(x)-\mathcal{L}_j(x))\, dx \right)\\
&\times \exp\left( -\int_{[\ell_{k}+d',r_k-d']}  \mathbf{H}_t(\ug (x)-\mathcal{L}_k(x))\, dx \right).
\end{split}
\end{equation*}
When $\mathsf{NoTouch}$ occurs, the integrands above are all bounded from below by $-1$. Hence
\begin{align*}
W_{\textup{rest}}(\mathcal{L};f_{\textup{J}}) \geq e^{-k(r_1-\ell_1)}\geq e^{-2k(k+1)T}.
\end{align*}
\end{proof}

\begin{proposition}\label{pro:notouch}
There exists a constant $D=D(k,s)$ such that for $E$ large enough, it holds that
\begin{align*}
\mathbb{P}_{J}(\mathsf{NoTouch})\geq D^{-1}e^{-DT^{5/2}}.
\end{align*}
\end{proposition}

\begin{proof}[{\bf Proof of Proposition \ref{clm:denominator}}]
From Lemma~\ref{lem:Westbdd} and Proposition \ref{pro:notouch} 
\begin{align*}
\mathbb{E}_J[W_{\textup{rest}}(J;f_{\textup{J}}) ]\geq  \mathbb{P}_J(\mathsf{NoTouch} )\times D^{-1}e^{-DT } \geq D^{-1}e^{-DT^{5/2}}.  
\end{align*}
\end{proof}

The rest of the section is devoted to prove Proposition \ref{pro:notouch}. In subsection~\ref{subsec:single}, we deal with a single curve $J_j$ and show that $\mathsf{A}_j$ occurs with probability at least $e^{-T^{5/2}}$. In subsection~\ref{subsec:multiple}, we show that, provided the boundary data is well separated, $J$ remains ordered and lying beyond $\ug$ under resampling. Lastly, we combine the two and finish the induction argument in subsection~\ref{subsec:induction}.\\

\subsection{Single curve}\label{subsec:single}
In this subsection, we focus on a single curve $J_j$ and prove Proposition \ref{pro:A}. Roughly speaking, Proposition \ref{pro:A} says the probability that $J_j$ is larger than $\ug$ in $[\ell_j+d',r_j-d']$ is at least $e^{-T^{5/2}}$. \\

We start by defining $\mathsf{A}_j$. Fix $j\in [1,k]_{\mathbb{Z}}$ and let $\mathsf{A}_j$ be the event that
\begin{itemize}
\item $J_j(x)> \ug (x)\ \textup{for all}\ x\in [\ell_j+d',r_j-d']$,
\item $J_j(x)> \textup{Tent}_{j+1}(x)+16k(k+1-j)sT\ \textup{for all}\ x\in [\ell_j+1,r_j-1 ]$.
\end{itemize}
\begin{proposition}\label{pro:A}
There exists $D =D (k)$ such that for $E$ large enough, we have
\begin{equation*}
\mathbb{P}_{ J }(\mathsf{A}_j)\geq D^{-1}e^{-DT^{5/2}}.
\end{equation*}
\end{proposition}

We prove Proposition \ref{pro:A} by gradually lift the curve $J_j$. In below we define the events which record the lifting process.

Let $\mathsf{A}_{j,1}$ be the event such that
\begin{itemize}
\item $J_j (\ell_j+d')> f_j(\ell_j)+kT^{1/2},J_j(r_j-d')> f_j(r_j)+kT^{1/2}$,
\item $ J_j (\ell_j+1)> f_j(\ell_j)+20k(k+1-j)sT$,
\item $J_j(r_j-1)> f_j(r_j)+20k(k+1-j)sT$.
\end{itemize}

Let $\mathsf{A}_{j,2}$ be the event such that
\begin{itemize}
\item $J_j(x)> f_{j+1}(x)\ \textup{for all}\ x\in [\ell_j+d',\ell_j+1]\cup [r_j-1,r_j-d']$,
\item $  J_j(\ell_j+1)> \textup{Tent}_{j+1}(\ell_j+1)+18k(k+1-j)sT $
\item  $ J_j(r_j-1)> \textup{Tent}_{j+1}(r_j-1)+18k(k+1-j)sT $.
\end{itemize}

Let $\mathsf{A}_{j,3}$ be the event such that
\begin{itemize}
\item $J_j(x)> \textup{Tent}_{j+1}(x)+16k(k+1-j) sT\ \textup{for all}\ x\in [\ell_j+1,r_j-1] $.
\end{itemize}

Recall that $J_j$ is sampled in $[\ell_j,r_j]$ with boundary given by $J_j(\ell_j)=f(\ell_j)$ and $J_j(r_j)=f(r_j)$. In the event $\mathsf{A}_{j,1}$, $J_j$ goes up by $T^{1/2}$ within distance $d'$ and then goes up by $sT$ within distance $1$. We then resample $J_j$ in $[\ell_j+d',\ell_j+1]$ and $[r_j-1,r_j-d']$. Conditioned on $\mathsf{A}_{j,1}$ the boundary data is high enough to ensure $\mathsf{A}_{j,2}$ occurs with probability $D^{-1}$. After that having $\mathsf{A}_{j,2}$, we resample $J_j$ in $[\ell_j+1,r_j-1]$. Applying Theorem \ref{thm:separation} ensures that conditioned on $\mathsf{A}_{j,2}$,  $\mathsf{A}_{j,3}$ occurs with probability $e^{-T^{5/2}}$. This argument is entailed in the following three lemmas. To simplify the notation, we denote $\mathbb{P}^{\textup{dom}(\bar{\ell},\bar{r}),f}_{\free}$ and $\mathbb{E}^{\textup{dom}(\bar{\ell},\bar{r}),f}_{\free}$ by $\mathbb{P}_{\free}$ and $\mathbb{E}_{\free}$ respectively.

\begin{lemma}\label{lem:A1}
There exists $D =D (k,s)$ such that  
\begin{align*}
\mathbb{P}_{\free }(\mathsf{A}_{j,1})\geq D^{-1}e^{-DT^{5/2}}.
\end{align*}
\end{lemma} 

\begin{lemma}\label{lem:A2}
There exists $D =D (k,s)$ such that
\begin{align*}
\mathbb{P}_{\free }(\mathsf{A}_{j,2}\,|\, \mathsf{A}_{j,1})\geq D^{-1}.
\end{align*}
\end{lemma}
\begin{lemma}\label{lem:A3A2}
There exists $D =D (k,s)$ such that for $E$ large enough, it holds that
\begin{align*}
\mathbb{P}_{J}( \mathsf{A}_{j,3}\,|\,\mathsf{A}_{j,2})\geq D^{-1}e^{-DT^{5/2}}.
\end{align*}
\end{lemma}
\begin{proof}[{\bf Proof of Proposition \ref{pro:A}}]
From \eqref{Tentj_lowerbound}, we check $\mathsf{A}_{j,2}\cap\mathsf{A}_{j,3}\subset \mathsf{A}_j$.
By Lemmas \ref{lem:A1} and \ref{lem:A2},
\begin{align*}
\mathbb{P}_{\free }( \mathsf{A}_{j,2} )\geq \mathbb{P}_{\free }( \mathsf{A}_{j,2}\,|\,\mathsf{A}_{j,1} )\times \mathbb{P}_{\free }( \mathsf{A}_{j,1} ) \geq D^{-1}e^{-D^{5/2}}.
\end{align*}

By the stochastic monotonicity Lemma~\ref{monotonicity1},
\begin{align*}
\mathbb{P}_{J}( \mathsf{A}_{j,2} )\geq \mathbb{P}_{\free }( \mathsf{A}_{j,2} )\geq D^{-1}e^{-D^{5/2}}.
\end{align*}

With Lemma \ref{lem:A3A2}, we conclude that
\begin{align*}
\mathbb{P}_{J}(\mathsf{A}_{j})\geq \mathbb{P}_{J}(\mathsf{A}_{j,2}\cap\mathsf{A}_{j,3})= \mathbb{P}_{J}( \mathsf{A}_{j,3}|\mathsf{A}_{j,2})\times \mathbb{P}_{J}( \mathsf{A}_{j,2} )\geq D^{-1}e^{-DT^{5/2}}.
\end{align*}
\end{proof}

In the rest of this subsection, we prove Lemmas~\ref{lem:A1}, \ref{lem:A2} and  \ref{lem:A3A2}.

\begin{proof}[{\bf Proof of Lemma~\ref{lem:A1}}]
The first condition in $\mathsf{A}_{j,1}$ requires $J_j$ to raise by an amount $T^{1/2}$ within a distance $d'=T^{-3/2}.$ This event occurs with probability about $e^{-T^{5/2}}.$ The second and the third condition requires $J_j$ to raise by an amount $T$ within a distance $1$. This event occurs with probability about $e^{-T^{2}}.$ In below we carry out this calculation.\\

Under the law $\mathbb{P}_{\free}$, $B_j(\ell_j+d')$ is a Gaussian random variable with mean and variance\\[-0.25cm]
\begin{align*}
m=\frac{(r_j-\ell_j-d')f_j(\ell_j)+d' f_j(r_j)}{r_j-\ell_j},\ \sigma^2=\frac{(r_j-\ell_j-d')d'}{r_j-\ell_j}.
\end{align*}
Let $N$ be a Gaussian random variable with mean $0$ and variance $1$. It holds that\\[-0.25cm]
\begin{align*}
&\mathbb{P}_{\free }(B_j(\ell_j+d')> f_j(\ell_j)+kT^{1/2})=\mathbb{P}(N>\sigma^{-1}(f_j(\ell_j)-m)+ \sigma^{-1} kT^{1/2} ).
\end{align*}

From Condition \textbf{C1} and Remark \ref{rmk:U}, it holds that $|\ell_j+(k+j-1)T|,|r_j-(k+j-1)T|\leq 10^{-1}T.$ Recall that $d'=T^{-3/2}$. Hence $\sigma^{-1} \leq 2^{1/2}T^{3/4}$ and then
$$\sigma^{-1}kT\leq DT^{5/4}.$$ 
From Condition \textbf{C2} , we have
$\left|\frac{f(\ell_j)-f(r_j)}{\ell_j-r_j}\right|\leq DT.$ Thus
\begin{align*}
\sigma^{-1}| f_j(\ell_j)-m | \leq DT^{1/4}.
\end{align*} 
 As a result,  
\begin{align*}
\mathbb{P}_{\free }(B_j(\ell_j+d')> f_j(\ell_j)+kT^{1/2}) \geq \mathbb{P}(N\geq DT^{5/4})\geq D^{-1}e^{-DT^{5/2}}.
\end{align*}

Conditioned on $B_j(\ell_j+d')=y$, $B_j(\ell_j+1)$ is a normal distribution with mean and variance
\begin{align*}
m=\frac{(r_j-\ell_j-1 ) y  +(1-d') f_j(r_j) }{ r_j-\ell_j- d'} ,\ \sigma^2=\frac{(r_j-\ell_j-1 ) (1-d') }{ r_j-\ell_j-d'}.
\end{align*}
Thus
\begin{align*}
&\mathbb{P}_{\free }(B_j(\ell_j+1')\geq f_j(\ell_j)+20k(k+1-j) sT\, |\,  B_j(\ell_j+d' )=y )\\
=&\mathbb{P}(N>\sigma^{-1}(f_j(\ell_j)-m)+20\sigma^{-1} k(k+1-j)sT)
\end{align*}
Because $\sigma \approx 1$,
\begin{align*}
20\sigma^{-1} k(k+1-j)sT\leq DT.
\end{align*}
When $y\geq f_j(\ell_j)+kT^{1/2}$,
\begin{align*}
 f_j(\ell_j)-m \leq \frac{(1-d')(f_j(r_j)-f_j(\ell_j))}{r_j-\ell_j-d'}-\frac{(r_j-\ell_j-1) kT^{1/2}}{r_j-\ell_j-d'}\leq DT.
\end{align*}
As a result,
\begin{align*}
&\mathbbm{1}\{y\geq f_j(\ell_j)+kT^{1/2}\}\times \mathbb{P}_{\free }(B_j(\ell_j+1')> f_j(\ell_j)+20k(k+1-j) sT\, |\,  B_j(\ell_j+d' )=y )\\
\leq & \mathbbm{1}\{y\geq f_j(\ell_j)+kT^{1/2}\}\times D^{-1}e^{-DT^2}.
\end{align*}

This finishes the proof for $\ell_j+d'$ and $\ell_j+1$. The one for $r_j-d'$ and $r_j-1$ is similar and we omit the detail.
\end{proof}

\begin{proof}[{\bf Proof of Lemma \ref{lem:A2}}]
We first bound $f_{j+1}$ by a linear function. By \eqref{def:g2} and \eqref{ug_bdd}, it holds that for $x\in (\ell_j,\ell_j+1]$,
\begin{align*}
f_{j+1}(x)=\ug(x) \leq  f_{j+1}(\ell_j)+(x-\ell_j)(k+1-j)T.
\end{align*}
By Conditions \textbf{C4}, we obtain
\begin{align*}
f_{j+1}(x)=\ug(x) \leq f_j(\ell_j)+3\Delta+(x-\ell_j)(k+1-j)T=:L(x).
\end{align*} 

Next, we show that when $\mathsf{A}_{j,1}$ occurs, it holds that
\begin{align*}
& L(\ell_j+d')+2^{-1}kT^{1/2} < J_j(\ell_j+d'),\\
& L(\ell_j+1)+18 k(k+1-j) sT <    J_j(\ell_j+1).
\end{align*}
They can be deduced by
\begin{align*}
 L(\ell_j+d')= f_j(\ell_j)+3\Delta+ (k+1-j)T^{-1/2} \leq & f_j(\ell_j)+2^{-1}kT^{1/2}\\
 <  &J_j(\ell_j+d')-2^{-1}kT^{1/2},
\end{align*}
and
\begin{align*}
L(\ell_j+1)= f_j(\ell_j)+3\Delta+ (k+1-j)T \leq &f_j(\ell_j)+2k(k+1-j)T\\
<  &J_j(\ell_j+1)-18 k(k+1-j) sT.
\end{align*}
In particular, the second condition in $\mathsf{A}_{j,2}$ is ensured by $\mathsf{A}_{j,1}$.\\

Applying a similar argument near $r_j$, the third condition in $\mathsf{A}_{j,2}$ is also ensured by $\mathsf{A}_{j,1}$. Furthermore,
\begin{align*}
\mathsf{A}_{j,1}\cap\left\{\left|J_j^{[\ell_j+d',\ell_j+1]}\right|\leq 2^{-1}kT^{1/2} \right\}\cap\left\{\left|J_j^{[r_j-1 ,r_j-d']}\right|\leq 2^{-1}kT^{1/2} \right\}\subset \mathsf{A}_{j,2}.
\end{align*}
We conclude that
\begin{align*}
 \mathbb{P}_{\free }(\mathsf{A}_{j,2}\,|\,\mathsf{A}_{j,1})\geq D^{-1}.
\end{align*}
\end{proof}

\begin{proof}[{\bf Proof of Lemma \ref{lem:A3A2}}] We first replace $\mathsf{A}_{j,2}$ by $\mathsf{A}_{j,4}$, which fixes the value of $J_j(x)$ at $x=\ell_j+1$ and $x=r_j-1$. Let
\begin{align*}
\mathsf{A}_{j,4}\coloneqq \{    J_j(x)=\textup{Tent}_{j+1}(x)+18k(k+1-j)sT\ \textup{for}\ x\in\{\ell_j+1,r_j-1\}   \}.
\end{align*} 
From the stochastic monotonicity Lemma~\ref{monotonicity1}, it holds that
\begin{align*}
\mathbb{P}_{J}(\mathsf{A}_{j,3}\,|\,\mathsf{A}_{j,2})\geq\mathbb{P}_{J}(\mathsf{A}_{j,3}\,|\,\mathsf{A}_{j,4}).
\end{align*}
It remains to prove that
\begin{align*}
\mathbb{P}_{J}(\mathsf{A}_{j,3}\,|\,\mathsf{A}_{j,4})\geq D^{-1}e^{-DT^{5/2}}.
\end{align*}

Consider the event

\begin{equation*}
\begin{split}
\mathsf{L}_j \coloneqq \{ \mathcal{L}_{j}(p)\leq &\textup{Tent}_{j+1}(p)- \Delta_k\  \textup{for some}\  p\in P_{j+1}\cap [\ell_j+1,r_j-1] \}.
\end{split}
\end{equation*}
This event represents the scenario that $J_j$ goes below $\ug$ by an amount $\Delta_k$ at some point in $P_{j+1}\cap [\ell_j+1,r_j-1]$.\\

\noindent\underline{Step 1}: We show in this step that conditioned on $\mathsf{A}_{j,4}$, $\mathsf{L}_j$ is unlikely.
\begin{align*}
\mathbb{P}_J(\mathsf{L}_j  |\mathsf{A}_{j,4})\leq 2^{-1}.
\end{align*}
We begin by controlling the local fluctuation of $J_j$. Recall that we write $\mathbb{P}_{\free}$ and $\mathbb{E}_{\free}$ for $\mathbb{P}^{\textup{dom}(\bar{\ell},\bar{r}),f}_{\free}$ and $\mathbb{E}^{\textup{dom}(\bar{\ell},\bar{r}),f}_{\free}$ respectively. From \eqref{Tentj_lowerbound} and \eqref{Tentj_slope}, we can apply Lemma~\ref{lem:midpt} to get 
\begin{align*}
\mathbb{P}_ {\free}(B_{j}(x)\geq \ug(x)\ \textup{for}\ x\in [\ell_j+1,r_j-1] \,|\,\mathsf{A}_{j,4})\geq D^{-1}e^{-DT^3}.
\end{align*}
When such event occurs, $W^{j,[\ell_j+1,r_j-1]}_{\textup{Jump} }$ is bounded from below by some constant $D^{-1}$.  Therefore,
$$\mathbb{E}_{\free}\left[W^{j,[\ell_j+1,r_j-1]}_{\textup{Jump} }\,|\,\mathsf{A}_{j,4} \right]\geq D^{-1}e^{-DT^3}.$$

Let 
\begin{equation*} 
\mathsf{Flct}_j \coloneqq \{| {J}_{j}(x)- {J}_j(y)|\leq \Delta\ \textup{for all}\ x,y\in [\ell_j+1 ,r_j-1]\ \textup{with}\ |x-y|\leq d \}.
\end{equation*}
When $\mathsf{A}_{j,4}$ occurs, $ J_j(x)=\textup{Tent}_{j+1}(x)+18k(k+1-j)sT $ for $x\in \{\ell_j+1,r_j-1\}$. From \eqref{Tentj_slope}, we have  
\begin{align*}
\left| \frac{J_j(\ell_j+1)-J_j(r_j-1) }{r_j-\ell_j-2}\right|=\left| \frac{\textup{Tent}_{j+1}(\ell_j+1)-\textup{Tent}_{j+1}(r_j-1) }{r_j-\ell_j-2}\right|\leq 2kT.
\end{align*}
Hence
\begin{align*}
\mathbb{P}_{\free}  (\mathsf{Flct}^{\textup{c}}_j\,|\, \mathsf{A}_{j,4} )\leq &\mathbb{P}_{\free}(\omega_{d,[\ell,r]}(B^{[\ell,r]})\geq \Delta-2kdT)\\
\leq &\mathbb{P}_{\free}(\omega_{d,[\ell,r]}(B^{[\ell,r]})\geq 2^{-1} \Delta)\leq De^{-D^{-1}T^4}.
\end{align*}
Here
\begin{align*}
\omega_{d,[\ell,r]}(B^{[\ell,r]})\coloneqq \max_{x,y\in [\ell,r],\ |x-y|\leq d}\left|B^{[\ell,r]}(x)-B^{[\ell,r]}(y)  \right|.
\end{align*}

As a result, for $E$ large enough, we have
\begin{align*}
\mathbb{P}_{J}  ( \mathsf{Flct}^{\textup{c}}_j\,|\, \mathsf{A}_{j,4}  )\leq De^{DT^3-D^{-1}T^4}\leq 4^{-1}.
\end{align*} 

Next, we show that $\mathsf{L}_j\cap  \mathsf{Flct}_j $ is unlikely. From Condition \textbf{C3},\ $\mathsf{L}_j $ and $  \mathsf{Flct}_j $ imply there exists an interval in $\textup{Int}_j\cap [\ell_j+1,r_j-1]$ with length $d$ where
\begin{align*}
J_j(x)\leq \ug(x)-k\Delta-k\log k.
\end{align*}
Together with 
\begin{align*}
d\, \textbf{H}_{j,k+1}(k\Delta+k\log k)\geq d e^{\Delta }=64T^4(\log T)^2,
\end{align*}
it holds that
\begin{align*}
\mathbbm{1}_{ \mathsf{L}_j \cap  \mathsf{Flct}_j  } \times W^{j,[\ell_j+1,r_j-1]}_{\textup{jump}}\leq  e^{-64T^4(\log T)^2}\times \mathbbm{1}_{ \mathsf{L}_j \cap  \mathsf{Flct}_j  } .
\end{align*}
As a result, for $E$ large enough, it holds that
\begin{align*}
&\mathbb{P}_{J}  \left(\mathsf{L}_j \cap \mathsf{Flct}_j \,|\,\mathsf{A}_{j,4}\right)
\leq  \mathbb{E}_{\free} \left[ W^{j, [\ell_{j}+1,r_j-1]}_{\textup{jump}} \,|\,\mathsf{A}_{j,4} \right]^{-1}\times e^{-64T^4(\log T)^2}\\
\leq & De^{DT^3-64T^4(\log T)^2}\leq 4^{-1}.
\end{align*}

Putting the above estimates together, we conclude that
\begin{align*}
&\mathbb{P}_{J}  (\mathsf{L}_j\,|\,\mathsf{A}_{j,4} )\leq \mathbb{P}_{J}  (\mathsf{L}_j \cap\mathsf{Flct}_j\,|\,\mathsf{A}_{j,4})+\mathbb{P}_{J}  ( \mathsf{Flct}^{\textup{c}}_j\,|\,\mathsf{A}_{j,4})\leq 2^{-1}.
\end{align*}

\vspace{0.05cm}

\noindent \underline{Step 2}: In this step, we want to prove
\begin{align*}
 \mathbb{P}_\free ( \mathsf{A}_{j,3}\,|\,\mathsf{A}_{j,4}\cap \mathsf{L}^{\textup{c}}_j) \geq  D^{-1}e^{-DT^{5/2}}. 
\end{align*}	 

Let 

$$\mathsf{A}_{j,3}'\coloneqq \{J_j(p)>\textup{Tent}_{j+1}(p)+17k(k+1-j)sT\ \textup{for}\ p\in  P_{j+1}\cap (\ell_j+1,r_j-1)\}.$$

To apply Theorem \ref{thm:separation}, we rewrite the following objects. Set $\breve{\ell}=\ell_j+1,\breve{r}=r_j-1$,
\begin{align*}
\breve{f}(x)\coloneqq\textup{Tent}_{j+1}(x)-\Delta_k ,\ \breve{\mathcal{L}}(x):={J}_j(x),\ \breve{P}=P_{j+1}\cap (\breve{\ell},\breve{r})
\end{align*}
and
\begin{align*}
 \breve{a}^-=\textup{Tent}_{j+1} (\breve{\ell})+18k(k+1-j)sT,\  \breve{a}^+=\textup{Tent}_{j+1} (\breve{r})+18k(k+1-j)sT.
\end{align*} 
The law of $\breve{\mathcal{L}}$ in $[\breve{\ell},\breve{r}]$ under $\mathbb{P}_\free (\,|\,\mathsf{A}_{j,4})$ is given by $\mathbb{P}^{1,1,(\breve{\ell},\breve{r}),\breve{a}^-,\breve{a}^+}_\free$. Furthermore, it can easily verified that
 \begin{align*}
\mathsf{L}^{\textup{c}}_j=&\{\breve{\mathcal{L}}(p)>\breve{f}(p)  \ \textup{for}\ p\in  \breve{P}\},\\
\mathsf{A}_{j,3}'=& \{\breve{\mathcal{L}}(p)>\breve{f}(p)+17k(k+1-j)sT+\Delta_k\ \textup{for}\ p\in  \breve{P}\}.
\end{align*}

Because
$$\breve{a}^--\breve{f}(\ell_j)=\breve{a}^+-\breve{f}(r_j)=18k(k+1-j)sT+\Delta_k,$$  
we can applying Theorem \ref{thm:separation} to conclude that
\begin{align*}
 \mathbb{P}_\free ( \mathsf{A}_{j,3}' \,|\,\mathsf{A}_{j,4}\cap \mathsf{L}^{\textup{c}}_j) \geq  D^{-1}e^{-DT^{5/2}}. 
\end{align*}

Resample $J_j$ between pole points in $P_{j+1}\cap [\ell_j+1,r_j-1]$, we obtain
\begin{align*}
\mathbb{P}_{\free}  (\mathsf{A}_{j,3}  \,|\,\mathsf{A}_{j,4}\cap\mathsf{L}^{\textup{c}}_j )\geq\mathbb{P}_{\free}  (\mathsf{A}_{j,3}  \,|\,\mathsf{A}_{j,4}\cap\mathsf{L}^{\textup{c}}_j\cap \mathsf{A}'_{j,3} )\times \mathbb{P}_{\free}  (\mathsf{A}'_{j,3}  \,|\,\mathsf{A}_{j,4}\cap\mathsf{L}^{\textup{c}}_j ) \geq D^{-1}e^{-DT^{5/2}}.
\end{align*}

\vspace{0.05cm}

\noindent \underline{Step 3}: We finish the proof in this step. When $\mathsf{A}_{j,3} $ occurs, $W^{j,[\ell_j+1,r_j-1]}_{\textup{Jump} }\geq D^{-1}$. Hence

\begin{align*}
\mathbb{P}_{J}  (\mathsf{A}_{j,3}  \,|\,\mathsf{A}_{j,4}\cap\mathsf{L}^{\textup{c}}_j )\geq D^{-1}e^{-DT^{5/2}}.
\end{align*}	

In conclusion,
\begin{align*}
\mathbb{P}_{J} ( \mathsf{A}_{j,3} \, |\,\mathsf{A}_{j,4} )\geq&  \mathbb{P}_{J} ( \mathsf{A}_{j,3}   \,|\,\mathsf{A}_{j,4} \cap\mathsf{L }^{\textup{c}}_j )\times  \mathbb{P}_{J}  (\mathsf{L }^{\textup{c}}_j |\mathsf{A}_{j,4} )\geq  D^{-1}e^{-DT^{5/2}}.
\end{align*}
\end{proof}

\subsection{Curve separation over the resampling interval}\label{subsec:multiple}
In this subsection, we prove Proposition \ref{pro:seperation}. Roughly speaking, Proposition \ref{pro:seperation} says the following. Suppose at two points $\ell$ and $r$, the jump ensemble $J$ is ordered and lie above $\ug $. Then the probability that $J$ remains ordered and lie above $\ug $ in $[\ell,r]$ is at least $e^{-T^{5/2}}$. \\

We begin with introducing the notation in order to resample $J$ in the interval $[\ell,r]$.  Fix $j\in [1,k]_{\mathbb{Z}}$. Given $\ell_j\leq \ell<r\leq r_j$ and $a^{\pm} =(a^-_1,\dots,a^-_k, a^+_1,\dots,a^+_k)$, define\\[-0.25cm]
\begin{align*}
\frac{{d}\mathbb{P}^{1,j,[\ell,r],a^+,a^-}_{\textup{jump}}}{{d}\mathbb{P}^{1,j,[\ell,r],a^+,a^-}_{\free}}(\mathcal{L})\propto \prod_{i=1}^j W^{i,[\ell,r]}_{\textup{jump}}(\mathcal{L}_i;f_\textup{J}).
\end{align*}

Next, we prescribe the boundary condition which is well-spaced. Let $bS_j(\ell,r)$ be the collection of $(a^+,a^-) \in\mathbb{R}^{j}\times \mathbb{R}^{j}$ which satisfies
\begin{align}\label{equ:bS1}
 a^\pm_i-a^\pm_{i+1}\geq   (k+2-j)T^{1/2}\ \textup{   for all}\  i\in [1,j-1]_{\mathbb{Z}},
\end{align}
\begin{align}\label{equ:bS2}
 a^-_j\geq  \textup{Tent}_{j+1}(\ell)+8k(2k+3-2j)sT\ \textup{and}\  a^+_j \geq \textup{Tent}_{j+1}(r)+8k(2k+3-2j)sT,
\end{align}
and
\begin{align}\label{equ:bS3}
  a^\pm_1\leq  jT^2. 
\end{align}
Let $ \underline{a}^\pm$  be the smallest member is $bS_j(\ell,r)$. That is, 
\begin{align*}
\underline{a}^{\pm}_i-\underline{a}^{\pm}_{i+1}=&(k+2-j)T^{1/2},\ i\in [1,j-1]_{\mathbb{Z}},\\
\underline{a}^-_j=&\textup{Tent}_{j+1}(\ell)+8k(2k+3-2j)sT,\\
\underline{a}^+_j=&\textup{Tent}_{j+1}(r)+8k(2k+3-2j)sT.
\end{align*}
 

The following two events capture the scenario that the curves are well-separated in $[\ell,r]$.

\begin{equation}\label{def:Sj}
\begin{split}
\mathsf{S}_j(\ell,r)\coloneqq \{ \mathcal{L}_i(x)>  \mathcal{L}_{i+1}(x)+(k+1-j)T^{1/2}\ \textup{for all}\ i\in [1,j-1]_{\mathbb{Z}}\ \textup{and}\ x\in [\ell,r]  \}\\
\cap\{\mathcal{L}_j(x)>  \textup{Tent}_{j+1}(x)+16k(k+1-j)T \ \textup{for all} \ x\in [\ell,r] \}.
\end{split}
\end{equation}
   
\begin{equation}\label{def:Upj}
\begin{split}
\mathsf{Up}_j(\ell,r)\coloneqq   \{ \mathcal{L}_1(x)<  (j+1)T^2 \ \textup{for all} \ x\in [\ell,r]\}.
\end{split}
\end{equation}

\begin{proposition}\label{pro:seperation}
There exists a constant $D=D(k,s)$ such that for $E$ large enough, the following statement holds. For all $j\in [1,k]_{\mathbb{Z}}$, $\ell,r$ in $[\ell_j,r_j]$ and $(a^+,a^-)\in bS_j(\ell,r)$, we have
\begin{equation*}
\mathbb{P}^{1,j,[\ell,r],a^-,a^+}_{\textup{jump}} (\mathsf{S}_j(\ell,r)\cap\mathsf{Up}_j(\ell,r))\geq D^{-1}e^{-DT^{5/2}}.
\end{equation*}
\end{proposition} 
\begin{proof} The proof involves three steps.\\

\vspace{0.05cm}

\noindent \underline{Step 1}: We first show that
\begin{equation*}
\mathbb{P}^{1,j,[\ell,r],a^-,a^+}_{\textup{jump}}  ( \mathsf{Up}^{\textup{c}}_j(\ell,r)  )\leq D e^{-D^{-1}T^{3}}.
\end{equation*}

This can be achieved by a simple stopping domain argument. Consider the intervals
\begin{equation*}
\begin{split}
I_-\coloneqq  \{x\in [\ell,r]\ |\ \mathcal{L}_1(y)< (j+2^{-1})T^2\ \textup{for all}\ y\in [\ell,x]\},\\
I_+\coloneqq  \{x\in [\ell,r]\ |\ \mathcal{L}_1(y)< (j+2^{-1})T^2\ \textup{for all}\ y\in [x,r] \}.
\end{split}
\end{equation*}
Define 
\begin{equation*}
\mathfrak{l}' \coloneqq\left\{ \begin{array}{cc}
 \sup I_- & I_-\neq \phi,\\
 r & I_-=\phi.
\end{array}
\right.
\end{equation*}
\begin{equation*}
\mathfrak{r}' \coloneqq\left\{ \begin{array}{cc}
 \inf I_+ & I_+\neq \phi,\\
 \ell & I_+=\phi.
\end{array}
\right.
\end{equation*}

The last condition in $bS_j(\ell,r)$ ensures $I_\pm$ are non-empty. Moreover,
\begin{align*}
\mathsf{Up}_j^{\textup{c}}(\ell,r) \subset \{\ell< \mathfrak{l}' <\mathfrak{r}' <r \} .
\end{align*}

Because $(j+2^{-1})T^2$ is much larger than $\ug(x)$, we have
\begin{align*}
 \mathbb{P}^{1,1,[\mathfrak{l}'  ,\mathfrak{r}' ],(j+2^{-1})T^2,(j+2^{-1})T^2}_{\free }\left(B_1(x)\geq \ug (x)\ \textup{for all}\ x\in [\mathfrak{l}'  ,\mathfrak{r}' ] \right)\geq D^{-1} .
\end{align*} 	
And then
\begin{align*}
 \mathbb{E}^{1,1,[\mathfrak{l}'  ,\mathfrak{r}' ],(j+2^{-1})T^2,(j+2^{-1})T^2}_{\free }\left[W^{1,[\mathfrak{l}'  ,\mathfrak{r}' ]}_{\textup{jump}} \right]\geq D^{-1} .
\end{align*} 	
Therefore,
\begin{align*}
 &\mathbb{P}^{1,j,[\mathfrak{l}' ,\mathfrak{r}' ],(j+2^{-1})T^2,(j+2^{-1})T^2}_{\textup{jump}}  ( \mathsf{Up}^{\textup{c}}_j(\mathfrak{l}',\mathfrak{r}' )  )\\
 = & \mathbb{E}^{1,1,[\mathfrak{l}'  ,\mathfrak{r}' ],(j+2^{-1})T^2,(j+2^{-1})T^2}_{\free }\left[W^{1,[\mathfrak{l}'  ,\mathfrak{r}' ]}_{\textup{jump}} \right]^{-1}\\
 &\times \mathbb{P}^{1,j,[\mathfrak{l}' ,\mathfrak{r}' ],(j+2^{-1})T^2,(j+2^{-1})T^2}_{\free}  ( \mathsf{Up}^{\textup{c}}_j(\mathfrak{l}',\mathfrak{r}' )  )\leq  De^{-D^{-1}T^3}.
\end{align*}
Through the strong Markov property,
\begin{align*}
\mathbb{P}^{1,j,[\ell,r],a^-,a^+}_{\textup{jump}}   ( \mathsf{Up}^{\textup{c}}_j(\ell,r)  )\leq& \mathbb{P}^{1,j,[\ell,r],a^-,a^+}_{\textup{jump}}   ( \mathsf{Up}^{\textup{c}}_j (\mathfrak{l}',\mathfrak{r}')\cap\{ \ell< \mathfrak{l}' <\mathfrak{r}' <r  \}   )\\
  \leq& De^{-D^{-1}T^3}. 
\end{align*}
This finishes Step 1.\\

\vspace*{0.05cm}

\noindent \underline{Step 2}:
Define the event 
\begin{equation*}
\begin{split}
\mathsf{Low}_j(\ell,r) \coloneqq\bigcup_{i=1}^j \{ \mathcal{L}_{i}(p)\leq &\textup{Tent}_{j+1}(p)- \Delta_k\  \textup{for some}\  p\in P_{j+1}\cap [\ell,r] \}.
\end{split}
\end{equation*}
In Step 2, we show that
\begin{align*}
\mathbb{P}^{1,j,[\ell,r],a^-,a^+}_{\textup{jump}}  (\mathsf{Low}_j(\ell,r) )\leq 2^{-1}. 
\end{align*}
By the stochastic monotonicity Lemma~\ref{monotonicity1}, it suffices to prove
\begin{align*}
\mathbb{P}^{1,j,[\ell,r],\underline{a}^-,\underline{a}^+ }_{\textup{jump}}  (\mathsf{Low}_j(\ell,r) )\leq 2^{-1}. 
\end{align*}

Fix $i\in [1,j]_{\mathbb{Z}}$. Because $\underline{a}^-_i\geq \textup{Tent}_{j+1}(\ell)+8k(2k+3-2j)sT$, $\underline{a}^+_i\geq \textup{Tent}_{j+1}(r)+8k(2k+3-2j)sT$ and \eqref{Tentj_slope}, it's straightforward to show that
\begin{align*}
\mathbb{P}^{i,i,[\ell,r],\underline{a}^-_i ,\underline{a}^+_i}_{\free} ( B_i(x)\geq \textup{Tent}_{j+1}(x)+7k(2k+3-2j)sT\ \textup{for all}\ x\in [\ell,r]  )\geq D^{-1}e^{-DT^3}.
\end{align*}
This can be done by applying Lemma \ref{lem:midpt} with $x_1=\ell$, $x_2=r$ and $m=2kT$. Together with \eqref{Tentj_lowerbound},
\begin{align*}
\mathbb{P}^{1,j,[\ell,r],\underline{a}^-  ,\underline{a}^- }_{\free}  (B_i(x)\geq \ug(x) \ \textup{for all}\ x\in [\ell,r]\ \textup{and}\ i\in [1,j]_{\mathbb{Z}} )\geq D^{-1}e^{-DT^3}.
\end{align*}
Thus,
\begin{align*}
\mathbb{E}^{1,j,[\ell,r],\underline{a}^-  ,\underline{a}^- }_{\free} \left[\prod_{i=1}^j W^{i,[\ell,r]}_{\textup{jump}} \right]\geq D^{-1}e^{-DT^{3}}.
\end{align*}

Recall that $\Delta=8\log T$ and $d=64T^{-4}(\log T)^2$. Let 
\begin{equation}\label{def:Flct}
\mathsf{Flct}_i(\ell,r)\coloneqq \{|\mathcal{L}_{i}(x)-\mathcal{L}_i(y)|\leq \Delta\ \textup{for all}\ x,y\in [\ell,r]\ \textup{with}\ |x-y|\leq d \}.
\end{equation}
Because $\mathcal{L}_i(\ell)=\underline{a}^-_i$, $\mathcal{L}_i(r)=\underline{a}^+_i$ and \eqref{Tentj_slope},  
\begin{align*}
\left| \frac{\mathcal{L}_i(r)-\mathcal{L}_i(\ell) }{r-\ell}\right|=\left| \frac{\textup{Tent}_{j+1}(r)-\textup{Tent}_{j+1}(\ell) }{r-\ell}\right|\leq 2kT.
\end{align*}
Hence
\begin{align*}
\mathbb{P}^{1,j,[\ell,r],\underline{a}^-  ,\underline{a}^- }_{\free}  (\mathsf{Flct}^{\textup{c}}_i(\ell,r) )\leq &\mathbb{P}_{\free}(\omega_{d,[\ell,r]}(B^{[\ell,r]})\geq \Delta-2kdT)\\
\leq &\mathbb{P}_{\free}(\omega_{d,[\ell,r]}(B^{[\ell,r]})\geq 2^{-1} \Delta)\leq De^{-D^{-1}T^4}.
\end{align*}
As a result,
\begin{align*}
\mathbb{P}^{1,j,[\ell,r],\underline{a}^-  ,\underline{a}^+ }_{\textup{jump}}  (\bigcup_{i=1}^j \mathsf{Flct}^{\textup{c}}_i(\ell,r)  )\leq De^{DT^3-D^{-1}T^4}\leq 4^{-1}.
\end{align*} 

From Condition \textbf{C3},\ $\mathsf{Low}_j(\ell,r) $ and $ \bigcap_{i=1}^j\mathsf{Flct}_i(\ell,r)$ imply there exists $i\in [1,j]_{\mathbb{Z}}$ and an interval in $  \textup{Int}_j\cap [\ell,r]$ with length $d$ in which
\begin{align*}
\mathcal{L}_i(x)\leq \ug(x)-k\Delta-k\log k.
\end{align*}
Together with 
\begin{align*}
d \textbf{H}_{i,j+1}(k\Delta+k\log k)\geq d e^{\Delta }=64T^4(\log T)^2,
\end{align*}
we deduce that when $\mathsf{Low}_j(\ell,r)\cap \bigcap_{i=1}^j\mathsf{Flct}_i(\ell,r)$ occurs,
\begin{align*}
\prod_{i=1}^j W^{i,[\ell,r]}_{\textup{jump}}\leq e^{-64T^4(\log T)^2}.
\end{align*}
As a result,
\begin{align*}
&\mathbb{P}^{1,j,[\ell,r],\underline{a}^-  ,\underline{a}^+ }_{\textup{jump}}  \left(\mathsf{Low}_j(\ell,r)\cap \bigcap_{i=1}^j\mathsf{Flct}_i(\ell,r) \right)
\leq  \mathbb{E}^{1,j,[\ell,r],\underline{a}^-  ,\underline{a}^+ }_{\free} \left[\prod_{i=1}^j W^{i,[\ell,r]}_{\textup{jump}} \right]^{-1}\times e^{-64T^4(\log T)^2}\\
\leq & De^{DT^3-64T^4(\log T)^2}\leq 4^{-1}.
\end{align*}

Putting the above estimates together, we conclude that
\begin{align*}
&\mathbb{P}^{1,j,[\ell,r],\underline{a}^-  ,\underline{a}^+ }_{\textup{jump}}  (\mathsf{Low}_j(\ell,r) )\\
 \leq&  \mathbb{P}^{1,j,[\ell,r],\underline{a}^-  ,\underline{a}^+ }_{\textup{jump}}  \left(\mathsf{Low}_j(\ell,r)\cap \bigcap_{i=1}^j\mathsf{Flct}_i(\ell,r) \right)+\mathbb{P}^{1,j,[\ell,r],\underline{a}^-  ,\underline{a}^+ }_{\textup{jump}}  \left( \bigcup_{i=1}^j\mathsf{Flct}^{\textup{c}}_i(\ell,r) \right)\\
 \leq &2^{-1}.
\end{align*}

\noindent \underline{Step 3}: In this step, we show that
$$\mathbb{P}^{1,j,[\ell,r],a^-  ,a^+ }_{\textup{jump} }    ( \mathsf{S}_j(\ell,r)  \, |\,  \mathsf{Low}^{\textup{c}}_j(\ell,r))\geq D^{-1}e^{-DT^{5/2}}.$$

Consider the event $\mathsf{S}_j'(\ell,r)$ being defined by 
\begin{itemize}
\item $  \mathcal{L}_i(x)\geq  \mathcal{L}_{i+1}(x)+(k+3/2-j)T^{1/2}\ \textup{for all}\ i\in [1,j-1]_{\mathbb{Z}}\ \textup{and}\ x\in P_{j+1}\cap (\ell,r)$,
\item $ \mathcal{L}_j(x)\geq  \textup{Tent}_{j+1}(x)+8k(2k+5/2-2j)sT \ \textup{for all} \ x\in P_{j+1}\cap (\ell,r).$ 
\end{itemize}
Compared to $\mathsf{S}_j(\ell,r)$, $\mathsf{S}_j'(\ell,r)$ requires a larger gap between  $\mathcal{L}$ at points in $ P_{j+1}\cap (\ell,r)$.\\

In order to apply Theorem \ref{thm:separation}, we rewrite the following objects. Set
\begin{align*}
\breve{f}(x)\coloneqq\textup{Tent}_{j+1}(x)-\Delta_k,\ \breve{\mathcal{L}}_i(x):=\mathcal{L}_i(x),\ \breve{P}=P_{j+1}\cap (\ell,r),\ \breve{a}^\pm_i= {a}^\pm_i. 
\end{align*}
It can easily verified that
 \begin{align*}
\mathsf{Low}_j^{\textup{c}}=&\{\breve{\mathcal{L}}_i(p)>\breve{f}(p)  \ \textup{for}\ i\in [1,\mathbb{Z}]\ \textup{and}\ p\in  \breve{P}\},\\
\mathsf{S}'_{j }(\ell,r) =& \{\breve{\mathcal{L}}_i(p)\geq  \breve{\mathcal{L}}_{i+1}(p)+(k+3/2-j)T^{1/2}\ \textup{for all}\ i\in [1,j-1]_{\mathbb{Z}}\ \textup{and}\ p\in \breve{P} \}\\
&\cap \{ \breve{\mathcal{L}}_j(x)\geq  \breve{ f }(x)+8k(2k+5/2-2j)sT+\Delta_k \ \textup{for all} \ x\in \breve{P}_{j+1} .\}.
\end{align*}

From the conditions in $bS_j(\ell,r)$
\begin{align*}
& \breve{a}^\pm_i-\breve{a}^\pm_{i+1}\geq   (k+2-j)T^{1/2}\ \textup{  for all}\  i\in [1,j-1]_{\mathbb{Z}},\\ 
 &\breve{a}^-_j-\breve{f}(\breve{\ell}),\breve{a}^+_j -\breve{f}(\breve{r}) \geq   8k(2k+3-2j)sT+\Delta_k  
\end{align*}
and
\begin{align*}
\breve{a}^\pm_1\leq jT^2. 
\end{align*}
Together with \eqref{Pole} and \eqref{Tentj_slope}, the assumptions \eqref{assumptionC} and \eqref{assumptionC2} hold.  By Theorem \ref{thm:separation}, we conclude that
\begin{align*}
\mathbb{P}^{1,j,[\ell,r],a^-  ,a^+ }_{\free }  &( \mathsf{S}'_j(\ell,r) \, |\,  \mathsf{Low}^{\textup{c}}_j(\ell,r))\geq D^{-1}e^{-DT^{5/2}}.
\end{align*}

Resampling the curves between pole points in $P_{j+1}\cap(\ell,r)$, we deduce
\begin{align*}
\mathbb{P}^{1,j,[\ell,r],a^-  ,a^+ }_{\free }  &( \mathsf{S}_j(\ell,r) \, |\,  \mathsf{Low}^{\textup{c}}_j(\ell,r))\geq D^{-1}e^{-DT^{5/2}}.
\end{align*}
When $\mathsf{S}_j(\ell,r)$ occurs, $\prod_{i=1}^j W_{\textup{jump}}^{i,[\ell,r]}\geq D^{-1}$. Therefore
 \begin{align*}
 &\mathbb{P}^{1,j,[\ell,r],a^-  ,a^+ }_{\textup{jump} }    ( \mathsf{S}_j(\ell,r)  \, |\,  \mathsf{Low}^{\textup{c}}_j(\ell,r)) \geq  D^{-1}e^{-DT^{5/2}}.
 \end{align*}

In conclusion,
\begin{align*}
&\mathbb{P}^{1,j,[\ell,r],a^-,a^+}_{\textup{jump}} (\mathsf{S}_j(\ell,r)\cap\mathsf{Up}_j(\ell,r))\geq  \mathbb{P}^{1,j,[\ell,r],a^-  ,a^+ }_{\textup{jump}  }   ( \mathsf{S}_j(\ell,r) )-\mathbb{P}^{1,j,[\ell,r],a^-,a^+}_{\textup{jump}} ( \mathsf{Up}^{\textup{c}}_j(\ell,r))\\
 \geq&  \mathbb{P}^{1,j,[\ell,r],a^-  ,a^+ }_{\textup{jump} }    ( \mathsf{S}_j(\ell,r)  \, |\,  \mathsf{Low}^{\textup{c}}_j(\ell,r))\times  \mathbb{P}^{1,j,[\ell,r],a^-  ,a^+ }_{\textup{jump}  }   (\mathsf{Low}^{\textup{c}}_j(\ell,r)  )-\mathbb{P}^{1,j,[\ell,r],a^-,a^+}_{\textup{jump}} ( \mathsf{Up}^{\textup{c}}_j(\ell,r))\\
\geq  &D^{-1}e^{-DT^{5/2}}-De^{-D^{-1}T^3}\geq D^{-1}e^{-DT^{5/2}}.
\end{align*}
The proof is finished.
\end{proof}

\subsection{Proof of Proposition \ref{clm:denominator}}\label{subsec:induction} In this subsection, we combine Propositions \ref{pro:A} and \ref{pro:seperation} to prove Proposition \ref{pro:F}. Proposition \ref{pro:notouch} follows easily from Proposition \ref{pro:F}.\\

For $j\in [1,k]_{\mathbb{Z}}$, Let $\mathsf{F}_j$ be the event
\begin{itemize}
\item $J_i(x)> \ug(x) \ \textup{for all}\ i\in [1,j]_\mathbb{Z}\ \textup{and}\ x\in [\ell_i+d',r_i-d']$,
\item $ J_j(x)> \textup{Tent}_{j+1}(x)+16k(k+1-j)sT\ \textup{for all}\ x\in [\ell_{j}+1 ,r_{j}-1]$,
\item $ J_{i}(x)> J_{i+1}(x)+(k+1-j)T^{1/2}\ \textup{for all}\ i\in [1,j-1]_\mathbb{Z}\ \textup{and}\ x\in [\ell_{i+1},r_{i+1}] $,
\item$J_1(x)< (j+1)T^2\ \textup{for all}\ x\in [\ell_{1} ,r_{1}]$.
\end{itemize}
Roughly speaking, $\mathsf{F}_j$ is the event that $J_i> \ug $ in $[\ell_i+d',r_i-d']$ and $(J_1,J_2,\dots J_j)$ stays ordered.
\begin{proposition}\label{pro:F}
There exists a constant $D=D(k,s)$ such that for $E$ large enough, the following statement holds. For all $j\in [1,k]_{\mathbb{Z}}$, 
\begin{equation}
\mathbb{P}_{J}(\mathsf{F}_j)\geq D^{-1}e^{-DT^{5/2}}.
\end{equation}
\end{proposition}
\begin{proof}[{\bf Proof of Proposition \ref{pro:notouch}}]
It follows from $\mathsf{F}_k\subset\mathsf{NoTouch}$.
\end{proof}
\begin{proof}[{\bf Proof of Proposition \ref{pro:F}}]
We use induction on $j$ and start with $j=1$. Comparing $\mathsf{A}_1$ and $\mathsf{F}_1$, the extra requirement in $\mathsf{F}_1$ is that $J_1(x)< 2T^2$. By a stopping domain argument  we can show that $\mathbb{P}_J(J_1(x)< 2T^2)\geq 1-De^{-D^{-1}T^3}$. See Step 1 in the proof of Proposition \ref{pro:seperation} for a similar argument. We omit the detail here. For $T$ large enough, it holds that $$\mathbb{P}(\mathsf{F}_1)\geq \mathbb{P}(\mathsf{A}_1)-De^{-D^{-1}T^3}\geq D^{-1}e^{-DT^{5/2}}.$$ 

Let $j\in [2,k]_{\mathbb{Z}}$ and assume that
\begin{align*}
\mathbb{P}_{J}(\mathsf{F}_{j-1})\geq D^{-1}e^{-DT^{5/2}}.
\end{align*}
Under the law of $\mathbb{P}_J$, $\mathsf{F}_{j-1}$ and $\mathsf{A}_j$ are independent. From Propositions \ref{pro:A}, it holds that
\begin{align*}
\mathbb{P}_{J}(\mathsf{F}_{j-1}\cap\mathsf{A}_j)=\mathbb{P}_{J}(\mathsf{F}_{j-1} )\times \mathbb{P}_{J}( \mathsf{A}_j )\geq D^{-1}e^{-DT^{5/2}}.
\end{align*}
In other words, with probability $e^{-T^{5/2}}$, $J_1,\dots,J_{j+1}$ are well-separated and $J_j$  jumps over $\ug$. \\

Comparing $\mathsf{F}_{j-1}\cap\mathsf{A}_j$ with $\mathsf{F}_j$, the only requirement in $\mathsf{F}_j$ which is not guaranteed by $\mathsf{F}_{j-1}\cap\mathsf{A}_j$ is  
\begin{align}\label{extra}
 J_{j-1}(x)>J_{j }(x)+(k+1-j)T^{1/2}\ \textup{for all}\ x\in [\ell_j,r_j]. 
\end{align}
In other words, the only situation $\mathsf{F}_{j-1}\cap\mathsf{A}_j\cap \mathsf{F}^{\textup{c}}_{j}$ can occur is that $J_{j}$ raises to high and intersects with $J_{j-1}$. We use a stopping domain argument to resolve this issue. In particular, we will consider the stopping domain given by the leftmost and rightmost locations where \eqref{extra} fails.\\

Define a random set $I'_{j,-}\subset [\ell_1,r_1]$ which captures the requirements in $\mathsf{F}_{j-1}\cap\mathsf{A}_j$ from the left. Explicitly, $I'_{j,-}\subset [\ell_1,r_1]$ is the the collection of $x_0\in [\ell_1,r_1]$ such that
\begin{itemize}
\item $J_i(x)> \ug (x)\ \textup{for all}\ i\in [1,j]_\mathbb{Z}\ \textup{and}\ x\in [\ell_i+d',r_i-d']\cap (-\infty,x_0]$.
\item $J_{i}(x)> J_{i+1}(x)+(k+2-j)T^{1/2}\ \textup{for all}\ i\in [1,j-2]_\mathbb{Z}\ \textup{and}\ x\in [\ell_{i+1},r_{i+1}] \cap (-\infty,x_0]$.
\item $J_{j-1}(x)> \textup{Tent}_{j }(x)+16k(k+2-j)sT\ \textup{for all}\ x\in [\ell_{j-1}+1 ,r_{j-1}-1]\cap (-\infty,x_0]$.
\item $J_{j }(x)> \textup{Tent}_{j+1 }(x)+16k(k+1-j)sT\ \textup{for all}\ x\in [\ell_{j}+1 ,r_{j}-1]\cap (-\infty,x_0]$.
\item $ J_1(x)< jT^2\ \textup{for all}\ x\in [\ell_{1} ,r_{1}] \cap (-\infty,x_0]$.
\end{itemize}
Furthermore, we define $I''_{j,-}$ to be the subset of $I'_{j,-}$ with an extra requirement that
\begin{align*}
 \ J_{j-1}(x)> J_{j}(x)+(k+2-j)T^{1/2}\ \textup{for all}\ x\in [\ell_{j},r_{j}] \cap (-\infty,x_0].
\end{align*}
Define $I'_{j,+},I''_{j,+}\subset [\ell_1,r_1]$ by replacing $(-\infty,x_0]$ with $[x_0,\infty)$. 
Define \begin{align*}
\mathfrak{l}''_j\coloneqq\left\{ \begin{array}{cc}
\sup I''_{j,-},\ & \textup{if}\ I''_{j,-}\neq \phi\\
\ell_1, & \textup{if}\ I''_{j,-}= \phi.
\end{array}\right.\ \textup{and}\ \ \mathfrak{r}''_j \coloneq\left\{ \begin{array}{cc}
\inf I''_{j,+},  & \textup{if}\ I''_{j,+}\neq \phi\\
r_1, & \textup{if}\ I''_{j,+}= \phi.
\end{array}\right.
 \end{align*}
 
As $\mathsf{F}_{j-1}\cap \mathsf{A}_j$ occurs, $ I'_{j,\pm} =[\ell_1,r_1]$. Moreover, $\ell_i\in I''_{j,-}$. It can be shown as the following. By the second requirement of $\mathsf{F}_{j-1}$, it holds that
\begin{align*}
 J_{j-1}(\ell_j)\geq  &\textup{Tent}_{j}( \ell_j  )+16k(k+2-j)sT.
\end{align*}
By \eqref{Tentj_lowerbound}, it holds that
\begin{align*}
\textup{Tent}_{j}( \ell_j  )+16k(k+2-j)sT \geq &\ug(\ell_j )-3ksT+16k(k+2-j)sT.
\end{align*}
Because $\ug(\ell_j )=\max\{f_j(\ell_j),f_{j+1}(\ell_j)\}\geq f_j(\ell_j)=J_j(\ell_j)$, we have
 \begin{align*}
\ug(\ell_j )-3ksT+16k(k+2-j)sT  \geq  J_{j}(\ell_j )-3ksT+16k(k+2-j)sT  >  J_{j}(\ell_j )+(k+2-j)T^{1/2}.
 \end{align*}
Similarly, $\mathsf{F}_{j-1}\cap \mathsf{A}_j$  implies $r_i\in I''_{j,+}$.\\
 
Next, we discuss the implications of $\mathsf{F}_{j-1}\cap \mathsf{A}_j\cap \mathsf{F}^{\textup{c}}_{j }$. $\mathsf{F}_{j-1}\cap \mathsf{A}_j\cap \mathsf{F}^{\textup{c}}_{j }$ implies \eqref{extra} fails at a point in $(\ell_j,r_j)$. In particular, $\ell_j<\mathfrak{l}''_j <\mathfrak{r}''_j<r_j$,
 \begin{align*}
 J_j(\mathfrak{l}''_j)=J_{j-1}(\mathfrak{l}''_j)-(k+2-j)T^{1/2},
 \end{align*}
 and
 \begin{align*}
 J_j(\mathfrak{r}''_j)=J_{j-1}(\mathfrak{r}''_j)-(k+2-j)T^{1/2}.
 \end{align*}
 Moreover, all the conditions in $I'_{j,\pm}$ hold at $\mathfrak{l}''_j$ and $\mathfrak{r}''_j$.\\
 
Define the event $\mathsf{  K }_j$ by
\begin{align*}
\mathsf{  K }_j \coloneqq \{\ell_j<\mathfrak{l}''_j <\mathfrak{r}''_j<r_j\}\cap\{(J_1(\mathfrak{l}''_j),\dots,J_j(\mathfrak{l}''_j),J_1(\mathfrak{r}''_j),\dots,J_j(\mathfrak{r}''_j)) \in bS_j(\mathfrak{l}''_j,\mathfrak{r}''_j)\}.
\end{align*}
Here $bS_j(\ell,r)$  defined by \eqref{equ:bS1}, \eqref{equ:bS2} and \eqref{equ:bS3}. We claim that $\mathsf{F}_{j-1}\cap \mathsf{A}_j\cap \mathsf{F}^{\textup{c}}_{j } \subset \mathsf{ {K}}_j$. Comparing the condition in $bS_j(\mathfrak{l}''_j,\mathfrak{r}''_j)$ and  in $I'_{j,\pm}$, we only need to check
\begin{align*}
\mathsf{F}_{j-1}\cap \mathsf{A}_j\cap \mathsf{F}^{\textup{c}}_{j } \subset  &\{ J_j(\mathfrak{l}''_j)\geq \textup{Tent}_{j+1}(\mathfrak{l}''_j) +8k(k+3-2j)sT  \}\\
&\cap\{ J_j(\mathfrak{r}''_j)\geq \textup{Tent}_{j+1}(\mathfrak{r}''_j) +8k(k+3-2j)sT  \}.
\end{align*} 

We give the proof for $\mathfrak{l}''_j$. The argument for $\mathfrak{r}''_j$ is similar. Suppose $\mathfrak{l}''_j\in [\ell_j+1,r_j-1]$, then the assertion holds by the fourth requirement of $I'_{j,-}$. Assume $\mathfrak{l}''_j\in [\ell_j,\ell_j+1)$. By the third requirement of $I_{j,-}'$,

\begin{align*}
J_{j-1}(\mathfrak{l}''_j)\geq &\textup{Tent}_j(\mathfrak{l}''_j)+16k(k+2-j)sT.
\end{align*} 
Because of \eqref{Tentj_slope},
\begin{align*}
\textup{Tent}_j(\mathfrak{l}''_j)+16k(k+2-j)sT \geq&   \textup{Tent}_j(\ell_j)-2ksT +16k(k+2-j)sT.
\end{align*}
From  \eqref{Tentj_lowerbound} and $ \ug(\ell_j)\geq f_{j+1}(\ell_j)=\textup{Tent}_{j+1}(\ell_j)$,
\begin{align*}
 \textup{Tent}_j(\ell_j)-2ksT +16k(k+2-j)sT \geq   \textup{Tent}_{j+1}(\ell_j)-5ksT +16k(k+2-j)sT.
\end{align*} 
Using \eqref{Tentj_slope}, again, we have
\begin{align*}
J_{j-1}(\mathfrak{l}''_j) \geq&  \textup{Tent}_{j+1}(\mathfrak{l}''_j )-7ksT +16k(k+2-j)sT. 
\end{align*}

As a result, $\mathsf{F}_{j-1}\cap \mathsf{A}_j\cap \mathsf{F}^{\textup{c}}_{j }  $ implies
\begin{align*}
J_j(\mathfrak{l}''_j)= &J_{j-1}(\mathfrak{l}''_j)-(k+2-j)T^{1/2}\\
\geq&  \textup{Tent}_{j+1}(\mathfrak{l}''_j )-8ksT +16k(k+2-j)sT\\
=&\textup{Tent}_{j+1}(\mathfrak{l}''_j)+8k(2k+3-2j)sT.
\end{align*}
The argument for $\mathfrak{l}''_j\in (r_j-1,r_j]$ is the similar.\\
 
Next, we check that  $$\mathsf{S}_j(\mathfrak{l}''_j,\mathfrak{r}''_j)\cap\mathsf{Up}_j(\mathfrak{l}''_j,\mathfrak{r}''_j)\cap \mathsf{K}_j \subset \mathsf{F}_j\cap  {\mathsf{K}}_j.$$
The events $\mathsf{S}_j(\mathfrak{l}''_j,\mathfrak{r}''_j)$ and  $\mathsf{Up}_j(\mathfrak{l}''_j,\mathfrak{r}''_j)$ are defined in \eqref{def:Sj} and \eqref{def:Upj}. The definition of $\mathsf{K}_j$ ensures the requirements in $\mathsf{F}_j$ holds outside the interval $(\mathfrak{l}''_j,\mathfrak{r}''_j)$. $\mathsf{S}_j(\mathfrak{l}''_j,\mathfrak{r}''_j)\cap\mathsf{Up}_j(\mathfrak{l}''_j,\mathfrak{r}''_j)$ then takes cares of the interval $[ \mathfrak{l}''_j,\mathfrak{r}''_j  ]$.\\

We are ready to give a lower bound for $\mathbb{P}_{J }(\mathsf{F}_j)$. Let $\delta\coloneqq  \mathbb{P}_{J}(\mathsf{F}_{j-1}\cap \mathsf{A}_j\cap \mathsf{F}^{\textup{c}}_{j } ).$ From the above discussion,
\begin{align*}
\mathbb{P}_{J }(\mathsf{F}_j)\geq  \mathbb{P}_{J }(\mathsf{F}_j\cap {\mathsf{K}}_j )\geq &\mathbb{P}_{J }(\mathsf{S}_j(\mathfrak{l}''_j,\mathfrak{r}''_j)\cap\mathsf{Up}_j(\mathfrak{l}''_j,\mathfrak{r}''_j)\cap  {\mathsf{K}}_j ).
\end{align*} 
From the strong Markov property and Proposition \ref{pro:seperation}, it holds that 
\begin{align*}
 \mathbb{P}_{J }(\mathsf{S}_j(\mathfrak{l}''_j,\mathfrak{r}''_j)\cap\mathsf{Up}_j(\mathfrak{l}''_j,\mathfrak{r}''_j)\cap  {\mathsf{K}}_j )= &\mathbb{P}_{J }(\mathsf{S}_j(\mathfrak{l}''_j,\mathfrak{r}''_j)\cap\mathsf{Up}_j(\mathfrak{l}''_j,\mathfrak{r}''_j)\,|\,  {\mathsf{K}}_j ) \times \mathbb{P}_J( {\mathsf{K}}_j) \\
\geq &D^{-1}e^{-DT^{5/2}}\times \mathbb{P}_{J}(\mathsf{K}_j).
\end{align*}
Together with $ \mathsf{F}_{j-1}\cap \mathsf{A}_j\cap \mathsf{F}^{\textup{c}}_{j } \subset\mathsf{K}_j$, we have
\begin{align*}
\mathbb{P}_{J }(\mathsf{F}_j) \geq \delta\times D^{-1}e^{-DT^{5/2}} .
\end{align*}
On the other hand,
\begin{align*}
\mathbb{P}_{J}(\mathsf{F}_j)\geq \mathbb{P}_{J}(\mathsf{F}_{j-1}\cap \mathsf{A}_j\cap \mathsf{F}_{j } )= \mathbb{P}_{J}  (\mathsf{F}_{j-1}\cap \mathsf{A}_j )-\delta \geq D^{-1}e^{-D T^{5/2}}-\delta ,
\end{align*}
We conclude that 
\begin{align*}
\mathbb{P}_{J}(\mathsf{F}_j)\geq \inf_{\delta \geq 0} \max\{ \delta \times D^{-1} e^{-DT^{5/2}},D^{-1}e^{-DT^{5/2}}-\delta\}\geq D^{-1}e^{-DT^{5/2}}.
\end{align*}
The induction argument is finished.
\end{proof}

\section{Proof of Proposition \ref{clm:numerator}}\label{sec:numerator}
In this section we prove Proposition \ref{clm:numerator}. The jump ensemble $J$ is defined through \eqref{def:J}. Recall that $A\subset {C}_{0,0}([0,s])$ is a Borel set and $\varepsilon=\mathbb{P}_{\free}(B^{[0,s]}\in A).$ The goal is to prove
\begin{align*}
\sup \{\mathbb{P}_{J}(  J_k^{[0,s]}\in A  )\, |\,  (\bar{\ell},\bar{r},f_{\textup{J}})\in \mathcal{G}  \}  \leq D\varepsilon \exp\left( D\left(\log\varepsilon^{-1}\right)^{5/6} \right).
\end{align*}

\indent Starting from now, we fix a realization of the favorable event $(\bar{\ell},\bar{r},f_{\textup{J}})\in\mathcal{G}$ (see Definition \ref{def:G}). Suppose $P^*\cap (0,s)=\phi$, i.e. no pole in $(0,s)$. By the construction of the jump ensemble $J$, the law of $J_k^{[0,s]}$ is equivalent to that of a free Brownian bridge. Then $\mathbb{P}_{J}(J_k^{[0,s]}\in A)=\varepsilon$ and Proposition \ref{clm:numerator} follows easily. \\[0.1cm]
\indent From now on we assume $P^*\cap (0,s)\neq \phi$. From the construction of the pole set $P^*$, see \eqref{Pole_three} and \eqref{def:Pstar}, $P^*\cap (0,s)$ consists of only one element, denoted by $p_0$. We enlarge the interval from $[0,s]$ to $[q_1,q_2]$ to prevent $p_0$ from being too close to the boundary. Let
\begin{align*}
q_1&\coloneqq p_0-\max\{p_0,\, 2^{-1}s\},\\
q_2&\coloneqq p_0+\max\{s-p_0,\, 2^{-1}s\}.
\end{align*}
Let $\pi:{C}_{0,0}([q_1,q_2])\to {C}_{0,0}([0,s])$ be a restriction map, defined as
\begin{align*}
\pi[h](x):=h(x)-s^{-1}xh(s)-s^{-1}(s-x)h(0).
\end{align*}
Note that $\pi\circ B^{[q_1,q_2]}  {=} B^{[0,s]}$ and $\pi\circ J_k^{[q_1,q_2]} {=} J_k^{[0,s]}$. For $A'=\pi^{-1}(A)$, it holds that
\begin{align*}
\mathbb{P}_{\free}(B^{[q_1,q_2]}\in A')=&\mathbb{P}_{\free}(B^{[0,s]}\in A)=\epsilon,\\
\mathbb{P}_{J}(J_k^{[q_1,q_2]}\in A')=&\mathbb{P}_{J}(J_k^{[0,s]}\in A).
\end{align*}
In the remaining of this section, we consider the event $\{J_k^{[q_1,q_2]}\in A'\}$ instead of $\{J_k^{[0,s]}\in A\}$.\\ 

Let $p_-=\max\{p\in P^*\,|\,p<p_0\}$ and $p_+=\min\{p\in P^*\,|\,p>p_0\}$ be the elements in $P^*$ next to $p_0$.
From \eqref{Pole_three}, it holds that
 $$-2^{-1}T<p_-<q_1<p_0<q_2< p_+<2^{-1}T. $$ 

Recall $\ug(x)$ in \eqref{def:g2}. Define the tent map
\begin{equation*} 
\textup{Tent}(x)\coloneqq \left\{ \begin{array}{cc}
\ug(x) & x= p_0\ \textup{or}\ x=p_\pm,\\
\textup{linear inteperlation} & x\in (p_-,p_0 )\cup (p_0,p_+).
\end{array} \right.
\end{equation*}
We remark that $\textup{Tent}$ agrees with $\textup{Tent}_{j+1}$ in $[p_-,p_+]$ for all $j\in [1,k]_{\mathbb{Z}}$.\\

Recall that $\Delta_k=(k+2)\Delta+k\log k$. Define the quantities (see Figure~\ref{fig:q1q2} for an illustration).
\begin{align*}
\underline{m}_0=&\textup{Tent}(p_0)-\Delta_k,\ \underline{m}_\pm=\textup{Tent}(p_\pm)- \Delta_k,\\
\underline{m}_1=& \textup{Tent}(q_1)- \Delta_k,\ \underline{m}_2=\textup{Tent}(q_2)- \Delta_k,\\
\sigma_1^2=&\frac{(p_0-q_1)(q_1- {p}_-)}{ {p_0}- {p}_-},\ \sigma_2^2=\frac{( {p}_+-q_2)(q_2- {p_0}) }{ {p}_+- {p_0}},\ \sigma_3^2=\frac{(q_2-p_0)(p_0-q_1)}{q_2-q_1}.
\end{align*}
 
\begin{figure}[H]
\includegraphics[width=10cm]{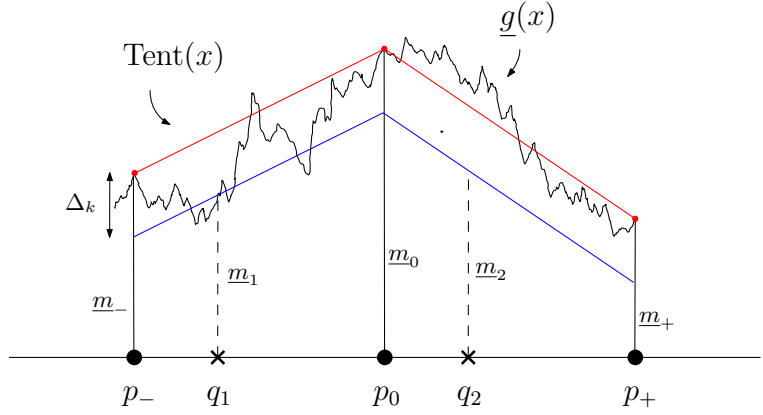}
\caption{Tent is constructed using values of $\ug$ at poles $p_{-}, p_0, p_{+}$ with $p_0$ being the single pole in $[0,s]$. We expect the values of $J_k$ stay around Tent, hence above Tent$-\Delta_k$ with high probability. }
\label{fig:q1q2}
\end{figure}  

The following lemma shows that $J_k( {p_0} )\geq\underline{m}_0$ and $J_k( {p_\pm} )\geq\underline{m}_\pm$ occur with high probability. Define the event
\begin{align*}
\mathsf{Y}\coloneqq \left\{ J_k( {p_0} )\geq\underline{m}_0 ,\ J_k( {p}_\pm )\geq\underline{m}_\pm \right\}. 
\end{align*}

\begin{lemma}\label{lem:Y}
There exists a constant $D=D(k)$ such that for $E$ large enough, we have
\begin{align*}
\mathbb{P}_J(\mathsf{Y}^{\textup{c}})\leq De^{-D^{-1}T^4}.
\end{align*}
\end{lemma}
\begin{proof}
The idea is that if $J_k( {p_0} )<\underline{m}_0$ or $J_k( {p_\pm} )<\underline{m}_\pm$ occurs, then with high probability $J_k$ will stay below $\ug-k\Delta-k\log k$ in an interval with length $d$. Because $d \, \mathbf{H}_t^{k,k+1}(k\Delta+k\log k)\approx T^4$, the Boltzmann weight $W_{\textup{jump}}^k$ is less than $e^{-T^4}$ which is unlikely. In below we fill in details of this argument.\\[0.1cm]
\indent From the definition of $\mathbb{P}_J$ ( see \eqref{def:J} and \eqref{def:Wjump}), $J_k$ is independent of other curves and is distributed according to
\begin{align*}
\frac{{d} \mathbb{P}_{J_k}}{{d} \mathbb{P}^{k,k,(\ell_k,r_k),x_k,y_k}_{\free}}(\mathcal{L})=\frac{ W^{k }_{\textup{jump}}(\mathcal{L})}{ \mathbb{E}^{k,k,(\ell_k,r_k),x_k,y_k}_{\free}[ W^{k }_{\textup{jump}}]} .
\end{align*}
Here $x_k=f_k(\ell_k)$ and $y_k=f_k(r_k)$. We know that when $\mathsf{Y}^{\textup{c}}$ occurs, $W^{k }_{\textup{jump}}$ is, with high probability, less than $e^{-T^4}$. To conclude $\mathsf{Y}^{\textup{c}}$ is unlikely, we need to bound the normalizing constant. To do so, we turn off the interaction in $W^k_{\textup{jump}}$ except near $p_0$ and $p_\pm$ to get a new curve $J'_k$. The law of $J'_k$ satisfies the following Radon-Nikodym derivative relation,
\begin{align*}
\frac{{d} \mathbb{P}_{J'_k}}{{d} \mathbb{P}^{k,k,(\ell_k,r_k),x_k,y_k}_{\free}}(\mathcal{L})=\frac{ W'^{k }_{\textup{jump}}(\mathcal{L})}{ \mathbb{E}^{k,k,(\ell_k,r_k),x_k,y_k}_{\free}[ W'^{k }_{\textup{jump}}]} .
\end{align*}
Here the weight $W'^{k }_{\textup{jump}}$ is defined by
\begin{equation*}
\begin{split}
W'^{k}_{\textup{jump}}(\mathcal{L} )\coloneqq &  \exp\left( -\int_{[p_--d,p_-]\cup [p_0,p_0+d]\cup [p_+,p_++d] } \textbf{H}_t^{k,k+1}(\ug (x)-\mathcal{L}(x)) dx \right) .
\end{split}	
\end{equation*}


By the stochastic monotonicity Lemma~\ref{monotonicity1}, it holds that
\begin{align*}
\mathbb{P}_{J_k}(\mathsf{Y}^{\textup{c}})\leq \mathbb{P}_{J'_k}(\mathsf{Y}^{\textup{c}}).
\end{align*}

From condition \textbf{C2} and Remark \ref{rmk:U}, it is straightforward to check that 
\begin{align*}
\mathbb{P}^{k,k,(\ell_k,r_k),x_k,y_k}_{\free}(\{B(x)\geq \ug(x)\ \textup{for}\ x\in [p_--d,p_-]\cup [p_0,p_0+d]\cup [p_+,p_++d]\})\geq C^{-1}e^{-CT^3}.
\end{align*}
Here $x_k=f_k(\ell_k)$ and $y_k=f_k(r_k)$. When this event occurs, $W'^k_{\textup{jump}}$ is bounded from below. As a result,
$$\mathbb{E}^{k,k,(\ell_k,r_k),x_k,y_k}_{\free}[ W'^{k }_{\textup{jump}}]\geq D^{-1}e^{-DT^3}.$$

Next, we bound the local fluctuation of $J'_k$. Recall the event $\mathsf{Flct}_k(\ell_k,r_k)$ is defined in \eqref{def:Flct} as
\begin{equation*}
\mathsf{Flct}_k(\ell_k,r_k)\coloneqq \{|\mathcal{L}_{k}(x)-\mathcal{L}_k(y)|\leq \Delta\ \textup{for all}\ x,y\in [\ell_k,r_k]\ \textup{with}\ |x-y|\leq d \}.
\end{equation*}
Because $\Delta^2/d=T^2,$ 
$$\mathbb{P}^{k,k,(\ell_k,r_k),x_k,y_k}_{\free}(\mathsf{Flct}^{\textup{c}}_k(\ell_k,r_k))\leq D e^{-D^{-1}T^4}.$$
With the bound on the normalizing constant, it holds that
\begin{align*}
\mathbb{P}_{J'_k}(\mathsf{Flct}_k^{\textup{c}}(\ell_k,r_k))\leq De^{DT^3-D^{-1}T^4}\leq De^{-D^{-1}T^4} .
\end{align*}

When $\mathsf{Flct}_k(\ell_k,r_k)$ and $\mathsf{Y}^{\textup{c}}$ both occur, there exists an interval with length $d$ and on which $\ug(x)-J_k(x)\geq k\Delta+k\log k$. From
\begin{align*}
d\, \textbf{H}_t^{k,k+1}(k\Delta+k\log k)\geq d e^{\Delta }=64T^4(\log T)^2,
\end{align*}
we get
\begin{align*}
 \mathbb{E}^{k,k,(\ell_k,r_k),x_k,y_k}_{\free}[ W'^{k }_{\textup{jump}}\cdot\mathbbm{1}_{\mathsf{Y}^{\textup{c}}}\cdot\mathbbm{1}_{\mathsf{Flct}_k(\ell_k,r_k) }]\leq De^{-D^{-1}T^4(\log T)^2}.
\end{align*}
Using the bound on the normalizing constant again, we conclude that
\begin{align*}
\mathbb{P}_{J'_k}(\mathsf{Y}^{\textup{c}})\leq &\mathbb{P}_{J'_k}(\mathsf{Y}^{\textup{c}}\cap \mathsf{Flct}_k(\ell_k,r_k))+\mathbb{P}_{J'}(\mathsf{Flct}_k^{\textup{c}}(\ell_k,r_k))\\
\leq &De^{-D^{-1}T^4(\log T)^2}+De^{-D^{-1}T^4 }.
\end{align*}
\end{proof}

We now start to derive the expression \eqref{eq:P_J(A)} about $\mathbb{P}_J(J_k^{[q_1,q_2]}\in A')$, which will be the starting point for the estimates. For any $(x_1,x_2)\in\mathbb{R}^2$, define
\begin{align*}
F(x_1,x_2)\coloneqq \mathbb{P}_J(J_k^{[q_1,q_2]}\in A'|J_k(q_1)=x_1,J_k(q_2)=x_2).
\end{align*}
Let $Q(x_1,x_2)$ be the joint p.d.f. of $(J_k(q_1), J_k(q_2))$. Then
\begin{align*}
\mathbb{P}_J(J_k^{[q_1,q_2]}\in A')=\int_{\mathbb{R}^2} F(x_1,x_2)Q(x_1,x_2)dx_1dx_2.
\end{align*}
From the definition \eqref{def:Wjump_j}, $W^{k,[p_-,p_0]}_{\textup{jump}}(\mathcal{L}_k;f_\textup{J})=1$. Also,
\begin{align*}
&W^{k,[p_0,p_+]}_{\textup{jump}}(\mathcal{L}_k;f_\textup{J})=\exp\left( -\int_{[p_0,p_0+d]} \textbf{H}_t^{k,k+1}(\ug(x)-\mathcal{L}_k)\, dx \right).
\end{align*} 
For any $\textbf{b}=(b_0,b_\pm)\in\mathbb{R}^3$, let
\begin{align*}
 m_1(\mathbf{b})\coloneqq \frac{b_-  ({p}_0-q_1)+b_0 (q_1- {p}_-)}{ {p}_0- {p}_-},\  m_2(\mathbf{b})\coloneqq \frac{b_0 (  {p}_+-q_2) +b_+ (q_2- {p}_0) }{ {p}_+- {p}_0}.
\end{align*}
Conditioned on $J_k( {p}_0)=b_0$ and $J_k( {p}_\pm)=b_\pm$, $J_k(q_1)$ and $J_k(q_2)$ are independent. Let $\rho_{\sigma,m}(x)=(2\pi\sigma^2)^{-1/2}e^{-x^2/(2\sigma^2)}$ be the p.d.f. of a normal distribution with mean $m$ and variance $\sigma^2$. The p.d.f. of $J_k(q_1)$ is given by
\begin{align*}
  \rho_{\sigma_1,m_1(\mathbf{b})}(x_1) =:  \rho_{\sigma_1,\underline{m}_1}(x_1) f_1(x_1;\mathbf{b}).
\end{align*}
The function $f_1(x_1;\mathbf{b})$ is defined through the last equality. Note that as $m_1(\mathbf{b})\geq \underline{m}_1$, it is straightforward to check that $f_1(x_1;\mathbf{b})$ is non-decreasing in $x_1$. Similarly, the p.d.f. of $J_k(q_2)$ is given by
\begin{align*}
 &\rho_{\sigma_2,m_2(\mathbf{b})}(x_2)\left( \mathbb{E}_{\free}^{1,1,( {p}_0 ,q_2),b_0 ,x_2}[W^{k,[p_0,p_0+d]}_{\textup{jump}} ] \right)/\mathbb{E}_{\free}^{1,1,( {p}_0 , {p}_+),b_0,b_+}[W^{k,[p_0,p_0+d]}_{\textup{jump}}]  \\
    =:  &\rho_{\sigma_2,\underline{m}_2}(x_2) f_2(x_2;\mathbf{b}).
\end{align*}
The function $f_2(x_2;\mathbf{b})$ is defined through the last equality and is non-decreasing provided $m_2(\mathbf{b})\geq \underline{m}_2$. Write $ \bar{Q}(\mathbf{b})$ for the joint p.d.f of $(J(p_0),J(p_\pm))$. Then 
\begin{align*}
Q(x_1,x_2)=\int_{\mathbb{R}^3} \rho_{\sigma_1,\underline{m}_1}(x_1)\rho_{\sigma_2,\underline{m}_2}(x_2)  f_1(x_1;\mathbf{b})f_2(x_2;\mathbf{b})\bar{Q}(\mathbf{b}) {d}\mathbf{b}.
\end{align*}

The following expression serves as a starting point for proving Proposition \ref{clm:numerator}.
\begin{align}\label{eq:P_J(A)}
\mathbb{P}_J(J_k^{[q_1,q_2]}\in A')=\int_{\mathbb{R}^3} \int_{\mathbb{R}^2} F(x_1,x_2)\rho_{\sigma_1,\underline{m}_1}(x_1)\rho_{\sigma_2,\underline{m}_2}(x_2)  f_1(x_1;\mathbf{b})f_2(x_2;\mathbf{b})\bar{Q}(\mathbf{b})\, {d}x_1  {d}x_2    {d}\mathbf{b}.
\end{align}
Let $\underline{ \mathbf{m}}\coloneqq (\underline{ {m}},\underline{ {m}}_\pm)$ and denote $\{\mathbf{b}\geq \underline{\mathbf{m}}\}\coloneqq \{b_0\geq \underline{m}_0,b_\pm\geq \underline{m}_\pm \}$. Because $F(x_1,x_2)\leq 1$, together with Lemma \ref{lem:Y}, we could bound $\mathbb{P}_J(J_k^{[q_1,q_2]}\in A' )$ by
\begin{align*}
&\int_{\mathbb{R}^3} \int_{\mathbb{R}^2} F(x_1,x_2)\rho_{\sigma_1,\underline{m}_1}(x_1)\rho_{\sigma_2,\underline{m}_2}(x_2)  f_1(x_1;\mathbf{b})f_2(x_2;\mathbf{b})\mathbbm{1}_{\{\mathbf{b}\geq \underline{\mathbf{m}}\}}\bar{Q}(\mathbf{b}) {d}x_1 {d}x_2  {d}\mathbf{b}+\mathbb{P}_J(\mathsf{Y}^{\textup{c}})\\
\leq &\sup_{\mathbf{b}\geq \underline{\mathbf{m}}}\int_{\mathbb{R}^2} F(x_1,x_2)\rho_{\sigma_1,\underline{m}_1}(x_1)\rho_{\sigma_2,\underline{m}_2}(x_2)  f_1(x_1;\mathbf{b})f_2(x_2;\mathbf{b}) {d}x_1 {d}x_2 +\mathbb{P}_J(\mathsf{Y}^{\textup{c}}). 
\end{align*}

From now on we fix $\mathbf{b}\geq \underline{\mathbf{m}}$. The estimates we derive below are uniform in any such choice of $\mathbf{b}$. By the Gibbs property, $F(x_1,x_2)$ can be rewritten as
\begin{align*}
F(x_1,x_2)=&\mathbb{E}_{\free}^{k,k,(q_1,q_2),x_1,x_2}\left[\mathbbm{1}{\{B^{[q_1,q_2]}\in A\}} W^{k,[p_0,p_0+d]}_{\textup{jump}} \right]\left/\mathbb{E}_{\free}^{k,k,(q_1,q_2),x_1,x_2}\left[ W^{k,[p_0,p_0+d]}_{\textup{jump}} \right]\right.\\
\leq &\mathbb{P}^{k,k,(q_1,q_2),0,0}_{\free}(\{B^{[q_1,q_2]}\in A\})\bigg( \mathbb{E}_{\free}^{k,k,(q_1,q_2),x_1,x_2}[W^{k,[p_0,p_0+d]}_{\textup{jump}} ]\bigg)^{-1}.\\
=&\varepsilon\bigg( \mathbb{E}_{\free}^{1,1,(q_1,q_2),x_1,x_2}[ W^{k,[p_0,p_0+d]}_{\textup{jump}} ]\bigg)^{-1}.
\end{align*}
Here we have used $W^{k,[p_0,p_0+d]}_{\textup{jump}} \leq 1$.\\

Let 
$$\bar{F}(x_1,x_2)\coloneqq \min\left\{ 1,\varepsilon \bigg( \mathbb{E}_{\free}^{k,k,(q_1,q_2),x_1,x_2}[W^{k,[p_0,p_0+d]}_{\textup{jump}} ]\bigg)^{-1}\right\}.$$
Hence $F(x_1,x_2)\leq \bar{F}(x_1,x_2)$. Since $\bar{F}(x_1,x_2)$ is non-increasing in $x_1$ for fixed $x_2$ and and $f_1(x_1;\mathbf{b}) $ is non-decreasing, for any $x_2\in\mathbb{R}$,
\begin{align*}
&\int_{\mathbb{R} } \rho_{\sigma_1,\underline{m}_1}(x_1)  \bar{F}(x_1,x_2)f_1(x_1;\mathbf{b})  {d}x_1-\int_{\mathbb{R} } \rho_{\sigma_1,\underline{m}_1}(x_1)   \bar{F}(x_1,x_2)  {d}x_1\\
=&\left( \int_{\mathbb{R} } \rho_{\sigma_1,\underline{m}_1}(x_1)  \bar{F}(x_1,x_2)f_1(x_1;\mathbf{b})  {d}x_1\right)\left(\int_{\mathbb{R}} \rho_{\sigma_1,\underline{m}_1}(x_1')\,{d}x'_1 \right)\\
&-\int_{\mathbb{R} } \rho_{\sigma_1,\underline{m}_1}(x_1)   \bar{F}(x_1,x_2)  {d}x_1\left(\int_{\mathbb{R}} \rho_{\sigma_1,\underline{m}_1}(x_1')f_1(x'_1;\mathbf{b})\,{d}x'_1 \right)\\
=&   2^{-1}\int_{\mathbb{R}^2 } \rho_{\sigma_1,\underline{m}_1}(x_1)\rho_{\sigma_1,\underline{m}_1}(x'_1) (\bar{F}(x_1,x_2)-\bar{F}(x'_1,x_2) )(f_1(x_1;\mathbf{b}) -f_1(x'_1;\mathbf{b}) ) {d}x_1{d}x_1'\\
\leq &0.
\end{align*}
Since $\bar{F}(x_1,x_2)$ is non-increasing in $x_2$ for fixed $x_1$ and and $f_2(x_2;\mathbf{b}) $ is non-decreasing, we can repeat the same deduction to get 
\begin{align*}
&\int_{\mathbb{R}^2} \bar{F}(x_1,x_2)\rho_{\sigma_1,\underline{m}_1}(x_1)\rho_{\sigma_2,\underline{m}_2}(x_2)  f_1(x_1;\mathbf{b})f_2(x_2;\mathbf{b}) {d}x_1 {d}x_2\\
\leq &\int_{\mathbb{R}^2} \bar{F}(x_1,x_2)\rho_{\sigma_1,\underline{m}_1}(x_1)\rho_{\sigma_2,\underline{m}_2}(x_2)   f_2(x_2;\mathbf{b}) {d}x_1 {d}x_2\\
\leq &\int_{\mathbb{R}^2} \bar{F}(x_1,x_2)\rho_{\sigma_1,\underline{m}_1}(x_1)\rho_{\sigma_2,\underline{m}_2}(x_2) {d}x_1 {d}x_2\\
\leq & \int_{\underline{m}_2-\sigma_2T^{3/2}}^\infty\int_{\underline{m}_1-\sigma_1T^{3/2} }^\infty \bar{F}(x_1,x_2)\rho_{\sigma_1,\underline{m}_1}(x_1)\rho_{\sigma_2,\underline{m}_2}(x_2) {d}x_1 {d}x_2+Ce^{-C^{-1}T^3}.
\end{align*}
We have used $\bar{F}(x_1,x_2)\leq 1$ and $\displaystyle\int_{-\infty}^{\underline{m}_i-\sigma_iT^{3/2}} \rho_{\sigma_i,\underline{m}_i}(x)dx\leq Ce^{-C^{-1}T^3}$ in the last inequality.\\

Now it suffices to estimate $\bigg( \mathbb{E}_{\free}^{k,k,(q_1,q_2),x_1,x_2}[ W^{k,[p_0,p_0+d]}_{\textup{jump}} ]\bigg)^{-1}$. Let $B$ be a Brownian bridge with $B(q_1)=x_1$ and $B(q_2)=x_2$. We start by considering an event which ensures $B(x)\geq \ug(x)$ for $x\in [p_0,p_0+d]$. 
\begin{lemma}
For $E$ large enough, the following statement holds. Let $B:[q_1,q_2]\to\mathbb{R}$ be a continuous curve with $B(q_2)=x_2$. Suppose that
\begin{itemize}
\item $x_2\geq \underline{m}_2-\sigma_2T^{3/2}$,\\
\item $B({p}_0)\geq \ug( {p}_0)+3\Delta$,\\
\item $\displaystyle\inf_{x\in [p_0,q_2]} B^{[p_0,q_2]}(x)  >-\Delta$.
\end{itemize}     
Then $B(x)\geq \ug(x)$ in $[ {p}_0 ,{p}_0 +d]$.
\end{lemma}
\begin{proof}
It is sufficient to prove the assertion for the case $x_2=\underline{m}_2-\sigma_2T^{3/2}$ and $B(p_0)=g(p_0)+3\Delta$.\\

From Condition \textbf{C3} and \eqref{def:g2}, for all $x\in [p_0,p_0+d]$ we have 
\begin{align*}
	\ug(x)\leq   \ug(p_0)+ \Delta. 
\end{align*} 

Under the assumptions of this lemma, we bound the slope of the secant of $B$ in $[p_0,q_2]$ as 
\begin{align*}
&\left| \frac{x_2-B(p_0)}{q_2-p_0} \right|=\left| \frac{(\underline{m}_2-\sigma_2T^{3/2})-(g(p_0)+3\Delta)}{q_2-p_0} \right|\\ 
=&\left| \frac{\textup{Tent}(q_2)-\textup{Tent}(p_0)-\sigma_2T^{3/2} - (k+5)\Delta-k\log k }{q_2-p_0} \right| .
\end{align*}
Using \eqref{Tentj_slope}, $\sigma_2\leq s^{1/2}$ and $q_2-p_0\geq 2^{-1}s$, the above is bounded by
\begin{align*}
 2T+2T^{3/2}+2(k+5) \Delta+2(k\log k)T^2. 
\end{align*} 
For large enough $T$, it holds that 
$$\left| \frac{x_2-B(p_0)}{q_2-p_0} \right|\leq 4T^{3/2}.$$ 

As a result, for $x\in [p_0,p_0+d]$, we have
\begin{align*}
B(x)=&B(p_0)+B^{[ {p}_0,q_2]}(x)+\frac{x_2-B(p_0)}{q_2-p_0} (x-p_0)\\
\geq & (\ug(p_0)+3\Delta)+(-\Delta)-4T^{3/2}d\\
\geq & \ug(x)+\left( \Delta -4T^{3/2}d \right)\\
=&\ug(x)+\left( 8\log T -256 T^{-5/2}(\log T)^2 \right).
\end{align*}
Here we have used $\Delta=8\log T$ and $d=64T^{-4}(\log T)^2$. For $T$ large enough, we conclude that $B(x)\geq \ug(x)$ for all $x\in [p_0,p_0+d]$.
\end{proof}

When $B(x)\geq g(x)$ in $[p_0,p_0+d]$, we have $$W^{k,[p_0,p_0+d]}_{\textup{jump}}(B)\geq e^{-d\textbf{H}_t^{k,k+1}(0)}\geq 2^{-1}.$$ Therefore, provided $x_2\geq \underline{m}_2-\sigma_2T^{3/2}$,
\begin{align*}
&\mathbb{E}_{\free}^{k,k,(q_1,q_2),x_1,x_2}[W^{k,[p_0,p_0+d]}_{\textup{jump}} ]\\
\geq &2^{-1}\mathbb{P}_{\free}^{k,k,(q_1,q_2),x_1,x_2}(B(p_0)\geq \ug(p_0)+3\Delta, \inf_{x\in [p_0,q_2]} B^{[p_0,q_2]}(x)  >-\Delta\ )\\
=&2^{-1}\mathbb{P}_{\free}^{k,k,(q_1,q_2),x_1,x_2}(B(p_0)\geq \ug(p_0)+3\Delta)\, \mathbb{P}_{\free} ( \inf_{x\in [p_0,q_2]} B^{[p_0,q_2]}(x)  >-\Delta\ )\\
\geq &4^{-1}\mathbb{P}_{\free}^{k,k,(q_1,q_2),x_1,x_2}(B(p_0)\geq \ug(p_0)+3\Delta).
\end{align*}
Under the law $\mathbb{P}_{\free}^{k,k,(q_1,q_2),x_1,x_2}$, $B(p_0)$ is a normal distribution with variance $\sigma_3^2$ and mean 
$$m_3(x_1,x_2)\coloneqq \frac{x_1(q_2-p_0)+x_2(p_0-q_1) }{q_2-q_1} .$$ Therefore,

\begin{align*}
\mathbb{E}_{\free}^{k,k,(q_1,q_2),x_1,x_2}[W^{[p_0,p_0+d]}_{\textup{jump}} ] \geq 4^{-1}\int^{\infty}_{  \ug( {p}_0)+3\Delta-m_3(x_1,x_2)} \rho_{\sigma_3,0}(z_3)\,  {d}z_3\\
=:4^{-1}\nu_{\sigma_3,0}\left[  \ug( {p}_0)+3\Delta-m_3(x_1,x_2)\right].  
\end{align*}
This implies
\begin{align*}
\bar{F}(x_1,x_2)\leq 4\varepsilon \nu_{\sigma_3,0}\left[  \ug( {p}_0)+3\Delta-m_3(x_1,x_2)\right]^{-1}.
\end{align*}
Thus
\begin{align*}
&\int_{\underline{m}_2-\sigma_2T^{3/2}}^\infty\int_{\underline{m}_1-\sigma_1T^{3/2} }^\infty \bar{F}(x_1,x_2)\rho_{\sigma_1,\underline{m}_1}(x_1)\rho_{\sigma_2,\underline{m}_2}(x_2) {d}x_1 {d}x_2\\
\leq & 4\varepsilon \int_{\underline{m}_2-\sigma_2T^{3/2}}^\infty\int_{\underline{m}_1-\sigma_1T^{3/2} }^\infty \rho_{\sigma_1,\underline{m}_1}(x_1)\rho_{\sigma_2,\underline{m}_2}(x_2)\nu_{\sigma_3,0}\left[  \ug( {p}_0)+3\Delta-m_3(x_1,x_2)\right]^{-1} {d}x_1 {d}x_2.
\end{align*} 

To simplify the expression, we perform a change of variables $z_i=x_i-\underline{m}_i$.  
\begin{align*}
m_3(x_1,x_2)=&\frac{z_1(q_2-p_0)+z_2 (p_0-q_1) }{q_2-q_1} +\frac{\text{Tent}(q_1)(q_2-p_0)+\text{Tent}(q_2) (p_0-q_1) }{q_2-q_1}    - \Delta_k.
\end{align*}
From \eqref{Tentj_slope}, the slope of $\text{Tent}$ is bounded by $\pm 2T$. Hence 
\begin{align*}
&\frac{\text{Tent}(q_1)(q_2-p_0)+\text{Tent}(q_2) (p_0-q_1) }{q_2-q_1} \\
\geq &\frac{(\text{Tent}(p_0)-(p_0-q_1)2T)(q_2-p_0)+(\text{Tent}(p_0)-(q_2 -p_0)2T) (p_0-q_1) }{q_2-q_1}\\
=& \ug(p_0)-4\sigma_3^2 T.
\end{align*} 
It follows that
\begin{align*}
m_3(x_1,x_2)\geq \frac{z_1(q_2-p_0)+z_2 (p_0-q_1) }{q_2-q_1} +\ug(p_0)-4\sigma_3^2 T- \Delta_k.
\end{align*}

Take $M= \sigma_3^{-1}(4\sigma_3^2 T+\Delta_k+3\Delta) $. In particular, $4\sigma_3^2 T+\Delta_k+3\Delta= \sigma_3 M .$ Then the upper bound becomes
\begin{align}\label{equ:long}
4\varepsilon \int_{ -\sigma_2T^{3/2}}^\infty\int_{ -\sigma_1T^{3/2} }^\infty  \rho_{\sigma_1,0}(z_1) \rho_{\sigma_2,0}(z_2)\cdot\bigg[ \nu_{\sigma_3,0}\left(-\frac{z_1(q_2-p_0)+z_2 (p_0-q_1) }{q_2-q_1} +\sigma_3 M \right)  \bigg]^{-1} dz_1dz_2.
\end{align} 
Let $N_1$ and $N_2$ be two independent Gaussian random variables with mean zero and variances $\sigma_1^2$ and $\sigma_2^2$ respectively. Then $$N_4\coloneqq \frac{ (q_2-p_0)N_1+ (p_0-q_1)N_2 }{q_2-q_1} $$ 
is a Gaussian random variable with mean zero and variance 
$$\sigma^2_4\coloneqq \left(\frac{  q_2-p_0   }{q_2-q_1}\right)^2  \sigma_1^2+\left(\frac{  p_0-q_1   }{q_2-q_1}\right)^2 \sigma_2^2.$$  Moreover, as $N_1\geq -\sigma_1T^{3/2}$ and $N_2\geq -\sigma_2T^{3/2}$, 
$$N_4\geq -\left(\left(\frac{  q_2-p_0   }{q_2-q_1}\right)\sigma_1+\left(\frac{  p_0-q_1   }{q_2-q_1}\right)\sigma_2\right)T^{3/2}.$$ 
Before continuing the estimate, we give a simple bound to the above. 
\begin{lemma}\label{lem:sigma} \hfill
\begin{enumerate}
\item $4^{-1}\sigma_3^2\leq \sigma_4^2\leq \sigma_3^2$,
\item $\left(\frac{  q_2-p_0   }{q_2-q_1}\right)\sigma_1+\left(\frac{  p_0-q_1   }{q_2-q_1}\right)\sigma_2 \leq 2\sigma_3  $
\end{enumerate}
\end{lemma}
\begin{proof}
\textbf{(1)} From a direct computation,
\begin{align*}
\sigma_4^2=&\frac{(q_2-p_0)^2}{(q_2-q_1)^2}\times \frac{(p_0-q_1)(q_1-p_-)}{p_0-p_-}+\frac{(p_0-q_2)^2}{(q_2-q_1)^2}\times \frac{(p_+-q_2)(q_2-p_0)}{p_+-p_0}\\
=&\sigma_3^2\left( \frac{q_1-p_-}{p_0-p_-}\times\frac{q_2-p_0}{q_2-q_1}+\frac{p_+-q_2}{p_+-p_0}\times\frac{p_0-q_1}{q_2-q_1} \right) .
\end{align*}
Because $p_-<q_1<p_0<q_2<p_+ $, it holds that
\begin{align*}
\frac{q_1-p_-}{p_0-p_-}\leq 1,\  \frac{p_+-q_2}{p_+-p_0}\leq 1.
\end{align*}
Therefore,
\begin{align*}
\sigma_4^2\leq \sigma_3^2\left(  \frac{q_2-p_0}{q_2-q_1}+ \frac{p_0-q_1}{q_2-q_1} \right)=\sigma_3^2.
\end{align*}

Let $\alpha=\frac{q_2-p_0}{q_2-q_1}$ and $\beta=\frac{p_0-q_1}{q_2-q_1}$. Because $q_2-q_1\leq \frac{3}{2}s$ and $p_0-p_-\geq s$, it holds that
\begin{align*}
\frac{q_1-p_-}{p_0-p_-}=1-\frac{p_0-q_1 }{p_0-p_-}\geq 1-\frac{3}{2}\frac{p_0-q_1 }{q_2-q_1}=1-\frac{3}{2}\beta.
\end{align*}  
Similarly,
\begin{align*}
\frac{p_+-q_2}{p_+-p_0}\geq 1-\frac{3}{2}\alpha.
\end{align*}
Together with $\alpha+\beta=1$, it holds that
\begin{align*}
  &\frac{q_1-p_-}{p_0-p_-}\times\frac{q_2-p_0}{q_2-q_1}+\frac{p_+-q_2}{p_+-p_0}\times\frac{p_0-q_1}{q_2-q_1} \\
  \geq  &\left(1-\frac{3}{2}\beta\right)\alpha+\left(1-\frac{3}{2}\alpha\right)\beta \\
  =&3(\alpha-2^{-1})^2+4^{-1}\geq 4^{-1}.
\end{align*}
We conclude that $\sigma_4^2\geq 4^{-1}\sigma_3^2$.\\
 
\noindent \textbf{(2)} By a direct computation,
\begin{align*}
&\left(\frac{  q_2-p_0   }{q_2-q_1}\right)\sigma_1+\left(\frac{  p_0-q_1   }{q_2-q_1}\right)\sigma_2\\
=&\left(\frac{(q_2-p_0)(p_0-q_1)}{q_2-q_1} \right)^{1/2}\left( \left( \frac{(q_2-p_0)(q_1-p_-)}{(q_2-q_1)(p_0-p_-)} \right)^{1/2}+\left(\frac{(p_0-q_1)(p_+-q_2)}{(q_2-q_1)(p_+-p_0)} \right)^{1/2} \right)
\end{align*}
Using $\sigma_3^2=\frac{(q_2-p_0)(p_0-q_1)}{q_2-q_1}$, $\alpha=\frac{q_2-p_0}{q_2-q_1}$ and $\beta=\frac{p_0-q_1}{q_2-q_1}$, it becomes
\begin{align*}
\sigma_3\left(\alpha^{1/2}\left( \frac{  q_1-p_- }{  p_0-p_- } \right)^{1/2}+\beta^{1/2}\left(\frac{ p_+-q_2 }{ p_+-p_0 } \right)^{1/2} \right)
\end{align*}
Because $p_-<q_1<p_0<q_2<p_+ $, it holds that
\begin{align*}
\frac{q_1-p_-}{p_0-p_-}\leq 1,\  \frac{p_+-q_2}{p_+-p_0}\leq 1.
\end{align*}
Therefore,
 \begin{align*}
 \left(\frac{  q_2-p_0   }{q_2-q_1}\right)\sigma_1+\left(\frac{  p_0-q_1   }{q_2-q_1}\right)\sigma_2\leq \sigma_3(\alpha^{1/2}+\beta^{1/2})\leq 2\sigma_3.
 \end{align*}
\end{proof}
From Lemma~\ref{lem:sigma}, \eqref{equ:long} is bounded by
\begin{align*}
& 4\varepsilon\int_{-2\sigma_3T^{3/2} }^\infty\rho_{\sigma_4,0}(z_4)\cdot\bigg[ \nu_{\sigma_3,0}(-z_4+\sigma_3 M  )  \bigg]^{-1}\, dz_4\\
=& 4\varepsilon\int_{-\infty }^{ 2\sigma_3T^{3/2}}\rho_{\sigma_4,0}(z_4)\cdot\bigg[ \nu_{\sigma_3,0}( z_4+\sigma_3 M   )  \bigg]^{-1}\, dz_4.
\end{align*}

Because $\bigg[ \nu_{\sigma_3,0}( z_4+\sigma_3M   )  \bigg]^{-1}$ is non-decreasing in $z_4$,
\begin{align*}
&\int_{-\infty }^{0}\rho_{\sigma_4,0}(z_4)\cdot\bigg[ \nu_{\sigma_3,0}( z_4+\sigma_3M      )  \bigg]^{-1}\, dz_4\leq 2^{-1} \bigg[ \nu_{\sigma_3,0}( \sigma_3M  )  \bigg]^{-1}\\
\leq &2^{-1}Z^{-1}\int_{0 }^{ 2\sigma_3T^{3/2}}\rho_{\sigma_4,0}(z_4)\cdot\bigg[ \nu_{\sigma_3,0}( z_4+\sigma_3M   )  \bigg]^{-1}\, dz_4,
\end{align*}
Here
\begin{align*}
Z=\int_{0 }^{ 2\sigma_3T^{3/2}}\rho_{\sigma_4,0}(z_4) \, dz_4.
\end{align*}
By Lemma~\ref{lem:sigma}, for large enough $T$, it holds that
\begin{align*}
 Z\geq \int_{0 }^{ 2\sigma_4T^{3/2}}\rho_{\sigma_4,0}(z_4) \, dz_4\geq 4^{-1}.
\end{align*}
As a result, \eqref{equ:long} is bounded by
\begin{align*}
  &12\varepsilon  \int_{0 }^{ 2\sigma_3T^{3/2}}\rho_{\sigma_4,0}(z_4)\cdot\bigg[ \nu_{\sigma_3,0}( z_4+\sigma_3M  )  \bigg]^{-1}\, dz_4. 
\end{align*}

Next, we replace $\nu_{\sigma_3,0}( z_4+\sigma_3M  )$ with $\rho_{\sigma_3,0}(z_4+\sigma_3M)$. We first show that $M\geq 1$. Recall that $M= \sigma_3^{-1}(4\sigma_3^2 T+\Delta_k+3\Delta)$. Because $T\geq 10$,
\begin{align*}
M\geq 4\sigma_3 T+\sigma_3^{-1}\Delta\geq 4(T\Delta)^{1/2}\geq 1.
\end{align*} 
For $z\geq 1$, [Wil91, Section 14.8] we have
\begin{align*}
\nu_{1,0}(z)\geq  \frac{z}{z^2+1}(2\pi)^{-1/2}e^{-z^2/2}\geq (2z)^{-1}(2\pi)^{-1/2}e^{-z^2/2}=(2z)^{-1}\rho_{1,0}(z).
\end{align*}
By the Brownian scaling, for $z\geq \sigma$ it holds that 
$$\nu_{\sigma,0}(z)=\nu_{1,0}(\sigma^{-1}z)\geq \sigma (2z)^{-1}\rho_{1,0}(\sigma^{-1}z)=\sigma^2 (2z)^{-1}\rho_{\sigma,0}( z).$$

Because $M\geq 1$ and $z_4\geq 0$, we deduce
\begin{align*}
\nu_{\sigma_3,0}( z_4+\sigma_3 M   )\geq 2^{-1}\sigma_3^2 (z_4+\sigma_3 M)^{-1}\rho_{\sigma_3,0}(z_4+\sigma_3 M).
\end{align*}

Hence \eqref{equ:long}  is bounded by
\begin{align*}
 &24\varepsilon  \sigma_3^{-2}\int_{0 }^{ 2\sigma_3T^{3/2}}\rho_{\sigma_4,0}(z_4)\rho_{\sigma_3,0}(z_4+\sigma_3 M )^{-1}   ( z_4+ \sigma_3 M )  \, dz_4\\
= &24\varepsilon\int_{0 }^{ 2\sigma_3T^{3/2}}\rho_{\sigma_4,0}(z_4)\rho_{\sigma_3,0}(z_4)^{-1}e^{ M\sigma_3^{-1}z_4+M^2/2 }  (\sigma_3^{-1} z_4+  M )\sigma_3^{-1}  \, dz_4 .
\end{align*}
Because $ 4^{-1}\sigma_3^2\leq \sigma_4^{2}\leq \sigma_3^{2}$, $$\rho_{\sigma_4,0}(z_4)\rho_{\sigma_3,0}(z_4)^{-1}=\sigma_4^{-1}  \sigma_3 \exp( {-2^{-1}(\sigma_4^{-2}-\sigma_3^{-2})z_4^2}) \leq  2.$$
\eqref{equ:long} is bounded by
\begin{align*}
& 48\varepsilon\int_{0 }^{ 2\sigma_3T^{3/2}}e^{ M\sigma_3^{-1}z_4+M^2/2 }  (\sigma_3^{-1} z_4+  M )\sigma_3^{-1} \, dz_4\\
 =& 48\varepsilon\times e^{M^2/2} \int_{0 }^{ 2 T^{3/2}} e^{Mw_4} (w_4+  M )  \, dw_4.\\
\leq & 48\varepsilon\times 2T^{3/2}(2T^{3/2}+M) e^{2MT^{3/2}+ M^2/2}.
\end{align*}
We proceed to bound $M$. Because $\frac{3 }{2}s\geq q_2-q_1 $ and $q_2-p_0,p_0-q_1\geq \frac{1}{2}s$, it holds that $\frac{3}{8}s\geq \sigma_3^2\geq \frac{1}{6}s$. Therefore $ M=\sigma_3^{-1}(4\sigma_3^2 T+\Delta_k+3\Delta)\leq DT $.  We arrive at

\begin{align*}
\eqref{equ:long} \leq\varepsilon\times De^{DT^{5/2}}.
\end{align*}
In conclusion,
\begin{align*}
\sup_{\mathbf{b}\geq \underline{\mathbf{m}}}\int_{\mathbb{R}^2} F(x_1,x_2)\rho_{\sigma_1,\underline{m}_1}(x_1)\rho_{\sigma_2,\underline{m}_2}(x_2)  f_1(x_1;\mathbf{b})f_2(x_2;\mathbf{b})  {d}x_1  {d}x_2\leq \varepsilon\times   De^{DT^{5/2} }+Ce^{-C^{-1}T^3}.
\end{align*}

We are ready to complete the proof of Proposition \ref{clm:numerator}.
\begin{proof}[{\bf Proof of Proposition \ref{clm:numerator}}]
From the above discussion,
\begin{align*}
 \mathbb{P}_J(J_k^{[q_1,q_2]}\in A')\leq &\sup_{\mathbf{b}\geq \underline{\mathbf{m}}}\int_{\mathbb{R}^2} F(x_1,x_2)\rho_{\sigma_1,\underline{m}_1}(x_1)\rho_{\sigma_2,\underline{m}_2}(x_2)  f_1(x_1;\mathbf{b})f_2(x_2;\mathbf{b})  {d}x_1  {d}x_2+\mathbb{P}_J(\mathsf{Y}^{\textup{c}})\\
 \leq & \varepsilon\times   De^{DT^{5/2}  }+Ce^{-C^{-1}T^3}+De^{-D^{-1}T^4}.
\end{align*}
Recall that $T=E\max\{\left(\log\varepsilon^{-1} \right)^{1/3},s\}.$ Therefore by taking $E$ large enough,
\begin{align*}
\mathbb{P}_J(J_k^{[0,s]}\in A)= \mathbb{P}_J(J_k^{[q_1,q_2]}\in A')\leq \varepsilon\times D\exp\left( D\left(\log\varepsilon^{-1}\right)^{5/6} \right).
\end{align*}
\end{proof}

\begin{appendix}
\section{Concave functions}
In this section we record basic properties of concave functions. Throughout this section, $h:[\ell,r]\to\mathbb{R}$ is a continuous concave function. \\
For $x\in (\ell,r)$, the left/right derivative of $h$ at $x$ is defined by
\begin{align*}
h'_{\pm}(x)\coloneqq \lim_{y\to x^{\pm}}\frac{h(y)-h(x)}{y-x}.
\end{align*}
We may extend $h'_+(x)$to $x=\ell$ by allowing it to take value $\infty$. We may also extend $h'_-(r)$ to $x=r$ by allowing it to take value $-\infty$. $h'_{\pm}(x)$ are monotone non-increasing where they are defined. Moreover, for all $x \in (\ell,r)$, $h'_-(x )\geq h'_+(x )$.
\begin{lemma}\label{lem:onesidecontinuous}
$h'_+(x):[\ell,r)\to \mathbb{R}\cup\{\infty\}$ is right continuous and $h'_-(x):(\ell,r]\to \mathbb{R}\cup\{-\infty\}$ is left continuous.
\end{lemma}
\begin{proof}
Fix $x_0\in [\ell,r)$ and take $x_j \searrow x_0$. Given $y>x_0$, we have for $x_j<y$,
\begin{align*}
\frac{h(y)-h(x_j)}{y-x_j}\leq h'_+(x_j).
\end{align*}
By the continuity of $h(x)$ at $x_0$, 
\begin{align*}
\frac{h(y)-h(x_0)}{y-x_0}\leq \liminf_{j\to\infty} h'_+(x_j).
\end{align*}
Hence $h'_+(x_0)\leq \displaystyle  \liminf_{j\to\infty} h'_+(x_j)$. By the monotonicity of $h'_+(x)$, $h'_+(x_0)\geq \displaystyle  \limsup_{j\to\infty} h'_+(x_j)$. This shows $h'_+(x)$ is right continuous. Left continuity of $h'_+(x)$ can be derived similarly. 
\end{proof}
Define the extreme points of $h(x)$ as
\begin{align*}
\text{xExt}(h)\coloneqq \{\ell,r \}\cup \left\{x\in (\ell,r)\ \big|\ h\big|_{[x-\delta,x+\delta]}\ \textup{is not linear for any}\ \delta>0  \right\}. 
\end{align*}
For any $x_0\in [\ell,r]\setminus \textup{xExt}(h)$, by considering the largest interval containing $x_0$ in which $h(x)$ is linear, we get the following lemma.
\begin{lemma}
For all $x_0\in [\ell,r]\setminus \textup{xExt}(h)$, there exist $x_1,x_2\in \textup{xExt}(h)$ such that $x_0\in (x_1,x_2)$ and $h\big|_{[x_1,x_2]}$ is linear. 
\end{lemma}
As a direct corollary, we have
\begin{corollary}
Let $\bar{h}:[\ell,r]\to \mathbb{R}$ be another concave function. Suppose $\bar{h}(x)\geq h(x)$ for all $x\in \textup{xExt}(h)$. Then $\bar{h}(x)\geq h(x)$ for all $x\in [\ell,r]$.
\end{corollary}
 
Let $P\subset \textup{xExt}(h)$ be a subset that satisfies the following properties
\begin{itemize}
\item $\ell ,r \in P$
\item For any $y\in \textup{xExt}(h)\cap [\ell ,r ]$, there exists $x\in P$ such that $|x-y|<s$.
\end{itemize}
\begin{lemma}\label{lem:Tentlb}
Suppose that $|h'_{\pm}|\leq m$ for some $m\geq 0$. Further assume that $\bar{h}:[\ell ,r ]\to\mathbb{R}$ is a concave function with $\bar{h}(x)\geq h(x)$ for all $x\in P\cap [\ell ,r ]$. Then $\bar{h}(x)+2 ms \geq h(x)$ for all $x\in [\ell ,r ]$.
\end{lemma}
\begin{proof}
It suffices to show that $\bar{h}(x)+2m \geq h(x)$ for all $x\in \textup{xExt}(h)$. Given $x\in \textup{xExt}(h)$. If $x\in P$, by the assumption $\bar{h}(x)\geq h(x)$. Suppose $x\in \textup{xExt}(h) \setminus P$. There exists $y\in P$ with $|x-y|< d$. We further assume that $x<y$. Because $\bar{h}(y)\geq h(y)$, $\bar{h}(\ell)\geq h(\ell)$ and $\bar{h}$ is concave, we have
\begin{align*}
\bar{h}(x )\geq&\frac{(y-x)\bar{h}(\ell)+(x-\ell)\bar{h}(y)}{y-\ell} 
\geq  \frac{(y-x)  h (\ell)+(x-\ell) h (y)}{y-\ell}\\
  =&h(y)+\frac{h(\ell)-h(y)}{y-\ell}(y-x)\geq h(y)-ms .
\end{align*}
Together with $h(x)\leq h(y)+ms  $, we obtain that $\bar{h}(x)\geq h(x)-2ms $. The argument for $y<x$ is similar. The proof is finished.
\end{proof}
\begin{lemma}\label{lem:liri}
Let $m\in\mathbb{R}$. Assume that $\{x\in [\ell,r]\ |\ h'_+(x)\leq m\}$ is non-empty. Then 
$$x_1\coloneqq\inf \{x\in [\ell,r]\ |\ h'_+(x)\leq m\}\in \textup{xExt}(h)$$
and $h'_+(x_1)\leq m$. Similarly, assume $\{x\in [\ell,r]\ |\ h'_-(x)\geq m\}$ is non-empty. Then $$x_2\coloneqq\sup \{x\in [\ell,r]\ |\ h'_-(x)\geq m\}\in \textup{xExt}(h)$$
and $h'(x_2)\geq m$.
\end{lemma}
\begin{proof}
$h'_+(x_1)\leq m$ and $h'_-(x_2)\geq m$ follows from one-sided continuity of $h'_{\pm}.$ If $x_1\notin \textup{xExt}(h)$, then $h$ is linear around $x_1$ which contradicts the definition of $x_1$. The proof for $x_2$ is similar.
\end{proof}
Fix $m\in\mathbb{R}$ and $[\ell',r']\subset (\ell,r)$. Define $$\mathbf{l}(h)\coloneqq\left\{ \begin{array}{cc}
\inf\{x\in [\ell',r']\,|\, h'_+(x)\leq m \} & \{x\in [\ell',r']\,|\, h'_+(x)\leq m \}\neq \phi,\\
r' & \{x\in [\ell',r']\,|\, h'_+(x)\leq m \}=\phi.
\end{array}  
\right..$$
Similarly,
$$\mathbf{r}(h)\coloneqq\left\{ \begin{array}{cc}
\sup\{x\in [\ell',r']\,|\, h'_-(x)\geq m \} & \{x\in [\ell',r']\,|\, h'_-(x)\geq m \} \neq \phi,\\
\ell' & \{x\in [\ell',r']\,|\, h'_-(x)\geq m \} =\phi.
\end{array}  
\right..$$
The following lemma shows that $\textbf{l}$ and $\textbf{r}$ are semi-continuous with respect to the uniform topology.
\begin{lemma}\label{lem:Alsc}
Given a sequence of concave functions $h_i$ converging to $h_0$ uniformly on $[\ell,r]$, we have
\begin{align*}
\liminf_{i\to\infty} \mathbf{l}(h_i)\geq \mathbf{l}(h_0),\ \limsup_{i\to\infty} \mathbf{r}(h_i)\leq \mathbf{r}(h_0).
\end{align*} 
 
\end{lemma}
\begin{proof}
We give the proof for $\mathbf{l}$. The argument for $\mathbf{r}$ is similar. We assume $\mathbf{l}(h_0)>\ell'$ otherwise the assertion is clearly true. For any $\ell'\leq x_0<\mathbf{l}(h_0))$, we have $(h_0)'_+(x_0)>m$. Hence there exists $\varepsilon>0$ such that
\begin{align*}
\varepsilon^{-1}( h_0(x_0+\varepsilon)-h_0(x_0))>m.
\end{align*}
Because $h_i$ converges to $h_0$ uniformly, for $i$ large enough, 
\begin{align*}
\varepsilon^{-1}( h_i(x_0+\varepsilon)-h_i(x_0))>m.
\end{align*}
By the concavity of $h_i$, $(h_i)'_+(x_0)>m.$ For such $i$, $\mathbf{l}(h_i)\geq x_0$. Hence $$\displaystyle\liminf_{i\to\infty} \mathbf{l}(h_i)\geq x_0.$$
Because the above holds for all $\ell'\leq x_0<\textbf{l}(h_0), $ we have
$$\displaystyle\liminf_{i\to\infty} \mathbf{l}(h_i)\geq \textbf{l}(h_0).$$
The proof is finished.
\end{proof}

Let $g:[\ell ,r ]\to \mathbb{R}$ be an upper semi-continuous function. Let $h(x):[\ell,r]\to \mathbb{R}$ be the concave majorant of $g$ in $[\ell,r]$. In other words, for any $x_0\in [\ell,r]$
\begin{align*}
h(x_0)\coloneqq\inf\left\{ ax_0+b\ |\ ax+b\geq g(x)\ \textup{for all}\ x\in [\ell,r] \right\}.
\end{align*}
We note that $h(x)$ is continuous in $[\ell,r]$.
\begin{lemma}\label{lem:contact}
For all $x\in\textup{xExt}(h)\setminus\{\ell,r\}$, we have $h(x)=g(x)$.
\end{lemma}
\begin{proof}
Suppose that $x_0\in\textup{xExt}(h)\setminus\{\ell,r\}$ and that $h(x_0)-g(x_0)>0$. Let $L_{\pm}(x)=h'_\pm (x_0)(x-x_0)+h(x_0)$.  We have $L_{\pm}(x)\geq h(x)\geq g(x)$ for all $x\in [\ell,r]$. Suppose that $L_+(x)>g(x)$ for all $x\in (x_0,r]$. By the upper semi-continuity of $g(x)$, $$\varepsilon=\displaystyle\inf_{[x_0,r]}(L_+(x)-g(x))>0. $$ Then for all $x\in [\ell,r]$,
$$ L_+(x)-\frac{\varepsilon}{r-x_0}(x-x_0)\geq g(x).$$
Hence $L_+(x)-\frac{\varepsilon}{r-x_0}(x-x_0)\geq h(x)$. This implies $h'_+(x_0)-\frac{\varepsilon}{r-x_0} \geq h'_+(x_0)$ which is a contradiction. Therefore there exists $x_+\in (x_0,r]$ such that $L_+(x_+)=g(x_+)$. Similarly, there exists $x_-\in [\ell,x_0)$ such that $L_-(x_-)=g(x_-)$. Because $x_0\in\textup{xExt}(h)$, we must have $h'_+(x_0)<h'_-(x_0)$. Take $L_0(x)=\frac{1}{2}(h'_+(x_0)+h'_-(x_0) )(x-x_0)+h(x_0)$. At $x=x_0$, $L_0(x_0)=h(x_0)>g(x_0)$. For $x\in (x_0,r]$, $L_0(x)> L_+(x)\geq g(x)$. For $x\in [\ell,x_0)$, $L_0(x)>L_-(x)\geq g(x)$. In particular,
\begin{align*}
\varepsilon_0= \inf_{x\in [\ell,r]} L_0(x)-g(x)>0.
\end{align*}
Then $h(x)\leq L_0(x)-\varepsilon_0$, which is a contradiction at $x=x_0$.
\end{proof}

Next, we consider the situation that $h(x)$ or $g(x)$ is close to a parabola $-2^{-1}x^2$.

\begin{lemma}\label{lem:slope}
Let $\delta>0$ be a positive number. Let $h:[\ell,r]\to \mathbb{R}$ be a continuous concave function with $|h(x)+2^{-1}x^2|\leq 4^{-1}\delta^2$. Then for all $x\in [\ell+\delta ,r-\delta ]$, we have
\begin{equation}
-x-\delta \leq h'_+(x)\leq h'_-(x)\leq -x+\delta.
\end{equation} 
\end{lemma}
\begin{proof}
Let $x\in [\ell+\delta,r-\delta].$ By the concavity of $h$, we have 
\begin{align*}
h'_+(x)\geq \frac{h(x)-h(x-\delta )}{\delta}\geq \delta^{-1}\left( (2^{-1}x^2-4^{-1}\delta^2)-(2^{-1}(x-\delta)^2+4^{-1}\delta^2) \right)=-x-\delta.
\end{align*}
Similarly,
\begin{align*}
h'_-(x)\leq \frac{h(x+\delta)-h(x  )}{\delta}\leq \delta^{-1}\left( (2^{-1}(x+\delta)^2+4^{-1}\delta^2)-(2^{-1}x^2-4^{-1}\delta^2) \right)=-x+\delta.
\end{align*}
\end{proof}
\begin{lemma}\label{lem:majorant}
Let $g:[\ell,r]\to \mathbb{R}$ be a function with $|g(x)+2^{-1}x^2|\leq 4^{-1}\delta^2$. Take $x_0\in [\ell+3\delta,r-3\delta]$. Let $h_1$ and $h_2$ be the concave majorant of $g(x)$ in $[\ell,r]$ and $[x_0-3\delta,x_0+3\delta]$ respectively. Then $h_1(x_0)=h_2(x_0)$. 
\end{lemma}
\begin{proof}
Directly from the definition, we have $h_1(x_0)\geq h_2(x_0)$. Consider the line 
\begin{align*}
L(x)\coloneqq h_2(x_0)+(h'_{2})_+(x_0)(x-x_0).
\end{align*}
It suffices to show that $L(x)\geq g(x)$ for all $x\in [\ell,r]$. Because $|h_2(x)+2^{-1}x^2|\leq 4^{-1}\delta^2$ for all $x\in [x_0-\delta,x_0+\delta]$, from Lemma \ref{lem:slope} we have
\begin{align*}
|(h'_{2 })_+(x_0)+x_0|\leq \delta.
\end{align*} 
As a result, for all $x\geq x_0$, we have
\begin{align*}
L(x)-g(x)&\geq \left[ -2^{-1}x_0^2-4^{-1}\delta^2+(-x_0+\delta)(x-x_0) \right]-[-2^{-1}x^2+4^{-1}\delta^2]\\
&=2^{-1}(x-x_0)^2-\delta (x-x_0)-2^{-1}\delta^2.
\end{align*}
In particular, for $x\geq x_0+3\delta$, we have $L(x)\geq g(x)$. Similar argument can show that $\ell(x)\geq g(x)$ for all $x\leq x_0-3\delta$. Together with $L(x)\geq h_2(x)\geq f(x)$ for all $x\in [x_0-3\delta,x_0+3\delta]$, we conclude that $L(x)\geq f(x)$ for all $x\in [\ell,r]$. Thus
\begin{align*}
h_2(x_0)=L(x_0)\geq h_1(x_0). 
\end{align*}
The proof is finished.
\end{proof}
\begin{corollary}\label{cor:majorant}
Let $\ell<\ell' <r'<r$ be real numbers and $g:[\ell,r]\to \mathbb{R}$ be a function with $|g(x)+2^{-1}x^2|\leq 4^{-1}\delta^2$. Let $h_1$ and $h_2$ be the concave majorant of $g(x)$ in $[\ell,r]$ and $[\ell',r']$ respectively. Then $h_1(x)=h_2(x)$ for all $x\in [\ell'+3\delta,r'-3\delta]$. 
\end{corollary} 

\section{Miscellaneous}
\begin{lemma}\label{lem:B1}
For any $k\in\mathbb{N}$ and $(y_1,y_2,\dots y_k, y_{k+1})\in\mathbb{R}^{k+1} $,
$$ -\sum_{j=1}^ke^{y_{j+1}-y_j}\leq -k e^{(y_{k+1}-y_1)/k} .$$
\end{lemma}
\begin{proof}
The assertion holds clearly when $k=1$. Assume $k\geq 2$ and the assertion holds for $k-1$. Then
 \begin{align*}
  -\sum_{j=1}^ke^{y_{j+1}-y_j}= -\sum_{j=1}^{k-1}e^{y_{j+1}-y_j}-e^{y_{k+1}-y_k}\leq -(k-1)e^{(y_k-y_1)/(k-1)}-e^{y_{k+1}-y_k}.
 \end{align*}
It is straightforward to show that 
 $$-(k-1)e^{(y_k-y_1)/(k-1)}-e^{y_{k+1}-y_k}\leq -k e^{(y_{k+1}-y_1)/k}$$
 and the inequality is achieved when $y_k=\frac{1}{k}y_1+\frac{k-1}{k}y_{k+1}$. The proof is finished.
\end{proof}
Recall that $\textbf{H}_t(x)=e^{t^{1/3}x}$. For integers $1\leq j< k$, define
\begin{equation}\label{def:Hk}
 \mathbf{H}_{t}^{j,k }(x)\coloneqq    (k-1)^{-1} e^{t^{1/3}x/{(k-j)}} . 
\end{equation}

\begin{lemma}\label{lem:Hk}
For any $k\in\mathbb{N}$, and $(y_1,y_2,\dots y_{k+1})\in\mathbb{R}^{k+1}$, we have
\begin{align*}
-\sum_{j=1}^k \mathbf{H}_t(y_{j+1}-y_j) \leq -\sum_{j=1}^k \mathbf{H}_{t}^{j,k+1}(y_{k+1}-y_j).
\end{align*}
\end{lemma}
\begin{proof}
When $k=1$, $\mathbf{H}_{t}^{1,2}(x)=\mathbf{H}_{t}(x)$ and the assertion clearly holds. Assume $k\geq 2$ and the assertion holds for $k-1$. 
\begin{align*}
-\sum_{j=1}^k \mathbf{H}_t(y_{j+1}-y_j) &=-\frac{1}{k}\sum_{j=1}^k \mathbf{H}_t(y_{j+1}-y_j) -\frac{k-1}{k}\sum_{j=1}^k \mathbf{H}_t(y_{j+1}-y_j)\\
&\leq -\frac{1}{k}\sum_{j=1}^k \mathbf{H}_t(y_{j+1}-y_j) -\frac{k-1}{k}\sum_{j=2}^k \mathbf{H}_t(y_{j+1}-y_j).  
\end{align*}
By Lemma \ref{lem:B1}, $$-k^{-1}\displaystyle\sum_{j=1}^k \mathbf{H}_t(y_{j+1}-y_j)\leq -k\textbf{H}_t^{1,k+1}(y_{k+1}-y_1)\leq -\textbf{H}_t^{1,k+1}(y_{k+1}-y_1).$$
For $j\in [1,k-1]_\mathbb{Z}$, let $z_j=y_{j+1}$. Then by the induction hypothesis,
\begin{align*}
-\frac{k-1}{k}\sum_{j=2}^k \mathbf{H}_t(y_{j+1}-y_j)   =&-\frac{k-1}{k}\sum_{j=1}^{k-1} \mathbf{H}_t(z_{j+1}-z_j)\leq    -\frac{k-1}{k}\sum_{j=1}^{k-1} \mathbf{H}^{j,k}_t(z_{k}-z_j)\\
=&-\frac{k-1}{k}\sum_{j=2}^{k}  \mathbf{H}^{j-1,k}_t(y_{k+1}-y_j)=-\sum_{j=2}^{k}  \mathbf{H}^{j,k+1}_t(y_{k+1}-y_j).
\end{align*}
The proof is finished.
\end{proof}

\begin{lemma}\label{lem:Tentmod}
Fix $p_-<p<p_+$. Let $\mathring{L}  :[p_-,p_+]\to \mathbb{R}$ be a continuous function which is linear in $[p_-,p]$ and $[p ,p_+]$. Fix $q\in (p_-,p_+)$ and $b\in \mathbb{R}$. Let $L$ be a continuous function which is linear in $[p_-,q] $ and $[q  ,p_+]$ with $L(p_{\pm})=\mathring{L}(p_\pm)$ and $L(q)=b$. Define
\begin{align*}
\mathring{m}_\pm \coloneqq  \frac{\mathring{L}(p )-\mathring{L}(p_\pm)}{p -p_\pm},\ {m}_\pm\coloneqq  \frac{ {L}(q )-\mathring{L}(p_\pm)}{q -p_\pm}. 
\end{align*} 
Then for all $x\in [p_-,p_+]$,
\begin{equation}\label{equ:linear}
\mathring{L}  (x)-  L (x)\leq |L(p)-b|+|p-q|\max\left\{ |\mathring{m}_-|,|\mathring{m}_+|  \right\}
\end{equation}
Furthermore, suppose $|p-q|\leq 2^{-1}\min \{p-p_-,p_+-p\}$. Then
\begin{equation*}
\left| {m}_--\mathring{m}_- \right|\leq 2\left|\frac{ b-L (p) }{p-p_-}\right|+2\left|\frac{   p-q  }{ p-p_- }\right||\mathring{m}_-|.
\end{equation*}
\begin{equation*}
\left| {m}_+-\mathring{m}_+ \right|\leq     2\left|\frac{ b-L (p) }{p_+-p }\right|+2\left|\frac{   p-q  }{ p_+-p  }  \right||\mathring{m}_+|.
\end{equation*}
\end{lemma}
\begin{proof}
Without loss of generality, we assume that $q<p$. It suffice to prove \eqref{equ:linear} at $x=p$ and $x=q$. For $x=q$, 
\begin{align*}
\mathring{L}(q)- {L}(q)=[\mathring{L}(p)-\mathring{m}_-   (p-q)]-b\leq |\mathring{L} (p)-b|+|p-q||\mathring{m}_-|. 
\end{align*}
For $x=p$,
\begin{align*}
\mathring{L}(p)- {L}(p)=&\mathring{L}(p)-\left[ \frac{(p_+-p)b+(p-q)\mathring{L}(p_+)}{p_+-q} \right]\\
=&\frac{p_+-p}{p_+-q}(\mathring{L}(p)-b)+(p-q) \frac{p_+-p}{p_+-q}\mathring{m}_+\\
\leq & |\mathring{L} (p)-b|+|p-q||\mathring{m}_+|.
\end{align*}
This shows \eqref{equ:linear}.\\

Under the assumption that $|p-q|\leq 2^{-1}\min \{p-p_-,p_+-p\}$,
\begin{align*}
\left| {m}_--\mathring{m}_- \right|= &\left| \frac{b-\mathring{L}(p_-)}{q-p_-}-\frac{\mathring{L}(p)-\mathring{L}(p_-)}{p-p_-}\right|\\
=&\left| \frac{b-\mathring{L}(p)}{q-p_-}+\frac{(p-q)(\mathring{L}(p)-\mathring{L}(p_-))}{(q-p_-)(p-p_-)} \right|\\
\leq &2\left|\frac{ b-L (p) }{p-p_-}\right|+2\left|\frac{   p-q  }{ p-p_- }\times\frac{\mathring{L}(p)-\mathring{L}(p_-)}{p-p_-} \right|\\
=&2\left|\frac{ b-L (p) }{p-p_-}\right|+2\left|\frac{   p-q  }{ p-p_- }\right||\mathring{m}_-|.
\end{align*}
The bound for $\left| {m}_+-\mathring{m}_+ \right|$ is similar.
\end{proof}
\begin{lemma}\label{lem:midpt}
Let $x_1<x_2\in\mathbb{R}$, $b_1,b_2\in\mathbb{R}$ and $m>0$. Assume that $|b_2-b_1|\leq m(x_2-x_1)$. Let
\begin{align*}
  x_0=&2^{-1}(x_1+x_2)+(2m)^{-1}(b_2-b_1)\in [x_1,x_2],\\
  b_0=&2^{-1}(b_1+b_2)+2^{-1}m(x_2-x_1). 
\end{align*}
The point $(x_0,b_0)$ is determined by
\begin{align*}
\frac{b_0-b_1}{x_0-x_1}=m,\ \frac{b_2-b_0}{x_2-x_0}=-m.
\end{align*}
Then
\begin{align*}
\mathbb{P}(B(x_0)\geq b_0\ |B(x_1)=b_1, B(x_2)=b_2)\geq  \mathbb{P}(N\geq  (x_2-x_1)^{1/2}m).
\end{align*}
\end{lemma}
\begin{proof}
Then mean and variance of $B(x_0)$ is given by
\begin{align*}
M=&2^{-1}(b_2+b_1)+(2m(x_2-x_1))^{-1}(b_2-b_1)^2,\\
\sigma^2=&(4m^2)^{-1}((x_2-x_1) m^2-(x_2-x_1)^{-1}(b_2-b_1)^2).
\end{align*}
Then
\begin{align*}
\sigma^{-1}(b_0-M)=((x_2-x_1)m^2-(x_2-x_1)^{-1}(b_2-b_1)^2)^{1/2}\leq (x_2-x_1)^{1/2}m .
\end{align*}
\end{proof}
\end{appendix}

\end{document}